%% file: main.tex
\documentclass{article}
\usepackage[cm]{fullpage}
\usepackage{graphicx}
\usepackage{amsmath, amssymb, amsfonts, amsthm}
\usepackage{mathdots}
\usepackage{subcaption}
\usepackage{xcolor}
\usepackage{enumitem}
\usepackage{array}
\usepackage{makecell}
\usepackage{float}
\usepackage{tikz}
\usepackage[numbers]{natbib}
\usepackage{algorithm}
\usepackage{algpseudocode}
\usepackage{nicefrac}
\usepackage{multirow}
\usetikzlibrary{arrows.meta,positioning}

\theoremstyle{definition}

\theoremstyle{plain}
\newtheorem{theorem}{Theorem}
\newtheorem{lemma}{Lemma}
\newtheorem{proposition}{Proposition}
\newtheorem{corollary}{Corollary}

\usepackage[hidelinks]{hyperref}
\usepackage[capitalise]{cleveref}

\makeatletter
\renewcommand*\env@matrix[1][*\c@MaxMatrixCols c]{
  \hskip -\arraycolsep
  \let\@ifnextchar\new@ifnextchar
  \array{#1}}
\makeatother

\makeatletter
\if@titlepage\else
  \renewenvironment{abstract}{%
    \small
    \paragraph{\abstractname}
  }{\par\bigskip}
\fi
\makeatother

\allowdisplaybreaks

\title{Toward a Systematic Understanding and Interactive Search of \\ Lyapunov-Style Proofs in Optimization}

\author{TaeHo Yoon\textsuperscript{*,1}, Jaewook J.\ Suh\textsuperscript{*,2}, Edward D.~H.\ Nguyen\textsuperscript{3}, Bicheng Ying\textsuperscript{4}, Shiqian Ma\textsuperscript{2}}
\date{}

\input{macro}

\begin{document}

\renewcommand\thefootnote{\fnsymbol{footnote}}
\footnotetext[1]{Equal contribution.}
\renewcommand\thefootnote{\arabic{footnote}}
\setcounter{footnote}{0}
\footnotetext[1]{Department of Applied Mathematics and Statistics, Johns Hopkins University. Email: \texttt{tyoon7@jh.edu}.}
\footnotetext[2]{Department of Computational Applied Mathematics and Operations Research, Rice University. Email: \texttt{\{jacksuh, sqma\}@rice.edu}.}
\footnotetext[3]{Department of Electrical and Computer Engineering, Rice University. Email: \texttt{en18@rice.edu}.}
\footnotetext[4]{Google Inc. Email: \texttt{ybc@google.com}.}

\maketitle

\begin{abstract}
Lyapunov-style convergence proofs, which establish a nonincreasing sequence to provide a quantitative convergence rate for an algorithm, are popular and often considered desirable in first-order optimization.
However, existing approaches to finding such Lyapunov functions rely on hand-designed templates or prior insight on the proof structure, and do not certify that the resulting Lyapunov-style analysis provides the sharpest convergence bound.
In this work, we introduce a systematic framework for converting a tight, analytic convergence proof of an optimization algorithm, often found via computer assistance, into an equivalent proof based on Lyapunov functions.
We implement a concrete procedure that combines a performance estimation problem (PEP) toolbox with elementary linear algebra, and show that it captures a number of prior Lyapunov analyses \emph{within a single Jupyter notebook}.
Based on our implementation, a user can straightforwardly test our systematic and interactive procedure on their own optimization algorithm of interest to search for a tight Lyapunov-style proof \emph{via code}, without the need to comprehend the implementation details.
We extend the application of our framework and discover four novel analytic Lyapunov-style proofs, where notably, one of them identifies a new exact optimal proximal algorithm for strongly monotone inclusion problems.
\end{abstract}

\input{sections/1_introduction}
\input{sections/2_background}

\input{sections/3_theory}
\input{sections/4_existing_lyapunov_analyses}

\input{sections/5_new_lyapunov_analyses}
\input{sections/6_conclusion}

\bibliographystyle{plain}
\bibliography{ref}

\appendix

\input{sections/appendix_background}
\input{sections/appendix_linear_algebra_prop}
\input{sections/appendix_bppm_tightness}

\end{document}

%% file: macro.tex
\newcommand{\cut}[1]{{}}

\newcommand{\vxi}{{\boldsymbol{\xi}}}

\newcommand{\vb}{{\mathbf{b}}}

\newcommand{\vd}{{\mathbf{d}}}
\newcommand{\ve}{{\mathbf{e}}}
\newcommand{\vf}{{\mathbf{f}}}
\newcommand{\vg}{{\mathbf{g}}}
\newcommand{\vh}{{\mathbf{h}}}

\newcommand{\vq}{{\mathbf{q}}}
\newcommand{\vr}{{\mathbf{r}}}
\newcommand{\vs}{{\mathbf{s}}}
\newcommand{\vt}{{\mathbf{t}}}
\newcommand{\vu}{{\mathbf{u}}}
\newcommand{\vv}{{\mathbf{v}}}
\newcommand{\vw}{{\mathbf{w}}}
\newcommand{\vx}{{\mathbf{x}}}
\newcommand{\vy}{{\mathbf{y}}}
\newcommand{\vz}{{\mathbf{z}}}

\newcommand{\vA}{{\mathbf{A}}}
\newcommand{\vB}{{\mathbf{B}}}
\newcommand{\vC}{{\mathbf{C}}}
\newcommand{\vD}{{\mathbf{D}}}

\newcommand{\vG}{{\mathbf{G}}}

\newcommand{\vI}{{\mathbf{I}}}

\newcommand{\vL}{{\mathbf{L}}}
\newcommand{\vM}{{\mathbf{M}}}

\newcommand{\vP}{{\mathbf{P}}}
\newcommand{\vQ}{{\mathbf{Q}}}

\newcommand{\vS}{{\mathbf{S}}}

\newcommand{\vV}{{\mathbf{V}}}

\newcommand{\cC}{{\mathcal{C}}}

\newcommand{\cI}{{\mathcal{I}}}
\newcommand{\cJ}{{\mathcal{J}}}

\newcommand{\cL}{{\mathcal{L}}}
\newcommand{\cM}{{\mathcal{M}}}
\newcommand{\cN}{{\mathcal{N}}}
\newcommand{\cO}{{\mathcal{O}}}

\newcommand{\cR}{{\mathcal{R}}}
\newcommand{\cS}{{\mathcal{S}}}
\newcommand{\cT}{{\mathcal{T}}}

\usepackage[bb=boondox]{mathalfa}

\usepackage{xparse}
\DeclareFontFamily{U}{ntxmia}{}
\DeclareFontShape{U}{ntxmia}{m}{it}{<-> ntxmia }{}
\DeclareFontShape{U}{ntxmia}{b}{it}{<-> ntxbmia }{}
\DeclareSymbolFont{lettersA}{U}{ntxmia}{m}{it}
\SetSymbolFont{lettersA}{bold}{U}{ntxmia}{b}{it}

\ExplSyntaxOn
\NewDocumentCommand{\varmathbb}{m}
 {
  \tl_map_inline:nn { #1 }
   {
    \use:c { varbb##1 }
   }
 }
\tl_map_inline:nn { ABCDEFGHIJKLMNOPQRSTUVWXYZ }
 {
  \exp_args:Nc \DeclareMathSymbol{varbb#1}{\mathord}{lettersA}{\int_eval:n { `#1+67 }}
 }
\exp_args:Nc \DeclareMathSymbol{varbbk}{\mathord}{lettersA}{169}
\ExplSyntaxOff

\newcommand{\Tr}{{\mathrm{Tr}}}

\newcommand{\tnabla}{\widetilde{\nabla}}

\DeclareMathOperator*{\argmin}{argmin}

\newcommand{\dom}{\mathrm{dom}\,}

\newcommand{\reals}{\mathbb{R}}

\newcommand{\opA}{{\varmathbb{A}}}

\newcommand{\opI}{{\varmathbb{I}}}
\newcommand{\opJ}{{\varmathbb{J}}}

\newcommand{\opT}{{\varmathbb{T}}}

\newcommand{\inprod}[2]{\left\langle #1,#2 \right\rangle}
\newcommand{\sqnorm}[1]{\left\| #1 \right\|^2}

\newcommand{\topa}{\tilde{\opA}}
\newcommand{\norm}[1]{\left\|#1\right\|}

\newcommand{\pr}[1]{\left( #1 \right)}

\newcommand{\set}[1]{\left\{ #1 \right\}}

\newcommand{\transpose}{\mathsf{T}}

\newcommand{\bmat}[1]{\begin{bmatrix} #1 \end{bmatrix}}

\newcommand{\fullpep}{PEP-style proof}

\newcommand{\compidx}{%
  \mathord{\mathchoice
    {\vcenter{\hbox{\rotatebox{45}{\rule{0.39em}{0.39em}}}}}
    {\vcenter{\hbox{\rotatebox{45}{\rule{0.39em}{0.39em}}}}}
    {\vcenter{\hbox{\rotatebox{45}{\rule{0.29em}{0.29em}}}}}
    {\vcenter{\hbox{\rotatebox{45}{\rule{0.23em}{0.23em}}}}}}%
}

\newcommand{\bregsm}{f}
\newcommand{\bregprox}{g}
\newcommand{\bregaux}{r}
\newcommand{\vbregsm}{\vf}
\newcommand{\vbregprox}{\vg}
\newcommand{\vbregaux}{\vr}
\newcommand{\dualoc}{Dual-OC-Halpern}

%% file: sections/1_introduction.tex
\section{Introduction}
\label{section:introduction}

Convergence proofs for optimization algorithms are clever combinations of inequalities.
Traditionally, the search for such proofs was a form of art, relying heavily on human ingenuity and intuition.
Over the last decade, however, more systematic computer-assisted approaches toward convergence analysis have emerged and quickly gained popularity \citep{DroriTeboulle2014_performance,TaylorHendrickxGlineur2017_smooth,TaylorBach2019_stochastic}.
These works provided frameworks for numerically finding a tight combination of inequalities which yields an optimized convergence guarantee within the search space.
This sparked the discovery of improved convergence guarantees for known algorithms \citep{DroriTeboulle2014_performance, TaylorBach2019_stochastic, guTightSublinearConvergence2020, GuYang2025_tight}, as well as the development of novel algorithms with accelerated/optimal convergence rates \citep{KimFessler2016_optimized, DroriTaylor2020_efficient, KimFessler2021_optimizing, Lieder2021_convergence, Kim2021_accelerated, YoonRyu2021_accelerated, LeeKim2021_fast, ParkRyu2022_exact, TaylorDrori2023_optimal, JangGuptaRyu2025_computerassisted}.

Our work is largely motivated by the history of how these computer-assisted proofs, the so-called performance estimation problem (PEP) style proofs as named in \citep{TaylorHendrickxGlineur2017_smooth}, have been processed within the literature.
Several optimized algorithms found by these techniques, including the optimized gradient method (OGM) \citep{KimFessler2016_optimized}, OGM-G \citep{KimFessler2021_optimizing}---a numerically optimized algorithm for decreasing the gradient norm, the accelerated proximal point method (APPM) \citep{Kim2021_accelerated} or its equivalent fixed-point iteration form (optimal Halpern method; OHM) \citep{Lieder2021_convergence},
were initially presented with convergence proofs establishing positive semidefiniteness of matrix certificates arising from semidefinite programming (SDP) formulation of PEP.
The same applies to some tightened analyses of simpler algorithms such as gradient descent \citep{DroriTeboulle2014_performance, altschulerAccelerationStepsizeHedging2025, altschulerAccelerationStepsizeHedging2025a, grimmerAcceleratedObjectiveGap2025} or proximal point method \citep{guTightSublinearConvergence2020,GuYang2025_tight}.

While the aforementioned proofs reflect how these numerical certificates are typically obtained through computational methods, they are not necessarily the most easily comprehensible form for humans.
Consequently, some of these proofs were later reworked into arguably more concise and human-interpretable forms \citep{dAspremontScieurTaylor2021_acceleration,ParkParkRyu2023_factorsqrt2,DiakonikolasWang2022_potential,LeeParkRyu2021_geometric,Diakonikolas2020_halpern,TeboulleVaisbourd2023_elementary},
often in the form of establishing a \emph{Lyapunov function}, or \emph{potential function}, a sequence of scalar quantities that is nonincreasing along the iterations of the algorithm.
However, translating the tight, numerically obtained PEP-style proofs into equivalent Lyapunov-style proofs is not a trivial task, and has been carried out on a case-by-case basis based on experts' intuition, rather than deductively from an overarching principle.
In this work, we fill in this gap by providing a general characterization of such procedures.

\paragraph{Main contributions.}
This work presents the following contributions.
\begin{itemize}
    \item First, we provide a systematic characterization of when a seemingly complex convergence proof can be translated into a concise Lyapunov-style proof, and in such cases, provide a procedure for constructing them (\cref{section:theory}).
    This procedure is general and is applicable to multiple popular problem settings in first-order optimization, including the ones presented in \cref{table:contribution-summary}.
    However, our framework is not confined to these problem classes and can be readily extended to other ones of interest.
    
    \item We recover a number of materially distinct Lyapunov analyses in the literature, each developed through at least one prior work, as a special case of our framework (\cref{section:existing-lyapunov}).
    This demonstrates that our framework is capable of unifying a non-trivial process that has been performed by experts on a case-by-case basis.

    \item Using our proposed framework, we discover four previously unknown Lyapunov-style proofs (\cref{section:new-lyapunov}).
    In one of them, we introduce a novel algorithm, the \emph{dual optimal contractive Halpern method (Dual-OC-Halpern)}, which is an exact minimax optimal algorithm for strongly monotone inclusion problems, or the equivalent contractive fixed-point problems.
    This demonstrates that our procedure is indeed a practical tool for conducting novel analysis with less manual effort.

    \item For all existing and new analyses that we cover, we provide the associated code implementation showing that our procedure can be streamlined into a single \texttt{Jupyter notebook}, from the numerical certificate computation to deriving and symbolically verifying the corresponding Lyapunov analysis.
\end{itemize}

\input{figures/contribution_table}

\subsection{Prior work}

\paragraph{Performance Estimation Problem (PEP).}

Performance Estimation Problem (PEP) is a computer-assisted framework for analyzing the quantitative convergence rate of optimization algorithms. The core idea was first introduced by \cite{DroriTeboulle2014_performance} and subsequently refined into an established framework in \cite{TaylorHendrickxGlineur2017_smooth} for smooth (strongly) convex minimization. 
It was then extended to a number of distinct settings including composite convex minimization \cite{TaylorHendrickxGlineur2018_exact}, convex minimization under Bregman geometry \cite{dragomirOptimalComplexityCertification2022}, operator splitting \cite{RyuTaylorBergelingGiselsson2020_operator} and linear operators \cite{BousselmiHendrickxGlineur2024_interpolation}.
PEP reformulates the problem of identifying the worst-case rate of a first-order algorithm into a tractable optimization problem—typically a semidefinite program (SDP)—that can be numerically solved by a solver, and the dual solution to this SDP encodes a proof of the form
\begin{equation}    \label{eq:proof_template} \tag{\fullpep}
    \texttt{(PerformanceMetric)} - \nu \times \texttt{(InitialCondition)}
= - \sum_{i,j} \lambda_{i,j} I_{i,j} - \sum_{i} \| v_i \|^2
\le 0 .
\end{equation}
Here, $I_{i,j} \ge 0$ are nonnegative inequality terms characterizing the function/operator class of consideration, and $\nu, \lambda_{i,j} \ge 0$ and $v_i \in \mathbb{R}^n$ are nonnegative scalars and vectors (proof certificates) corresponding to the dual variables. 
We refer to this form of proof as a \textit{\fullpep} in this paper. 
We discuss this in more detail with examples of problem classes in \Cref{section:background}.

\paragraph{Novel findings via PEP.}
PEP has been proved useful for discovering optimized analyses and algorithms over the last decade.
The tight convergence rate and analysis of constant-step-size gradient descent \cite{DroriTeboulle2014_performance}
and the optimized gradient method (OGM) \cite{DroriTeboulle2014_performance, KimFessler2016_optimized}, later shown to be an exact minimax optimal algorithm for smooth convex minimization \cite{Drori2017_exact}, 
were the earliest such discoveries enabled by PEP.
OGM improves upon Nesterov's accelerated gradient method \cite{Nesterov1983_method} approximately by a factor of $2$. 
Then, OGM-G \cite{KimFessler2021_optimizing} achieved the first 
algorithm that achieves an accelerated, up-to-constant optimal rate with respect to the gradient norm in smooth convex minimization: $\| \nabla f(x_k) \|^2 = \mathcal{O}\pr{ (f(x_0) - f(x_\star)) / k^2 }$, which further induces $\| \nabla f(x_k) \|^2 = \mathcal{O}\pr{ \| x_0 - x_\star \|^2 / k^4 }$ by concatenation with OGM. 
Optimal methods for other setups, such as strongly convex minimization \citep[ITEM]{TaylorDrori2023_optimal}, Bregman-relatively-smooth minimization \cite{dragomirOptimalComplexityCertification2022}, nonsmooth convex minimization with bounded gradients \cite{DroriTaylor2020_efficient} and composite optimization \citep[OptISTA]{JangGuptaRyu2025_computerassisted}, were also developed through the PEP framework; 
in particular, the discovery of OptISTA leveraged the broader BnB-PEP formulation \cite{DasGuptaVanParysRyu2023_branchandbound}, utilizing nonconvex quadratically constrained quadratic programs (QCQP) and branch-and-bound techniques. 
Based on the numerical observation from \cite{DasGuptaVanParysRyu2023_branchandbound}, subsequent works \cite{grimmerProvablyFasterGradient2024, altschulerAccelerationStepsizeHedging2025, altschulerAccelerationStepsizeHedging2025a, grimmerAcceleratedObjectiveGap2025} have presented the acceleration of gradient descent in smooth convex minimization by modifying the step-size schedule.

Beyond convex minimization, 
\cite{Kim2021_accelerated} proposed the accelerated proximal point method (APPM), an accelerated proximal algorithm for maximally monotone inclusion problems. 
Independently, \cite{Lieder2021_convergence} proposed the optimal Halpern method (OHM), an exact optimal algorithm for nonexpansive fixed-point problems in Hilbert spaces. 
In fact, APPM and OHM are equivalent, as it was later established in \cite{ContrerasCominetti2023_optimal, RyuYin2022_largescale}, and \cite{ParkRyu2022_exact} have shown that they are exact optimal for their respective problem classes. 
\cite{ParkRyu2022_exact} also introduced the equivalent algorithms, optimal strongly-monotone proximal point method (OS-PPM) for maximally $\mu$-strongly-monotone inclusions and the optimized contractive Halpern (OC-Halpern) for contractive fixed-point problems, both of which are exact minimax optimal with matching lower bounds.
Similar advances have been made in minimax optimization.
For smooth convex-concave minimax optimization, the extra anchored gradient (EAG) \citep{YoonRyu2021_accelerated} and anchored Popov's scheme \cite{Tran-DinhLuo2021_halperntype} have obtained accelerated $\cO(\sqnorm{x_0 - x_\star} / k^2)$ rates on the squared gradient/residual norm, which was extended to the setting with weakly comonotone saddle gradient operator by the fast extragradient (FEG) algorithm \cite{LeeKim2021_fast}.
Multiple tighter refinements to the analyses of existing methods have also been developed, including the proximal point method (PPM) \cite{guTightSublinearConvergence2020}, extragradient \cite{GorbunovLoizouGidel2022_extragradient}, optimistic gradient method \cite{gorbunovLastiterateConvergenceOptimistic2022}, Douglas--Rachford and Davis--Yin splitting \cite{RyuTaylorBergelingGiselsson2020_operator, NguyenSuhJiangMa2025_exact}, the alternating direction method of multipliers (ADMM) \cite{ZamaniAbbaszadehpeivastideKlerk2024_exact}, and primal--dual methods such as Chambolle--Pock/PDHG \cite{BousselmiHendrickxGlineur2024_interpolation}.

The PEP framework, and results driven by it, have additionally inspired a novel concept of duality operation between first-order algorithms \cite{KimOzdaglarParkRyu2023_timereversed}.
The step-sizes defining the OGM-G algorithm was numerically observed to be a rearrangement of those of OGM \cite{KimFessler2021_optimizing}, and \cite{KimOzdaglarParkRyu2023_timereversed} formalized and generalized this relationship by introducing the concept of H-duality.
Similar approaches were also presented in \cite{KimYang2023_unifying, KimYang2023_convergence}.
\cite{YoonKimSuhRyu2024_optimal} extended H-duality to nonexpansive fixed-point problems (or the equivalent proximal monotone inclusion setting) and showed that the H-duals of OHM and FEG achieve the exact same last-iterate rate as their respective primal algorithms.
For this particular setup, an infinite family of exact optimal algorithms was fully characterized in \cite{YoonRyuGrimmerInvariance2025} and the H-duality between these optimal algorithms has been further explored in \citep{yoon2026theorycompositiondualityextremal}.

\paragraph{Lyapunov analysis in optimization and computer assistance.} 
Once an appropriate Lyapunov function encoding the intermediate progress of an algorithm is given, one can verify its nonincreasingness along the iterations to carry out the convergence analysis.
Such a strategy has been extensively used in the first-order optimization literature, within prior work including but not limited to \cite{Nesterov1983_method, 
VanScoyFreemanLynch2018_fastest, 
BansalGupta2019_potentialfunction, TaylorBach2019_stochastic, dAspremontScieurTaylor2021_acceleration, YoonRyu2021_accelerated, LeeKim2021_fast, LeeParkRyu2021_geometric, DiakonikolasWang2022_potential, ParkRyu2022_exact, KimOzdaglarParkRyu2023_timereversed,  ParkRyu2024_optimal, YoonKimSuhRyu2024_optimal, yoonAcceleratedMinimaxAlgorithms2025,JangGuptaRyu2025_computerassisted}.
A central challenge in establishing a Lyapunov analysis is to identify a correct form to work with, as there is no universal strategy for it.
To address this, several computer-assisted approaches for identifying Lyapunov functions or similar type of convergence argument have been developed, either based on the integral quadratic constraints (IQC) perspective \citep{lessardAnalysisDesignOptimization2016, TaylorVanScoyLessard2018_lyapunov} or via general SDP formulation \citep{TaylorVanScoyLessard2018_lyapunov, TaylorBach2019_stochastic, UpadhyayaBanertTaylorGiselsson2025_automated}.
The ideas from this line of work have also been integrated into the software package \texttt{AutoLyap} \citep{UpadhyayaGuptaTaylorBanertGiselsson2026_autolyap}. 
To clarify, while alternative terminologies such as potential/energy functions are often used in the literature, sometimes with distinguished usage, we broadly use the terms ``Lyapunov function'' or ``Lyapunov-style proof'' to include any arguments that establish a nonincreasing quantity.
These frameworks automate the search over a prescribed candidate family of Lyapunov functions, which is a quadratic function (ansatz) involving a chosen collection of vectors.
A suitable Lyapunov function satisfying desirable properties exists within the family if and only if a corresponding semidefinite feasibility problem has a solution, as shown in \cite{UpadhyayaBanertTaylorGiselsson2025_automated}, which enables numerical search with computer assistance.

On the other hand, \ref{eq:proof_template} is more general in the sense that it is not constrained to a specific proof template, and because the search space is larger, it may produce a tighter result, often strictly.
In fact, without a good intuition on the candidate family of Lyapunov function, prior approaches may fail to identify the tight performance guarantee for a given algorithm.
Indeed, for historical examples such as OGM, OGM-G and APPM, 
the \ref{eq:proof_template}s were obtained first and their Lyapunov analyses appeared years later \citep{dAspremontScieurTaylor2021_acceleration,LeeParkRyu2021_geometric, ParkParkRyu2023_factorsqrt2,RyuYin2022_largescale}. 
Furthermore, there are remaining instances of \ref{eq:proof_template}s for which no corresponding Lyapunov analysis exists, such as \cite{guTightSublinearConvergence2020, dragomirOptimalComplexityCertification2022, GuYang2025_tight, altschulerAccelerationStepsizeHedging2025, GrimmerShuWang2025_composing,  YoonRyuGrimmerInvariance2025}. 
To the best of our knowledge, the only prior work that attempted to relate \ref{eq:proof_template}s to Lyapunov-style proofs is a short note from \cite[Section~6]{goujaudFundamentalProofStructures2023}.
We emphasize that the approach we take in this paper is distinct and \textit{complementary} to those of \citep{TaylorVanScoyLessard2018_lyapunov, TaylorBach2019_stochastic, UpadhyayaBanertTaylorGiselsson2025_automated, UpadhyayaGuptaTaylorBanertGiselsson2026_autolyap} that are based on predefined Lyapunov template.
We focus on how a \ref{eq:proof_template} can be transformed into a Lyapunov-style proof while maintaining the same proof certificates, and provide a systematic procedure for that conversion by identifying the basis vectors representing the incremental layers of the \ref{eq:proof_template}.

\paragraph{Software implementations of PEP.}

Several software packages have made computer-assisted performance analysis of algorithms more
accessible.
\texttt{PESTO} \citep{TaylorHendrickxGlineur2017_performance} is a MATLAB
toolbox for automatically computing tight worst-case performance guarantees
through PEP, and \texttt{PEPit} \citep{pepit2024} provides a generic open-source
Python interface for modeling and solving PEP.
For Lyapunov-style analysis, \texttt{AutoLyap}
\citep{UpadhyayaGuptaTaylorBanertGiselsson2026_autolyap} supports automated
search over prescribed Lyapunov templates.
We use \texttt{PEPFlow} \citep{SuhYingJiangNguyen2025_pepflow} as the 
basis for implementing our proposed procedure which transforms a 
\fullpep{} into a Lyapunov-style proof.
This is because \texttt{PEPFlow} provides a convenient tagging system for vectors corresponding to the algorithm iterates, function value or gradient/operator evaluated at those points, and the inner product terms among them.
This feature is useful for maintaining an interpretable implementation where mathematical notations can be directly used in the code to access the vectors and matrices needed for linear algebraic manipulations in our procedure.

\subsection{Notation and convention}

Here we quickly list the notation that will be used throughout the paper. 
$\reals^{m\times n}$ denotes the space of all $m\times n$ matrices.
$\mathbb{S}^n$ is the space of all symmetric $n\times n$ matrices, and we denote the space of all symmetric positive semidefinite $n\times n$ matrices by $\mathbb{S}_+^n$.
The $i$-th standard basis vector, in any Euclidean space of dimension $\ge i$, will be denoted by $\ve_i$.
Given a square matrix $A$, $\mathrm{Sym}(A)$ is its symmetrization: $\mathrm{Sym}(A) = \frac{1}{2}(A + A^\transpose)$.
For a general $m\times n$ matrix, $\cR (A)$ is its range (column space) and $\cN (A)$ is its null space.
We often use $\star$ as an index for denoting the optimum, and we will take the convention $\star=\infty$ and often use it with other number indices for specifying the iterates;
then we can conveniently state $\star > k$ for any $k\in \mathbb{N}$.
Finally, we will reserve the notation $N$ for the last iterate or termination horizon of the algorithm where the performance measure is evaluated.

\paragraph{Organization.}
The rest of the paper is organized as follows. 
\cref{section:background} sets up the preliminary technical concepts for the problem settings of interest for this work.
\cref{section:theory} explains how a Lyapunov function $V_k$ can be constructed from a PEP-style proof and how we characterize admissible $V_k$'s.
Then, we present detailed steps for detecting a concise expression for $V_k$ and carrying out the analytic proof from it.
\cref{section:existing-lyapunov} presents existing Lyapunov-style analyses that can be recovered by our framework.
\cref{section:new-lyapunov} presents new Lyapunov-style analyses discovered using our framework: gradient descent, Bregman proximal gradient method for composite minimization, proximal point method for monotone inclusion, and Dual-OC-Halpern.
For all algorithms covered in Sections~\ref{section:existing-lyapunov} and \ref{section:new-lyapunov}, we provide the accompanying \texttt{Jupyter notebook} implementation of the full procedure, which we open source through the codebase
{\color{blue}
\begin{center}
    \url{https://github.com/pepflow-lib/PEPFlow/tree/main/examples_lyapunov} 
\end{center}}

%% file: figures/contribution_table.tex
\begin{table}[H]
\centering
\renewcommand{\arraystretch}{1.25}
\begin{tabular}{cccc}
\hline
\textbf{Problem setting} & \textbf{Algorithm} & \textbf{Lyapunov proof} & \textbf{Section} \\
\hline
\multirow{4}{*}{Smooth convex minimization}
& GD \cite{DroriTeboulle2014_performance} 
& \textbf{New}
& \S\ref{section:gd}  \\
& OGM1 \cite{KimFessler2016_optimized}
& \cite{dAspremontScieurTaylor2021_acceleration, ParkParkRyu2023_factorsqrt2}
& \S\ref{section:ogm-without-z} \\
& OGM2 \cite{KimFessler2016_optimized}
& \cite{dAspremontScieurTaylor2021_acceleration, ParkParkRyu2023_factorsqrt2}
& \S\ref{section:ogm} \\
& OGM-G \cite{KimFessler2021_optimizing}
& \cite{LeeParkRyu2021_geometric}
& \S\ref{section:ogm-g} \\
\hline
\multirow{1}{*}{Bregman composite minimization}
& BPGM \cite{ZhouLiangShen2019_simple, GutmanPena2023_perturbed}
& \textbf{New}
& \S\ref{section:bpgm} \\
\hline
\multirow{2}{*}{Smooth minimax optimization}
& FEG \cite{LeeKim2021_fast}
& \cite{LeeKim2021_fast}
& \S\ref{section:feg} \\
& Dual-FEG \cite{YoonKimSuhRyu2024_optimal}
& \cite{YoonKimSuhRyu2024_optimal}
& \S\ref{section:dual-feg} \\
\hline
\multirow{3}{*}{%
\begin{tabular}[c]{@{}c@{}}
Proximal monotone inclusion \\
($\leftrightarrow$ Nonexpansive fixed-point problem)
\end{tabular}}
& APPM \cite{Kim2021_accelerated} / OHM \cite{Lieder2021_convergence}
& \cite{Diakonikolas2020_halpern, RyuYin2022_largescale}
& \S\ref{section:appm-ohm} \\
& Dual-OHM \cite{YoonKimSuhRyu2024_optimal}
& \cite{YoonKimSuhRyu2024_optimal}
& \S\ref{section:dual-ohm} \\
& PPM \cite{guTightSublinearConvergence2020}
& \textbf{New}
& \S\ref{section:ppm} \\
\hline
\multirow{2}{*}{%
\begin{tabular}[c]{@{}c@{}}
Proximal strongly monotone \\
($\leftrightarrow$ Contractive fixed-point problem)
\end{tabular}}
& OS-PPM / OC-Halpern \cite{ParkRyu2022_exact}
& \cite{ParkRyu2022_exact}
& \S\ref{section:os-ppm-and-oc-halpern} \\
& \textbf{Dual-OC-Halpern (New)}
& \textbf{New}
& \S\ref{section:dual-oc-halpern} \\
\hline
\end{tabular}
\caption{Summary of algorithms covered by the framework proposed in this paper.
For each algorithm, if it is classical, we cite the work providing its tight convergence proof, while if it is relatively recent, we cite the work that first introduced the algorithm and proved its convergence.
The ``Lyapunov proof'' column lists the work that provided a Lyapunov-style proof with tight rate for each algorithm, and ``New'' indicates that we newly introduce them in our \cref{section:new-lyapunov}. 
}
\label{table:contribution-summary}
\end{table}

%% file: sections/2_background.tex
\section{Problem settings and convergence proofs of consideration}
\label{section:background}

In this section, we review how tight \ref{eq:proof_template}s are written in the two settings considered in this work: smooth convex minimization and (strongly) monotone inclusion with proximal operations.
While we additionally consider smooth convex-concave minimax optimization (via reformulation into Lipschitz-continuous monotone inclusion) and composite minimization under Bregman geometry in the paper, we defer the detailed descriptions for these settings to \ref{section:minimax-optimization} and Appendix~\ref{section:setting-composite-minimization-bregman}.

\subsection{Smooth convex minimization}
\label{section:setting-smooth-convex-minimization}

Suppose that we are interested in analyzing the convergence of a deterministic first-order algorithm for solving
\[
\begin{array}{cc}
    \underset{x\in \reals^d}{\text{minimize}} & f(x)
\end{array}
\]
where the objective function $f\colon \reals^d \to \reals$ is convex and $L$-smooth (i.e., $f$ is differentiable and $\nabla f$ is $L$-Lipschitz continuous), where $L>0$, and $f$ has a minimizer $x_\star \in \reals^d$.
It is well known that $f$ is convex and $L$-smooth if and only if the \emph{cocoercivity inequalities}
\[
f(y) - f(x) + \inprod{\nabla f(y)}{x-y} + \frac{1}{2L} \sqnorm{\nabla f(x) - \nabla f(y)} \le 0
\]
holds for all $x,y\in \reals^d$ \citep[Theorem~2.1.5]{Nesterov2004_introductory}.
Furthermore, when the algorithm is a fixed-step first-order method with update rule of the form 
\[x_{k+1} = x_k - \sum_{j=0}^k h_{k+1,j} \nabla f(x_j)\]
where $h_{k+1,j} \in \reals$ are pre-defined step-sizes for $k=0,1,\dots$ and $j=0,\dots,k$, 
then it has been proved \citep{TaylorHendrickxGlineur2017_smooth} that its 
tight\footnote{This means that there exists a convex and $L$-smooth function $f$ defined on $\reals^d$ with sufficiently large $d$ such that running the algorithm on that function yields $f(x_N) - f(x_\star) = \nu\sqnorm{x_0 - x_\star}$,
i.e., the worst-case bound bound cannot be improved.}
dimension-independent worst-case convergence rate of the form $f(x_N) - f(x_\star) \le \nu \sqnorm{x_0 - x_\star}$ 
can always be established by a proof of the form
\begin{align}
\begin{aligned}
    f(x_N) - f(x_\star) - \nu \sqnorm{x_0 - x_\star} & = \sum_{i,j \in \{0,\dots,N,\star\}} \lambda_{i,j} \left( f(x_j) - f(x_i) + \inprod{\nabla f(x_j)}{x_i - x_j} + \frac{1}{2L} \sqnorm{\nabla f(x_i) - \nabla f(x_j)} \right) \\
    & \quad - \sum_{i \in \cJ} \alpha_i \sqnorm{\begin{bmatrix} x_0 & x_\star & \nabla f(x_0) & \cdots & \nabla f(x_N) \end{bmatrix} \vs_i} \\
    & \le 0 .
\end{aligned}
\label{eqn:pep-proof-smooth-convex}
\end{align}
Here $\lambda_{i,j} \ge 0$ are weights associated with the cocoercivity inequalities between the iterates, characterizing the problem class, 
while $\alpha_i \ge 0$ and $\vs_i \in \reals^{N+3}$ are quantities representing the sum of squares terms, indexed with a finite set $\cJ$.
The quantities $\lambda_{i,j}$, $\alpha_i$ and $\vs_i$ are what determine a convergence proof, and we call them \emph{proof certificates}.

Following \citep{TaylorHendrickxGlineur2017_smooth}, we introduce an equivalent reformulation enabling the numerical search of proof certificates within the PEP framework via solving semidefinite programs.
This will also allow us to design the procedure for Lyapunov-style conversion based on linear algebraic techniques, which will be explained in the next section. 
Let
\begin{align*}
    \vP = \begin{bmatrix} x_0 & x_\star & \nabla f(x_0) & \cdots & \nabla f(x_N) \end{bmatrix} \in \reals^{d \times (N+3)} ,
\end{align*}
$\vx_0 = \ve_1$, $\vx_\star = \ve_2$, $\vg_i = \ve_{i+3}$ for $i=0,\dots,N$ and inductively define $\vx_1, \dots, \vx_N$ by
\begin{align*}
    \vx_{i+1} = \vx_i - \sum_{j=0}^i h_{i+1,j} \vg_j
\end{align*}
for $i=0,\dots,N-1$, so that $\vP\vx_i = x_i$ and $\vP\vg_i = \nabla f(x_i)$ for all $i\in \{0,\dots,N\}$.
Then let $\vG = \vP^\transpose \vP \in \mathbb{S}_+^{N+3}$ be the Gram matrix of the column vectors of $\vP$.
Finally, let
\begin{align*}
    \vf = \begin{bmatrix} f(x_0) & f(x_1) & \cdots & f(x_N) & f(x_\star) \end{bmatrix}^\transpose \in \reals^{N+2} .
\end{align*}
With this setup, we can rewrite \eqref{eqn:pep-proof-smooth-convex} as
\begin{align*}
    \mathbf{f}_{N+1} - \mathbf{f}_{N+2} - \nu \mathrm{Tr}(\vG \vD) = \sum_{i,j\in \{0,\dots,N,\star\}} \lambda_{i,j} (\vf_j - \vf_i + \mathrm{Tr}(\vG\vC_{i,j})) - \sum_{i\in \cJ} \alpha_i \mathrm{Tr}(\vG \vS_i) ,
\end{align*}
where we let $\vf_\star := \vf_{N+2} = f(x_\star)$, $\vC_{i,j} := \mathrm{Sym}\left( \vg_j \left(\vx_i - \vx_j\right)^\transpose \right) + \frac{1}{2L} (\vg_i - \vg_j) (\vg_i - \vg_j)^\transpose \in \mathbb{S}^{N+3}$ encodes the inner product terms appearing in the cocoercivity inequality between $x_i$ and $x_j$,
$\vD = (\vx_0-\vx_\star)(\vx_0-\vx_\star)^\transpose \in \mathbb{S}_+^{N+3}$ represents the initial condition, and each $\vS_i = \vs_i \vs_i^\transpose \in \mathbb{S}_+^{N+3}$ is a rank~1 matrix corresponding to a single norm square term constituting \eqref{eqn:pep-proof-smooth-convex}.

\subsection{Maximally monotone inclusion with proximal algorithms}
\label{section:setting-proximal-monotone-inclusion}

A monotone inclusion is given by the form
\[
\begin{array}{cc}
    \underset{x\in \reals^d}{\text{find}} & 0 \in \opA(x)
\end{array}
\]
where $\opA\colon \reals^d \rightrightarrows \reals^d$ is a multi-valued (set-valued) operator that is \emph{maximally $\mu$-strongly monotone} for $\mu \ge 0$, i.e., $\inprod{u-v}{x-y} - \mu \sqnorm{x-y} \ge 0$ for all $x,y\in \reals^d$, $u\in \opA(x)$ and $v\in \opA(y)$ and has a graph that cannot be enlarged without violating monotonicity.
When $\mu=0$, this reduces to the maximally monotone setting.
For conciseness, we denote $\opA x = \opA(x)$ for $x\in \reals^d$.
We assume that $\opA$ has a zero, i.e., $x_\star \in \reals^d$ such that $0 \in \opA x_\star$.
We consider the proximal algorithms that use the resolvent $\opJ_\opA = (\opI + \opA)^{-1}$; more concretely, algorithms such that
\begin{align}
\label{eqn:H-matrix-algorithm-monotone-proximal}
    x_k = \opJ_\opA (y_{k-1}) , \quad
    y_k = y_{k-1} + \sum_{j=1}^k h_{k,j} \tilde{\opA} x_j
\end{align}
for $k=1,2,\dots$, where $x_0 = y_0$ is the initial point, $\tilde{\opA} x_{j+1} = y_j - x_{j+1} \in \opA x_{j+1}$, and $h_{k,j}$ are pre-determined step-sizes for $k=1,2,\dots$ and $j=1,\dots,k$.

With the correspondence $\opT = \left(1 + \frac{1}{1+2\mu}\right) \opJ_\opA - \frac{1}{1+2\mu} \opI$, the problem class becomes equivalent to the class of fixed point problems \citep{ParkRyu2022_exact}:
\begin{align*}
\begin{array}{cc}
    \underset{y \in \reals^d}{\text{find}}  &  y = \opT (y)
\end{array}
\end{align*}
where $\opT\colon \reals^d \to \reals^d$ is \emph{$\frac{1}{\gamma}$-contractive}, i.e., $\norm{\opT x - \opT y} \le \frac{1}{\gamma}\norm{x-y}$ for all $x,y\in \reals^d$,
where $\gamma = 1+2\mu \ge 1$.
(When $\gamma=1$, we say $\opT$ is \emph{nonexpansive}.)
The zero of $\opA$ is a fixed point of $\opT$, i.e., $x_\star = y_\star$ satisfies $y_\star = \opT y_\star$.
Furthermore, the algorithm class can be equivalently rewritten as
\begin{align*}
    y_{k+1} = y_k + \sum_{j=0}^k \frac{\gamma}{\gamma +1} h_{k+1,j+1} (y_j - \opT y_j) .
\end{align*}
The inequalities characterizing the proximal monotone inclusion setting and the nonexpansive fixed point setting are equivalent:
\begin{align*}
    \inprod{\tilde{\opA} x_i - \tilde{\opA} x_j}{x_i - x_j} - \mu \sqnorm{x_i - x_j} \ge 0 & \iff \frac{1}{\gamma^2} \sqnorm{y_{i-1} - y_{j-1}} - \sqnorm{\opT y_{i-1} - \opT y_{j-1}} \ge 0 \\
    \inprod{\tilde{\opA} x_i}{x_i - x_\star} - \mu \sqnorm{x_i - x_\star} \ge 0 & \iff \frac{1}{\gamma^2} \sqnorm{y_{i-1} - y_\star} - \sqnorm{\opT y_{i-1} - \opT y_\star} \ge 0
\end{align*}
The performance metric is $\sqnorm{\tilde{\opA} x_N}$, or equivalently up to a constant factor, $\sqnorm{y_{N-1} - \opT y_{N-1}}$.

By \citep{RyuTaylorBergelingGiselsson2020_operator}, any tight proof is given by the form
\begin{align}
\label{eqn:pep-proof-monotone-proximal}
\begin{aligned}
    \sqnorm{\topa x_N} - \nu \sqnorm{x_0 - x_\star} = & -\sum_{\substack{i,j \in \{1,\dots,N,\star\}\\ i<j}} \lambda_{i,j} \left( \inprod{\tilde{\opA} x_i - \tilde{\opA} x_j}{x_i - x_j} 
    - \mu \sqnorm{x_i - x_j} \right) \\
    & - \sum_{i\in \cJ} \alpha_i \sqnorm{\begin{bmatrix} x_0 & x_\star & \tilde{\opA} x_1 & \cdots & \tilde{\opA} x_N \end{bmatrix} \vs_i} .
\end{aligned}
\end{align}
Next, we present the similar PEP-style reformulation for this setting.
Let $g_i=\tilde{\opA}x_i\in \opA x_i$ for $i=1,\dots,N$ and set $g_\star = 0 \in \opA x_\star$.
Let
\[
    \vP = \begin{bmatrix}x_0 & x_\star & g_1 & \cdots & g_N\end{bmatrix}\in\reals^{d\times(N+2)}
\]
and let $\vx_0=\vy_0=\ve_1$, $\vx_\star=\ve_2$, $\vg_i=\ve_{i+2}$ for $i=1,\dots,N$ and $\vg_\star=0$.
Define $\vy_k, \vx_k$ for $k=1,2,\dots$ by
\[
    \vy_{k} = \vy_{k-1} + \sum_{j=1}^k h_{k,j} \vg_j , \quad \vx_k = \vy_{k-1} - \vg_k ,
\]
so that $\vP\vx_k=x_k$ and $\vP\vy_k = y_k$ for all $k\ge 0$.
With $\vG=\vP^\transpose\vP$, define
\[
    \vM_{i,j}^{(\mu)}
    :=
    \mathrm{Sym}\left((\vg_i-\vg_j)(\vx_i-\vx_j)^\transpose\right)
    -\mu(\vx_i-\vx_j)(\vx_i-\vx_j)^\transpose ,
\]
$\vQ_N=\vg_N\vg_N^\transpose$ and $\vD=(\vx_0-\vx_\star)(\vx_0-\vx_\star)^\transpose$.
Then \eqref{eqn:pep-proof-monotone-proximal} can be rewritten as
\[
    \mathrm{Tr}(\vG\vQ_N)-\nu\mathrm{Tr}(\vG\vD)
    =
    -\sum_{\substack{i,j \in \{1,\dots,N,\star\}\\ i<j}}\lambda_{i,j}\mathrm{Tr}(\vG\vM_{i,j}^{(\mu)})
    -\sum_{i\in\cJ}\alpha_i\mathrm{Tr}(\vG\vS_i)
\]
where $\lambda_{i,j}\ge0$, $\alpha_i\ge0$, and $\vS_i=\vs_i\vs_i^\transpose \succeq 0$.

%% file: sections/3_theory.tex
\section{Theory for searching Lyapunov-style proofs}
\label{section:theory}

Given a PEP-style proof as provided in the previous section, our qualitative goal is to translate it into a simple Lyapunov-style proof.
In this section, we first define what we consider as an \emph{admissible} Lyapunov analysis, and based on that characterization, provide a concrete five-step procedure for identifying a Lyapunov structure (\cref{fig:lyapunov-conversion-workflow}).
Throughout the section, we focus on the case of non-strongly monotone inclusion problems---the setting described in Section~\ref{section:setting-proximal-monotone-inclusion} with $\mu=0$---for simplicity, but the same idea easily extends to all the other settings.

\input{figures/lyapunov_conversion_workflow}

\subsection{Construction of a Lyapunov function}

Consider the proof format \eqref{eqn:pep-proof-monotone-proximal}. 
Let $\cI = \left\{(i,j): i,j \in \{1,\dots,N,\star\}, i < j , \lambda_{i,j} > 0 \right\}$ be the set of indices corresponding to inequalities that are ``active'' within the proof.
Then any increasing sequence of subsets $\cI_0 \subset \cI_1 \subset \dots \subset \cI_{N-1} \subset \cI$ and $\cJ_0 \subset \cJ_1 \subset \dots \subset \cJ_{N-1} \subset \cJ$ 
yields a Lyapunov function
\begin{align}
\label{eqn:Vk-definition-monotone-proximal}
\begin{aligned}
    V_k & = -\sum_{(i,j) \in \cI_k} \lambda_{i,j} \inprod{\topa x_i - \topa x_j}{x_i - x_j} - \sum_{i\in \cJ_k} \alpha_i \sqnorm{\begin{bmatrix} x_0 & x_\star & \topa x_1 & \cdots & \topa x_N \end{bmatrix} \vs_i} \\
    & = -\sum_{(i,j) \in \cI_k} \lambda_{i,j} \Tr(\vG\vM^{(0)}_{i,j}) - \sum_{i\in\cJ_k} \alpha_i \Tr(\vG\vS_i) .
\end{aligned}
\end{align}
Indeed, by definition of $V_k$ and the fact that $\lambda_{i,j} \Tr(\vG\vM^{(0)}_{i,j}) \ge 0$ and $\alpha_i \Tr(\vG\vS_i) \ge 0$, we naturally have $0 \ge V_0 \ge V_1 \ge \dots \ge V_{N-1}$.
The final convergence proof will then follow from an additional last step:
\begin{align*}
    & \sqnorm{\topa x_N} - \nu \sqnorm{x_0 - x_\star} \\
    & = V_{N-1} - \sum_{(i,j) \in \cI\setminus \cI_{N-1}} \lambda_{i,j} \inprod{\topa x_i - \topa x_j}{x_i - x_j}  - \sum_{i\in \cJ \setminus \cJ_{N-1}} \alpha_i \sqnorm{\begin{bmatrix} x_0 & x_\star & \topa x_1 & \cdots & \topa x_N \end{bmatrix} \vs_i} \\
    & \le V_{N-1} \le 0 .
\end{align*}
Thus, in principle, any PEP-style proof can become a Lyapunov-style proof.
However, the resulting Lyapunov function may not be insightful or concise.
Therefore, a more important question is to clarify when the choice of $\cI_k$ and $\cJ_k$ will yield a ``nice'' Lyapunov function, which we address in the following Section~\ref{section:characterization-of-admissibility}.

\paragraph{Remark.}
The length of the towers of indices $\cI_k, \cJ_k$ may vary;
it typically follows the total number of iterations, and the final index is chosen to be $N-1$ in the proximal monotone inclusion setting described above because the performance measure $\sqnorm{\topa x_N} = \frac{1}{4}\sqnorm{y_{N-1} - \opT y_{N-1}}$ is associated with $y_{N-1}$, which can be viewed as the main iterate where the measure is evaluated.
For other settings, such as smooth convex minimization, the performance measure is typically associated with $x_N$, so the final index will be $N$.
It is also possible to consider extensions where Lyapunov function indices do not necessarily follow the iteration indices, but we focus on the most natural cases where they are in agreement.

\paragraph{Comparison with \cite{goujaudFundamentalProofStructures2023}.} We note that a similar approach of defining the Lyapunov function as a partial sum of nonpositive terms appearing in the PEP-style proof has been considered in \cite[Section~6]{goujaudFundamentalProofStructures2023} for smooth (strongly) convex minimization setting. 
Their formulation aggregates all terms depending on the quantities observed before the $k$-th iteration, and provided Nesterov's accelerated gradient (NAG) \cite{Nesterov1983_method} as an example.
We offer several layers of extension compared to their approach.
First, we consider a number of distinct problem settings and materially demonstrate the broad applicability of this technique.
Second, our definition of partial sums is not confined to the observations prior to a specific iteration number, which is a necessary generalization to handle algorithms with fixed horizon such as OGM-G \cite{KimFessler2021_optimizing}, Dual-OHM \cite{YoonKimSuhRyu2024_optimal} and Dual-OC-Halpern (Section~\ref{section:dual-oc-halpern}).
Third, we complement the partial sum construction with concrete linear algebraic criteria and implementable procedure for determining whether it can be simplified to produce a meaningful Lyapunov function, as detailed in the subsequent sections.

\paragraph{Example: APPM/OHM.}
As a representative example, consider the \emph{accelerated proximal point method} of \citep{Kim2021_accelerated}:
\begin{align}
\begin{aligned}
    x_{k+1} = \opJ_\opA (y_k) , \quad y_{k+1}  = x_{k+1} + \frac{k}{k+2} (x_{k+1} - x_k) - \frac{k}{k+2} (x_k - y_{k-1})
\end{aligned}
\label{alg:appm}
\tag{APPM}
\end{align}
with $x_0=y_0$ and $\topa x_{k+1}:=y_k-x_{k+1}\in\opA x_{k+1}$.
In the equivalent fixed-point form with $\opT=2\opJ_\opA-\opI$, this becomes the \emph{optimal Halpern method} \citep{Lieder2021_convergence,ParkRyu2022_exact}
\begin{align}
    y_{k+1} = \frac{1}{k+2} y_0 + \frac{k+1}{k+2} \opT y_k .
\label{alg:ohm}
\tag{OHM}
\end{align}
The PEP-style convergence proof for these equivalent algorithms \citep{Kim2021_accelerated,Lieder2021_convergence,ContrerasCominetti2023_optimal, RyuYin2022_largescale} takes the form
\begin{align*}
    \sqnorm{\topa x_N} - \frac{\sqnorm{y_0 - y_\star}}{N^2}
    &=
    -\sum_{j=1}^{N-1} 
    \lambda_{j,j+1}
    \inprod{\topa x_{j} - \topa x_{j+1}}{x_j - x_{j+1}}
    - \lambda_{N,\star} \inprod{\topa x_N}{x_N - y_\star}
    - \sqnorm{\topa x_N - \frac{1}{N}(y_0 - y_\star)} 
\end{align*}
where we specify $\lambda_{j,j+1}, \lambda_{N,\star} > 0$ later.
In particular, $\cI = \{(j,j+1)\,|\,j=1,\dots,N-1\} \cup \{(N,\star)\}$ and $|\cJ|=1$ with the corresponding unique square term being $\sqnorm{\topa x_N - \frac{1}{N}(y_0 - y_\star)}$.
A natural increasing set of inequality indices is to follow the chronological order and take:
\[
    \cI_k=\{(j,j+1)\,|\,j=1,\dots,k\},
    \quad
    \cJ_k=\emptyset
\]
for $k=1,\dots,N-1$, while we take $\cI_0 = \cJ_0 = \emptyset$.
This gives $V_0 = 0$, and
\begin{align}
\label{eqn:appm-ohm-partial-sum-lyapunov}
    V_k
    =
    -\sum_{j=1}^{k} \lambda_{j,j+1}
    \inprod{\topa x_{j} - \topa x_{j+1}}{x_j - x_{j+1}} ,
    \qquad
    k=1,\dots,N-1.
\end{align}
By monotonicity of $\opA$, we have $0 = V_0\ge V_1\ge \dots \ge V_{N-1}$.

\subsection{Characterization of admissible Lyapunov functions}
\label{section:characterization-of-admissibility}

In this section, we propose a precise characterization of Lyapunov-style proofs that can be perceived as nice and concise to human readers.
We start with the following simple proposition, translating the behavior of $V_k$ into linear algebraic properties.
Its simple proof is presented in Appendix~\ref{appendix:lyapunov-proposition-proof}.

\begin{proposition}
\label{proposition:linear-algebra-of-lyapunov}
Let $\opA\colon \reals^d \rightrightarrows \reals^d$ be a maximally monotone operator.
Let $x_k$, $y_k$ be the iterates of an algorithm defined by~\eqref{eqn:H-matrix-algorithm-monotone-proximal},
and let $V_k$ be the Lyapunov function defined by \eqref{eqn:Vk-definition-monotone-proximal}.
Then, $V_k = \Tr(\vG\vV_k)$, where 
\[
    \vV_k = -\sum_{(i,j) \in \cI_k} \lambda_{i,j} \vM_{i,j}^{(0)} - \sum_{i\in \cJ_k} \alpha_i \vS_i \in \mathbb{S}^{N+2} .
\]
Furthermore, let $r_k = \mathrm{rank}(\vV_k)$ and $\cR(\vV_k)$ be the range of $\vV_k$.
Then, for any basis $\{\vv_{k,1}, \dots, \vv_{k,r_k}\}$ of $\cR(\vV_k)$, there exists $\vA_k \in \mathbb{S}^{r_k}$ such that
\begin{align}
\label{eqn:Lyapunov-expression-matrix-decomposition}
    \vV_k = \begin{bmatrix} \vv_{k,1} & \cdots & \vv_{k,r_k} \end{bmatrix} \vA_k \begin{bmatrix} \vv_{k,1} & \cdots & \vv_{k,r_k} \end{bmatrix}^\transpose ,
\end{align}
so $V_k$ has the vector quadratic form expression
\begin{align}
\label{eqn:Lyapunov-expression-quadratic-form}
    V_k = \sum_{p,q=1}^{r_k} (\vA_k)_{p,q} \inprod{\vP\vv_{k,p}}{\vP\vv_{k,q}} .
\end{align}
\end{proposition}

Given the above characterization, we set up the criteria for \emph{admissible} Lyapunov functions as follows:
\begin{itemize}
    \item \textbf{Consistency and conciseness.} This requires that the number of vectors needed to express $V_k$ should not increase with $k$ and stay virtually constant across $k=0,\dots,N-1$.
    Concretely, we require that $r_k = r$ for all $k=1,\dots,N-2$ for some fixed $r$.
    Note that we allow exceptions at the indices $k=0$ and $k=N-1$, where $V_k$ often gets simplified than they are at generic iteration numbers and the rank is reduced.
    We additionally require that $r\le 4$, as we find that this is the largest rank hit by the most complex Lyapunov analyses covered in this work.

    \item \textbf{Sufficiency.} Establishing the nonincreasingness property $V_{N-1} \le \dots \le V_0$ should essentially complete all the heavy lifting, and the remaining last step needed to derive the final convergence rate should be minimal.
    We quantify this with the condition $|\cI \setminus \cI_{N-1}| \le 1$ and $|\cJ \setminus \cJ_{N-1}| \le 1$, i.e., the last step should use at most one problem-specific inequality and one norm-square nonnegativity argument, such as Young's inequality.
    
    \item \textbf{Locality.} This requires that $V_k$ should be expressible in terms of the local quantities, such as the terms involving the local iterates $x_k, y_k$ or $x_{k\pm 1}, y_{k\pm 1}$.
    If necessary, one may introduce auxiliary sequences for the purpose of Lyapunov analysis, which we typically denote as $z_k$.
    Additionally, $V_k$ may involve some fixed iteration indices not depending on $k$ (typically $0, N$ or $\star$).
    In linear algebraic terms, there should be a basis $\{\vv_{k,1},\dots,\vv_{k,r}\}$ of $\cR(\vV_k)$ that is locally expressible, e.g., $\vv_{k,j} \in \mathrm{span}\left\{ \vx_0, \vx_k, \vx_N, \vx_\star, \vg_k, \vg_N \right\}$ for $j=1,\dots,r$.
    However, we keep $\vv_{k,j}$ flexible rather than restricting them to a limited format, and provide a practical procedure for searching these local basis vectors.
\end{itemize}

\paragraph{Example: APPM/OHM continued.}
For the partial sums $V_k$ in \eqref{eqn:appm-ohm-partial-sum-lyapunov}, the corresponding matrices $\vV_k$ have rank $2$ for generic $k$, regardless of $N$.
Hence, $V_k$ satisfies consistency and conciseness.
Furthermore, it satisfies the sufficiency property because $\cI\setminus\cI_{N-1} = \{(N,\star)\}$ and $\cJ\setminus\cJ_{N-1}=\cJ$ are both singletons.
Due to this property, the convergence proof is almost complete once $V_{N-1} \le 0$ is established; the remaining steps are as simple as
\begin{align}
\label{eqn:appm-ohm-last-step}
    0 \ge V_{N-1}
    \ge V_{N-1}
    - \lambda_{N,\star} \inprod{\topa x_N}{x_N-x_\star}
    - \sqnorm{\topa x_N-\frac{1}{N}(y_0-x_\star)} =
    \sqnorm{\topa x_N}
    -\frac{1}{N^2}\sqnorm{y_0-x_\star}.
\end{align}
Finally, it turns out that
\begin{align*}
    \{\vg_{k+1},\; \vy_0-\vx_{k+1}\} \subset \cR(\vV_k)
\end{align*}
is a basis of $\cR(\vV_k)$, which verifies the locality property. 
Then, by \cref{proposition:linear-algebra-of-lyapunov}, 
\begin{align}
\label{eqn:appm-ohm-lyapunov-template}
    V_k = a_k\sqnorm{\topa x_{k+1}} + b_k\inprod{\topa x_{k+1}}{y_0-x_{k+1}} + c_k\sqnorm{y_0-x_{k+1}} 
\end{align}
for some $a_k, b_k, c_k \in \reals$.

\paragraph{Remark.} In the setting where the objective functions' values are involved (such as convex minimization), $V_k$ will additionally involve the term of the form $\vf^\transpose \vb_k$.
In this case, we also require that $\vb_k$ is sparse and local, which can be checked straightforwardly without any linear algebraic tricks.

\paragraph{Cases with $|\cJ|>1$.} The APPM/OHM example shown above has $|\cJ|=1$ so we can take $\cJ_k = \emptyset$, which simplifies the analysis.
However, when $|\cJ|>1$, i.e., if more than one norm square terms are involved, one needs a systematic decomposition of the positive semidefinite matrix certificate 
$\vS = \sum_{i\in \cJ} \alpha_i \vs_i\vs_i^\transpose$, which is not unique.
In some cases, this can be done purely numerically, as in the case of gradient descent (presented in Section~\ref{section:gd}), 
while sometimes we use a known analytic decomposition from prior work, as in the case of proximal point method for monotone inclusion (presented in Section~\ref{section:ppm}).
We explain the details in those sections while we explore the corresponding examples.

\subsection{Procedure for Lyapunov conversion and analytic proof discovery}
\label{section:Lyapunov-conversion-procedure}

Based on our characterization from the previous section, we propose the procedure that determines admissibility of a Lyapunov function proposed via $\cI_k, \cJ_k$, 
and when that is the case, converts the PEP-style proof into the Lyapunov-style proof.
Here, we emphasize that with fixed $N$, the search of proof certificates via PEP is numerical, and all steps below, except for the last one, can be performed numerically based on those certificates.
In the last step, we assume that the analytic forms of the certificates are known, and use them to assemble the Lyapunov-style proof analytically.

\begin{enumerate}
    \item The user proposes a candidate $\cI_0 \subset \cI_1 \subset \dots \subset \cI_{N-1} \subset \cI$ and $\cJ_0 \subset \cJ_1 \subset \dots \subset \cJ_{N-1} \subset \cJ$.

    \item Compute $\vV_k$ for $k=0,\dots,N-1$ and verify rank consistency, conciseness and sufficiency.

    \item To verify locality, we need to find a basis of $\cR(\vV_k)$.
    We collect all numerical coordinate vectors corresponding to the algorithm-specific quantities.
    In case of APPM/OHM, an adequate choice is
    \begin{align}
    \label{eqn:special-vector-construction}
        \cT = \{\vy_0, \dots, \vy_{N-1}, \vx_\star, \vx_1, \dots, \vx_N, \vg_1, \dots, \vg_N \} , \quad \cS = \cT \cup \{\vu - \vv \,|\, \vu, \vv \in \cT \} .
    \end{align}
    For other algorithms, we still include into $\cT$ all algorithm iterates, operator/gradient evaluated there, and the optimum (so all columns of $\vP$ are included),
    but we may additionally include auxiliary vectors that are specially named, i.e., \emph{tagged}.
    Then we construct the set $\cS$ by augmenting $\cT$ with all pairwise differences of vectors in it.
    
    \item Check if $\cS$ contains a basis of $\cR(\vV_k)$. 
    \begin{itemize}
        \item[4-1.] When this is the case, given any basis $\{\vv_{k,1},\dots,\vv_{k,r}\} \subseteq \cS$ of $\cR(\vV_k)$, we can decompose $\vV_k$ as in \eqref{eqn:Lyapunov-expression-matrix-decomposition}.
        Among all bases within $\cS$, choose one that yields the sparsest coefficient matrix $\vA_k$.
        Verify that it is expressed with respect to local information (locality), and if possible, choose the bases for each $k$ so that it follows a consistent symbolic pattern.
        
        \item[4-2.] Otherwise, count the maximum number of linearly independent vectors in $\cS \cap \cR(\vV_k)$.
        When this is $r-1$, then we can numerically complete a basis of $\cR(\vV_k)$, defining an auxiliary sequence $\vz_k$ needed to perform the Lyapunov analysis.
        We outline this procedure in Section~\ref{section:basis-completion}.

        \item[4-3.] When there are at most $r-2$ linearly independent vectors in $\cS \cap \cR(\vV_k)$, 
        then iteration locality cannot be verified with our current procedure; hence abort. 
        This might be possible with more sophisticated linear algebraic tricks, but we do not pursue this extension for the current work.
    \end{itemize} 
    
    \item This final step requires the analytic characterization of proof certificates $\lambda_{i,j}, \alpha_i$ and $\vs_i$.
    We take the identity
    \begin{align*}
        V_k - V_{k+1} = \sum_{(i,j) \in \cI_{k+1}\setminus \cI_k} \lambda_{i,j} \inprod{\topa x_i - \topa x_j}{x_i - x_j}  
        + \sum_{i\in \cJ_{k+1} \setminus \cJ_k} \alpha_i \sqnorm{\begin{bmatrix} x_0 & x_\star & \topa x_1 & \cdots & \topa x_N \end{bmatrix} \vs_i}
    \end{align*}
    where $V_k$ and $V_{k+1}$ are expressed using the formula~\eqref{eqn:Lyapunov-expression-quadratic-form}.
    This should be a formal identity, which holds for any values of $x_0$ and $\topa x_i$.
    That is, once we plug in the algorithm's update rule and express both sides in vector quadratic forms, coefficients of all inner product terms among $x_0$, $\topa x_1, \dots, \topa x_N$ from both sides should be equal.
    Viewing that as a system of equations and solving it for Lyapunov coefficients $(\vA_k)_{p,q}$, we can determine analytic expressions for $V_k$.
    Note that this verification can be performed \emph{symbolically}, which establishes the Lyapunov analysis for any $N$ and $0\le k\le N-1$;
    therefore, at this point, the procedure has produced an \emph{analytic (not merely numerical)} convergence result.
\end{enumerate}

\paragraph{Remark.} 
We note that Step~3 above where we construct a set of vectors $\cS$ is a heuristic, and we have no general guarantee on how many candidate basis vectors $\cS \cap \cR(\vV_k)$ it will identify;
this is why we specify in Step 4-3 that we cannot complete the procedure if we detect less than $r-1$ of them.
In other words, the outcome of Steps~3 and 4 depends on whether $\cS$ provides a search space that is broad enough to reveal the structure of $V_k$.
Nevertheless, in all 12 examples that we provide throughout Sections~\ref{section:existing-lyapunov} and \ref{section:new-lyapunov}, these steps are executed successfully.

\paragraph{Example: APPM/OHM continued.}
For APPM/OHM, the Steps~1--4, performed numerically, easily identify the Lyapunov skeleton \eqref{eqn:appm-ohm-lyapunov-template}, and in fact $c_k = 0$.
Then in Step~5, we use the analytic formula $\lambda_{k+1,k+2} = \frac{2(k+1)(k+2)}{N^2}$ and solve the equation
\begin{align*}
    V_k-V_{k+1}
    & = a_k\sqnorm{\topa x_{k+1}} + b_k\inprod{\topa x_{k+1}}{y_0-x_{k+1}} 
    - a_{k+1}\sqnorm{\topa x_{k+2}} - b_{k+1}\inprod{\topa x_{k+2}}{y_0-x_{k+2}} \\
    & = \frac{2(k+1)(k+2)}{N^2} \inprod{\topa x_{k+1}-\topa x_{k+2}}{x_{k+1}-x_{k+2}} 
\end{align*}
for $a_k, b_k, a_{k+1}$ and $b_{k+1}$.
This reveals
\[
    a_k=\frac{2(k+1)^2}{N^2}, \quad b_k=-\frac{2(k+1)}{N^2}, \quad a_{k+1}=\frac{2(k+2)^2}{N^2} , \quad b_{k+1}=-\frac{2(k+2)}{N^2} 
\]
and we conclude that APPM/OHM has the concise Lyapunov function
\begin{align}
\label{eqn:appm-ohm-closed-form-lyapunov}
    V_k = \frac{2(k+1)^2}{N^2} \sqnorm{\topa x_{k+1}} - \frac{2(k+1)}{N^2}\inprod{\topa x_{k+1}}{y_0-x_{k+1}}
\end{align}
for all $k=0,1,\dots,N-1$, and establishing this essentially completes the convergence proof.

\subsection{Basis completion}
\label{section:basis-completion} 

Suppose we have found linearly independent basis candidate vectors $\vv_{k,1}, \dots, \vv_{k,r-1} \in \cS \cap \cR(\vV_k)$, where $\mathrm{rank}(\vV_k) = r$.
Then, take any $\vw_{k,r}$ satisfying $\cR(\vV_k) = \mathrm{span}\{\vv_{k,1}, \dots, \vv_{k,r-1}, \vw_{k,r}\}$, so that we can write
\begin{align*}
    \vV_k = \begin{bmatrix} \vv_{k,1} & \cdots & \vv_{k,r-1} & \vw_{k,r} \end{bmatrix} \Tilde{\vA}_k \begin{bmatrix} \vv_{k,1} & \cdots & \vv_{k,r-1} & \vw_{k,r} \end{bmatrix}^{\,\transpose}
\end{align*}
for some $\Tilde{\vA}_k \in \mathbb{S}^r$, by \cref{proposition:linear-algebra-of-lyapunov}.
Then write
\begingroup
\setlength{\arraycolsep}{8pt}
\renewcommand{\arraystretch}{1.2}
\begin{align*}
    \Tilde{\vA}_k = \begin{bmatrix}
    \begin{array}{c|c}
        \vB & \vu \\
        \hline
        \vu^\transpose & \tilde{c}
    \end{array}
    \end{bmatrix}
\end{align*}
\endgroup
If $\tilde{c} = 0$, then $\vu$ stays invariant up to a scalar multiple regardless of how $\vw_{k,r}$ is chosen, so it cannot be essentially simplified further.
Otherwise, we can rewrite $\tilde{\vA}_k$ as
\begingroup
\setlength{\arraycolsep}{8pt}
\renewcommand{\arraystretch}{1.25}
\[
    \tilde{\vA}_k = \begin{bmatrix}
        \begin{array}{c|c}
        \vI & \frac{\mathrm{sign}(\tilde{c})}{\sqrt{|\tilde{c}|}}\vu \\
        \hline
        \mathbf{0}^\transpose & \sqrt{|\tilde{c}|}
    \end{array}
    \end{bmatrix}
    \begin{bmatrix}
    \begin{array}{c|c}
        \vB - \frac{1}{\tilde{c}}\vu\vu^\transpose & \mathbf{0} \\
        \hline
        \mathbf{0}^\transpose & \mathrm{sign}(\tilde{c})
    \end{array}
    \end{bmatrix}
    \begin{bmatrix}
        \begin{array}{c|c}
        \vI & \mathbf{0} \\
        \hline
        \frac{\mathrm{sign}(\tilde{c})}{\sqrt{|\tilde{c}|}}\vu^\transpose & \sqrt{|\tilde{c}|}
    \end{array}
    \end{bmatrix} .
\]
\endgroup
Hence the coefficient matrix can be simplified with the basis change 
$\vv_{k,r} = \frac{\mathrm{sign}(\tilde{c})}{\sqrt{|\tilde{c}|}} \begin{bmatrix}
    \vv_{k,1} & \cdots & \vv_{k,r-1}
\end{bmatrix} \vu + \sqrt{|\tilde{c}|} \vw_{k,r}$:
\begingroup
\setlength{\arraycolsep}{8pt}
\renewcommand{\arraystretch}{1.2}
\begin{align*}
    \vV_k = \begin{bmatrix} \vv_{k,1} & \cdots & \vv_{k,r} \end{bmatrix} \vA_k \begin{bmatrix} \vv_{k,1} & \cdots & \vv_{k,r} \end{bmatrix}^{\,\transpose} ,
     \quad \vA_k = \begin{bmatrix}
    \begin{array}{c|c}
        \vB - \frac{1}{\tilde{c}}\vu\vu^\transpose & \mathbf{0} \\
        \hline
        \mathbf{0}^\transpose & \mathrm{sign}(\tilde{c})
    \end{array}
    \end{bmatrix} .
\end{align*}
\endgroup
In this way, we identify an auxiliary sequence $\vv_{k,r} = \vz_k$ that represents $V_k$ with maximally sparse coefficient matrix $\vA_k$,
provided that $\vv_{k,1}, \dots, \vv_{k,r-1}$ are fixed.

%% file: figures/lyapunov_conversion_workflow.tex
\begin{figure}[H]
\centering
\small
\begingroup
\renewcommand{\arraystretch}{1}
\setlength{\tabcolsep}{5pt}
\setlength{\jot}{1pt}

\def\lyapcellvpad{3pt}
\def\lyapheadheight{4.6ex}

\def\headcell#1#2{%
\begin{minipage}[c][\lyapheadheight][c]{#1}
\centering\bfseries #2
\end{minipage}%
}

\def\stepcell#1{%
\begin{minipage}[c]{0.07\textwidth}
\vspace*{\lyapcellvpad}
\centering #1\par
\vspace*{\lyapcellvpad}
\end{minipage}%
}

\def\taskcell#1{%
\begin{minipage}[c]{0.255\textwidth}
\vspace*{\lyapcellvpad}
\raggedright #1\par
\vspace*{\lyapcellvpad}
\end{minipage}%
}

\def\outcell#1{%
\begin{minipage}[c]{0.505\textwidth}
\vspace*{\lyapcellvpad}
\centering #1\par
\vspace*{\lyapcellvpad}
\end{minipage}%
}

\noindent\resizebox{0.99\textwidth}{!}{%
\begin{tabular}{c|c|c}
\hline
\headcell{0.07\textwidth}{Step}
&
\headcell{0.255\textwidth}{High-level task}
&
\headcell{0.505\textwidth}{Example outcome for APPM/OHM}
\\
\hline

\stepcell{1}
&
\taskcell{Define \(V_k\) as partial sums of nonpositive terms from the PEP-style proof}
&
\outcell{%
\(
\begin{aligned}
V_k
&=
-\sum_{j=1}^{k}\lambda_{j,j+1}
\inprod{\topa x_j-\topa x_{j+1}}{x_j-x_{j+1}} \\
\vV_k
&=
-\sum_{j=1}^{k}\lambda_{j,j+1}
\mathrm{Sym}\left(
(\vg_j-\vg_{j+1})(\vx_j-\vx_{j+1})^\transpose
\right)
\end{aligned}
\)
}
\\
\hline

\stepcell{2}
&
\taskcell{Check rank consistency, conciseness, and sufficiency}
&
\outcell{%
\(
\mathrm{rank}(\vV_k)=2 \text{ for all } k,\quad
|\cI\setminus\cI_{N-1}|=
|\cJ\setminus\cJ_{N-1}|=1
\)
}
\\
\hline

\stepcell{3}
&
\taskcell{Find local basis candidate vectors}
&
\outcell{%
\(
\vx_{k+1},\; \vy_0-\vx_{k+1}
\in \cR(\vV_k)
\)
}
\\
\hline

\stepcell{4}
&
\taskcell{Express the Lyapunov skeleton using local basis vectors}
&
\outcell{%
\(
\begin{aligned}
V_k
&=
a_k\sqnorm{\topa x_{k+1}}
+b_k\inprod{\topa x_{k+1}}{y_0-x_{k+1}}
+c_k\sqnorm{y_0-x_{k+1}}
\end{aligned}
\)
}
\\
\hline

\stepcell{5}
&
\taskcell{Solve symbolic equations to obtain closed-form coefficients}
&
\outcell{%
\(
\begin{aligned}
& V_k - V_{k+1}
=
\lambda_{k+1,k+2}
\inprod{\topa x_{k+1}-\topa x_{k+2}}{x_{k+1}-x_{k+2}}
\\
& \implies
a_k = \frac{2(k+1)^2}{N^2},
\quad
b_k = -\frac{2(k+1)}{N^2},
\quad
c_k = 0
\\
& \boxed{
V_k =
\frac{2(k+1)^2}{N^2}\sqnorm{\topa x_{k+1}}
-\frac{2(k+1)}{N^2}
\inprod{\topa x_{k+1}}{y_0-x_{k+1}}
}
\end{aligned}
\)
}
\\
\hline

\end{tabular}%
}
\endgroup
\caption[Workflow for Lyapunov conversion]{The proposed procedure for converting a PEP-style proof into a Lyapunov-style proof. In the rightmost column, we display the example outcomes from each step in the case of the APPM/OHM algorithm. Detailed explanation of all steps are provided throughout \Cref{section:theory}.}
\label{fig:lyapunov-conversion-workflow}
\end{figure}

%% file: sections/4_existing_lyapunov_analyses.tex
\section{Recovering existing Lyapunov analyses via unified framework}
\label{section:existing-lyapunov}

In this section, we apply our framework to algorithms whose Lyapunov-style proofs exist in the literature,
and show that we can rediscover their Lyapunov analysis from scratch.
This demonstrates the broad applicability of our characterization and procedure of Section~\ref{section:Lyapunov-conversion-procedure}.
The recovered Lyapunov functions are sometimes characterized in forms distinct from what prior work reported, providing original insights for novel extensions.
For each algorithm, we provide an associated \texttt{Jupyter notebook} outlining the execution of our procedure.

\subsection{Proximal monotone inclusion/Fixed-point problem}
\label{section:recovering-proximal-monotone}

We consider the setting provided in Section~\ref{section:setting-proximal-monotone-inclusion}, where we solve monotone inclusion problems 
with respect to maximally (strongly) monotone $\opA\colon \reals^d \rightrightarrows \reals^d$ that has at least one zero.
This includes the case of APPM/OHM, covered in the previous section as a representative example.
We quickly state the final result for APPM/OHM and provide two additional examples: Dual-OHM \citep{YoonKimSuhRyu2024_optimal} and OC-Halpern \citep{ParkRyu2022_exact} algorithms.

\subsubsection{Accelerated proximal point method \& Optimal Halpern method}
\label{section:appm-ohm}

\begin{proposition}
Let $\opA\colon\reals^d\rightrightarrows\reals^d$ be maximally monotone.
Let $x_\star$ satisfy $0\in\opA x_\star$, set $y_\star=x_\star$, and let $x_0 = y_0 \in \reals^d$ be an initial point.
Let $x_k, y_k$ be iterates generated from \eqref{alg:appm}, or equivalently from \eqref{alg:ohm} with $\opT = 2\opJ_\opA - \opI$, and define $V_k$ by \eqref{eqn:appm-ohm-closed-form-lyapunov}.
Then, for $k=1,\dots,N-2$,
\begin{align}
\label{eqn:appm-ogm-consecutive-difference}
    V_k-V_{k+1} = \frac{2(k+1)(k+2)}{N^2} \inprod{\topa x_{k+1}-\topa x_{k+2}}{x_{k+1}-x_{k+2}} .
\end{align}
Consequently, $V_{N-1}\le\dots\le V_1\le 0$, and this implies $\sqnorm{\topa x_N}\le \frac{1}{N^2}\sqnorm{y_0-y_\star}$.
\end{proposition}

\begin{proof}
The identity \eqref{eqn:appm-ogm-consecutive-difference} is verified in \texttt{examples\_lyapunov/appm/appm\_example\_lyap.ipynb} within the provided codebase 
by establishing a symbolic matrix identity, which is 
generated through \texttt{PEPFlow} and then passed to the \texttt{Python} package \texttt{sympy} for exact symbolic verification. 
Once this is established, we have $V_{N-1} = 2\sqnorm{\topa x_N} - \frac{2}{N}\inprod{\topa x_N}{y_0 - x_N} \le 0$, and we can directly verify that \eqref{eqn:appm-ohm-last-step} holds with $\lambda_{N,\star} = \frac{2}{N}$ as
\begin{align*}
    0 & \ge V_{N-1} - \frac{2}{N}\inprod{\topa x_N}{x_N-x_\star} - \sqnorm{\topa x_N-\frac{1}{N}(y_0-x_\star)} \\
    & = 2\sqnorm{\topa x_N} - \frac{2}{N} \inprod{\topa x_N}{y_0 - x_N} - \frac{2}{N}\inprod{\topa x_N}{x_N-x_\star} - \sqnorm{\topa x_N-\frac{1}{N}(y_0-x_\star)} \\
    & = \sqnorm{\topa x_N} - \frac{1}{N^2} \sqnorm{y_0 - x_\star}
\end{align*}
which proves the desired convergence guarantee.
\end{proof}

\paragraph{Remark.}
Although it is not the central emphasis of this work, APPM/OHM is an exact optimal algorithm that matches the complexity lower bound even without any constant factor gap for this problem setup \citep{ParkRyu2022_exact}.
This Lyapunov function for OHM, to the best of our knowledge, first appeared in \citep{Diakonikolas2020_halpern} and similar forms were used extensively in subsequent work on Halpern-type acceleration for minimax optimization \citep{YoonRyu2021_accelerated,LeeKim2021_fast,Tran-DinhLuo2021_halperntype,RyuYin2022_largescale,yoonAcceleratedMinimaxAlgorithms2025}.

\subsubsection{Dual optimal Halpern method}
\label{section:dual-ohm}

Next, consider a nonexpansive operator $\opT\colon \reals^d \to \reals^d$ and the \emph{dual optimal Halpern method} \citep{YoonKimSuhRyu2024_optimal}
\begin{align}
    y_{k+1} = y_k + \frac{N-k-1}{N-k} (\opT y_k - \opT y_{k-1}) 
\label{alg:dual-ohm}
\tag{Dual-OHM}
\end{align}
where $k=0,\dots,N-1$, and $\opT y_{-1} = y_0$. 
The corresponding algorithm for monotone inclusion uses the same update rule with $\opT = 2\opJ_\opA - \opI$, and we let $x_k = \opJ_\opA y_{k-1}$ and $\topa x_k = y_{k-1} - x_k \in \opA x_k$ for $k=1,\dots,N$.
This algorithm was shown to converge with the rate $\sqnorm{y_{N-1} - \opT y_{N-1}} \le \frac{4\sqnorm{y_0 - y_\star}}{N^2}$,
and equivalently $\sqnorm{\topa x_N} \le \frac{\sqnorm{y_0 - y_\star}}{N^2}$, achieving the same exact optimal last-iterate guarantee with APPM/OHM.

The PEP-style proof of this convergence takes the form
\begin{align*}
    \sqnorm{\topa  x_N} - \frac{\sqnorm{y_0 - y_\star}}{N^2} = -\sum_{k=1}^{N-1} \lambda_{k,N} \inprod{\topa x_{k} - \topa x_N}{x_k - x_N} - \lambda_{N,\star} \inprod{\topa x_N}{x_N - x_\star} - \sqnorm{\topa x_N - \frac{1}{N}(y_0 - y_\star)} 
\end{align*}
where $\lambda_{k,N}, \lambda_{N,\star} \ge 0$ are specified later.
Note that we have $\cI = \{(k,N)\,|\,k=1,\dots,N-1\} \cup \{(N,\star)\}$ and $|\cJ| = 1$. 

\paragraph{Step 1.}
As a natural choice, take $\cI_0 = \cJ_0 = \emptyset$ and for $k=1,\dots,N-1$,
\begin{align*}
    \cI_k = \{(j,N)\,|\,j=1,\dots,k\} , \quad \cJ_k = \emptyset .
\end{align*}

\paragraph{Step 2.}
This choice satisfies sufficiency, as $|\cI\setminus \cI_{N-1}| = 1 = |\cJ\setminus \cJ_{N-1}|$.
Numerical computation of $\vV_k$ reveals that it has an almost constant rank of $4$, with the exceptions of $k=1$ and $k=N-1$, where the rank becomes 2.

\paragraph{Step 3.}
We construct the set of special vectors $\cS$ as
\begin{align*}
    \cT = \{\vy_0, \dots, \vy_{N-1}, \vy_\star, \vx_1, \dots, \vx_N, \vg_1 , \dots , \vg_N, \vt_0 , \dots, \vt_{N-1} \} , \quad \cS = \cT \cup \{\vu - \vv \,|\, \vu, \vv \in \cT \}
\end{align*}
where $\vy_{i-1}, \vx_i, \vg_i$ and $\vt_{i-1}$ are respectively PEP context vectors representing the iterates $y_{i-1}, x_i, \topa x_i$ and $\opT y_{i-1}$.

\paragraph{Step 4.}
Numerically, we find that
\[
    \{\vg_N, \vy_0 - \vx_N, \vy_k - \vy_{N-1} , \vt_{k-1} - \vt_{N-1} \} \subset \cS \cap \cR(\vV_k) 
\] 
is a basis of $\cR(\vV_k)$ with majority of Lyapunov coefficients being zero.
Combining \cref{proposition:linear-algebra-of-lyapunov} with the numerical patterns observed, we hypothesize that
\begin{align*}
    V_k = a_k \sqnorm{y_k - \opT y_{N-1}} + \inprod{y_k - \opT y_{N-1}}{\opT y_{k-1} - \opT y_{N-1}} + b_k \sqnorm{\opT y_{k-1} - \opT y_{N-1}} + c \sqnorm{\topa x_N} + d \inprod{\topa x_N}{y_0 - x_N}
\end{align*}
where $c,d \in \reals$ stays invariant over iterations.
Since the last two terms are constants, we subtract them uniformly from all $V_k$ and consider the shifted Lyapunov function
\begin{align}
\label{eqn:dual-ohm-shifted-lyapunov}
    \widetilde{V}_k = a_k \sqnorm{y_k - y_{N-1}} + \inprod{y_k - y_{N-1}}{\opT y_{k-1} - \opT y_{N-1}} + b_k \sqnorm{\opT y_{k-1} - \opT y_{N-1}}  
\end{align}
with same consecutive differences
\begin{align}
\label{eqn:dual-ohm-lyapunov-difference}
    \widetilde{V}_k - \widetilde{V}_{k+1} =
    \lambda_{k+1,N} \inprod{\topa x_{k+1} - \topa x_N}{x_{k+1} - x_N} .
\end{align}

\paragraph{Step 5.}
In the final step, we plug into \eqref{eqn:dual-ohm-lyapunov-difference} the Lyapunov function \eqref{eqn:dual-ohm-shifted-lyapunov} and the formula $\lambda_{k+1,N} = \frac{2}{(N-k)(N-k-1)}$ that interpolates the numerical values. 
Then, solving the resulting equation for $a_k, b_k, a_{k+1}$ and $b_{k+1}$ successfully yields a solution
\begin{align*}{2}
    a_k = -\frac{N-k+1}{2(N-k)}, \quad b_k = -\frac{N-k-1}{2(N-k)} , \quad  a_{k+1} = -\frac{N-k}{2(N-k-1)}, \quad b_{k+1} = -\frac{N-k-2}{2(N-k-1)} .
\end{align*}

Based on the above procedure, we conclude with the following result:
\begin{proposition}
\label{proposition:dual-ohm}
Let $\opA\colon\reals^d\rightrightarrows\reals^d$ be maximally monotone.
Let $x_\star$ satisfy $0\in\opA x_\star$, set $y_\star=x_\star$, and let $x_0 = y_0 \in \reals^d$ be an initial point.
Let $x_k, y_k$ be iterates generated from \eqref{alg:dual-ohm} with $\opT = 2\opJ_\opA - \opI$, and define
\begin{align*}
    \widetilde{V}_k = -\frac{N-k+1}{2(N-k)} \sqnorm{y_k - y_{N-1}} + \inprod{y_k - y_{N-1}}{\opT y_{k-1} - \opT y_{N-1}} - \frac{N-k-1}{2(N-k)} \sqnorm{\opT y_{k-1} - \opT y_{N-1}}
\end{align*}
for $k=0,\dots,N-1$. Then for $k=0,\dots,N-2$,
\begin{align}
\label{eqn:dual-ohm-consecutive-difference-analytic}
    \widetilde{V}_k - \widetilde{V}_{k+1} = \frac{2}{(N-k)(N-k-1)} \inprod{\topa x_{k+1} - \topa x_N}{x_{k+1} - x_N} .
\end{align}
Consequently, $\widetilde{V}_{N-1} \le \dots \le \widetilde{V}_0$, and this implies $\sqnorm{\topa x_N}\le \frac{1}{N^2}\sqnorm{y_0-y_\star}$.
\end{proposition}

\begin{proof}
The identity \eqref{eqn:dual-ohm-consecutive-difference-analytic} is verified in \texttt{examples\_lyapunov/dual-ohm/dual\_ohm\_example\_lyap.ipynb} within the codebase. 
Once this is established, we have
\begin{align*}
    0 = \widetilde{V}_{N-1} \le \widetilde{V}_0 & = -\frac{N+1}{2N} \sqnorm{y_0 - y_{N-1}} + \inprod{y_0 - y_{N-1}}{\opT y_{-1} - \opT y_{N-1}} - \frac{N-1}{2N} \sqnorm{\opT y_{-1} - \opT y_{N-1}} \\
    & = -\frac{N+1}{2N} \sqnorm{y_0 - y_{N-1}} + \inprod{y_0 - y_{N-1}}{y_0 - y_{N-1} + 2\topa x_N} - \frac{N-1}{2N} \sqnorm{y_0 - y_{N-1} + 2\topa x_N} \\
    & = \frac{2}{N} \inprod{\topa x_N}{y_0 - y_{N-1}} - \frac{2(N-1)}{N} \sqnorm{\topa x_N} 
\end{align*}
where we use $\opT y_{-1} = y_0$ and $\opT y_{N-1} = y_{N-1} - 2\topa x_N$.
Then the proof is concluded by using $x_N = y_{N-1} - \topa x_N$, the remaining monotonicity inequality and the square term as
\begin{align*}
    0 & \le \frac{2}{N} \inprod{\topa x_N}{y_0 - y_{N-1}} - \frac{2(N-1)}{N} \sqnorm{\topa x_N} + \frac{2}{N} \inprod{\topa  x_N}{x_N - y_\star} + \sqnorm{\topa x_N - \frac{1}{N}(y_0 - y_\star)} \\
    & = -\sqnorm{\topa x_N} + \frac{1}{N^2} \sqnorm{y_0 - y_\star} .
\end{align*}
\end{proof}

\paragraph{Distinction from prior work.} 
Our Lyapunov analysis is different from that of \citep{YoonKimSuhRyu2024_optimal}, which defined $z_k$ as an auxiliary sequence satisfying $z_0 = 0$ and 
\begin{align*}
    z_{k+1} & = \frac{N-k-1}{N-k} z_k - \frac{1}{N-k} (y_k - \opT y_k) \\
    y_{k+1} & = \opT y_k - z_{k+1} .
\end{align*}
Here, $z_k$ is integrated into the update rule, and this was essential to perform the convergence proof in \citep{YoonKimSuhRyu2024_optimal}.
Without rewriting \ref{alg:dual-ohm} into the above form and proving that the two recurrences are equivalent, their Lyapunov analysis does not seem to work.
On the other hand, we show that our Lyapunov function $\widetilde{V}_k$ is nonincreasing directly from the simpler, original update rule.
This enables a smoother extension to the new algorithm \dualoc~in Section~\ref{section:dual-oc-halpern}.

\subsubsection{Optimal strongly-monotone proximal point method \& Optimal contractive Halpern}
\label{section:os-ppm-and-oc-halpern}

Now we assume that $\opA\colon \reals^d \rightrightarrows \reals^d$ is maximally $\mu$-strongly monotone.
The \emph{optimal strongly-monotone proximal point method (OS-PPM)}, or the equivalent 
\emph{optimal contractive Halpern} algorithm \citep{ParkRyu2022_exact} has the update rule 
\begin{align}
\label{alg:oc-halpern}
\tag{OC-Halpern}
    y_k = \left(1 - \frac{1}{\varphi_k} \right) \opT y_{k-1} + \frac{1}{\varphi_k} y_0
\end{align}
where $\gamma=1+2\mu$, $\opT = \left(1 + \frac{1}{\gamma}\right) \opJ_\opA - \frac{1}{\gamma} \opI$ is the $\gamma^{-1}$-contractive operator, and $\varphi_k = \sum_{i=0}^k \gamma^{2i}$.
This algorithm achieves the exact optimal complexity with respect to the squared residual norm $\sqnorm{\topa x_N}$ for this setting, and reduces to APPM/OHM in the case $\mu=0$.
Let $\Gamma_N=\sum_{i=0}^{N-1}\gamma^i$, $\beta_k=\frac{1}{\varphi_k}$, $x_k = \opJ_\opA y_{k-1}$, and $\topa x_k = y_{k-1}-x_k\in\opA x_k$.
Then we have $\opT y_{k-1}=x_k-\frac{1}{\gamma}\topa x_k$.

OS-PPM/OC-Halpern has the rate 
\(\sqnorm{\topa x_N}\le \Gamma_N^{-2}\sqnorm{y_0-x_\star}\),
whose PEP-style proof takes the form
\begin{align}
\label{eqn:oc-halpern-pep-proof}
    \sqnorm{\topa x_N}-\frac{1}{\Gamma_N^2}\sqnorm{y_0-x_\star}
    &=
    -\sum_{k=1}^{N-1}\lambda_{k,k+1}
    \left(
    \inprod{\topa x_k-\topa x_{k+1}}{x_k-x_{k+1}}
    -\mu\sqnorm{x_k-x_{k+1}}
    \right) \notag\\
    &\quad
    -\lambda_{N,\star}
    \left(
    \inprod{\topa x_N}{x_N-x_\star}
    -\mu\sqnorm{x_N-x_\star}
    \right) \notag\\
    &\quad
    -\frac{\gamma^{-N}}{\Gamma_N^2}
    \sqnorm{\Gamma_N \topa x_N+(x_N-y_0)-\gamma^N(x_N-x_\star)},
\end{align}
where $\lambda_{k,k+1}, \lambda_{N,\star} \ge 0$ are specified later.
As in APPM/OHM, we have $\cI = \{(k,k+1)\,|\,k=1,\dots,N-1\}$ and $|\cJ| = 1$. 

\paragraph{Step 1.}
Take $\cI_0=\cJ_0=\emptyset$ and $\cI_k=\{(j,j+1)\,|\,j=1,\dots,k\}, \cJ_k=\emptyset$ for $k=1,\dots,N-1$.
So $V_0=0$, and
\[
    V_k = -\sum_{j=1}^{k}\lambda_{j,j+1} \left(
    \inprod{\topa x_j-\topa x_{j+1}}{x_j-x_{j+1}}
    -\mu\sqnorm{x_j-x_{j+1}}
    \right) , 
    \quad k=1,\dots,N-1 .
\]

\paragraph{Step 2.}
This choice satisfies sufficiency, with the remaining terms in the proof are terminal strong monotonicity inequality for the pair $(N,\star)$ and the single square term in \eqref{eqn:oc-halpern-pep-proof}.
Numerically, $\vV_0=0$, while $\vV_k$ has rank $2$ for $k=1,\dots,N-1$.

\paragraph{Step 3.}
We construct the set of special vectors as
\[
    \cT=\{\vy_0,\dots,\vy_{N-1},\vx_\star,\vx_1,\dots,\vx_N,\vg_1,\dots,\vg_N\},
    \qquad
    \cS=\cT\cup\{\vu-\vv \,|\, \vu,\vv\in\cT\}. 
\]

\paragraph{Step 4.}
For $k=1,\dots,N-1$, we identify $\{\vg_{k+1},\vy_0-\vx_{k+1}\}\subset \cS\cap\cR(\vV_k)$ as a basis of $\cR(\vV_k)$.
Hence, by \cref{proposition:linear-algebra-of-lyapunov},
\[
    V_k
    =
    a_k\sqnorm{\topa x_{k+1}}
    +b_k\inprod{\topa x_{k+1}}{y_0-x_{k+1}}
    +c_k\sqnorm{y_0-x_{k+1}},
    \qquad k=0,\dots,N-1.
\]

\paragraph{Step 5.}
Substituting the above form into the identity
\begin{align}
\label{eqn:oc-halpern-lyapunov-difference}
    V_k-V_{k+1}
    =
    \lambda_{k+1,k+2}
    \left(
    \inprod{\topa x_{k+1}-\topa x_{k+2}}{x_{k+1}-x_{k+2}}
    -\mu\sqnorm{x_{k+1}-x_{k+2}}
    \right)
\end{align}
for $k=0,\dots,N-2$, and using the analytic form 
$\lambda_{k,k+1} = \frac{2(1+\gamma)}{(1+\gamma^{-N})\Gamma_N^2} \gamma^{-2k}\varphi_k\varphi_{k-1}$ 
and the update relation
\[
    y_0-x_{k+2}
    =
    (1-\beta_{k+1})(y_0-x_{k+1})
    +\frac{1-\beta_{k+1}}{\gamma}\topa x_{k+1}
    +\topa x_{k+2},
\]
we can symbolically verify that the following is a solution:
\begin{align*}
    a_k &= \frac{(1+\gamma)(1-\beta_{k+1})\lambda_{k+1,k+2}}{2\gamma^2} =
    \frac{(1+\gamma)^2}{(1+\gamma^{-N})\Gamma_N^2}
    \gamma^{-2(k+1)}\varphi_k^2  \\
    a_{k+1} & = \frac{(1+\gamma)\lambda_{k+1,k+2}}{2(1-\beta_{k+1})} = \frac{(1+\gamma)^2}{(1+\gamma^{-N})\Gamma_N^2}
    \gamma^{-2(k+2)}\varphi_{k+1}^2 \\
    b_k &= -\beta_{k+1} \lambda_{k+1,k+2} = -\frac{2(1+\gamma)}{(1+\gamma^{-N})\Gamma_N^2}
    \gamma^{-2(k+1)}\varphi_k \\
    b_{k+1} &= -\frac{\beta_{k+1} \lambda_{k+1,k+2}}{1-\beta_{k+1}} = 
    -\frac{2(1+\gamma)}{(1+\gamma^{-N})\Gamma_N^2} \gamma^{-2(k+2)}\varphi_{k+1} \\
    c_k &= - \mu \beta_{k+1} \lambda_{k+1,k+2} = 
    -\frac{2\mu(1+\gamma)}{(1+\gamma^{-N})\Gamma_N^2} \gamma^{-2(k+1)}\varphi_k \\
    c_{k+1} &= -\frac{\mu\beta_{k+1} \lambda_{k+1,k+2}}{1-\beta_{k+1}} =
    -\frac{2\mu(1+\gamma)}{(1+\gamma^{-N})\Gamma_N^2} \gamma^{-2(k+2)}\varphi_{k+1}
\end{align*}
where we use the identities $\beta_{k+1} = \frac{1}{\varphi_{k+1}}$ and $1-\beta_{k+1} = 1-\frac{1}{\varphi_{k+1}} = \frac{\varphi_{k+1}-1}{\varphi_{k+1}} = \frac{\gamma^2 \varphi_k}{\varphi_{k+1}}$, where the last equality follows from
\[
    \varphi_{k+1} - 1 = \sum_{i=1}^{k+1} \gamma^{2i} = \gamma^2 \sum_{i=0}^k \gamma^{2i} = \gamma^2 \varphi_k .
\]
Summarizing, we state the convergence result:
\begin{proposition}
\label{proposition:oc-halpern}
Let $\opA\colon\reals^d\rightrightarrows\reals^d$ be maximally $\mu$-strongly monotone with $\mu\ge 0$, and let $\gamma=1+2\mu$.
Let $x_\star$ satisfy $0 \in \opA x_\star$, set $y_\star=x_\star$, and let $y_0\in\reals^d$ be an initial point.
Let $x_k,y_k$ be generated by \eqref{alg:oc-halpern} with $\opT = \left(1 + \frac{1}{\gamma}\right) \opJ_\opA - \frac{1}{\gamma} \opI$, with $x_k=\opJ_\opA y_{k-1}$ and $\topa x_k=y_{k-1}-x_k$.
Define $V_k$ by
\begin{align}
\label{eqn:oc-halpern-lyapunov}
    V_k
    =
    \frac{1+\gamma}{(1+\gamma^{-N})\Gamma_N^2}
    \gamma^{-2(k+1)}\varphi_k
    \left(
    (1+\gamma)\varphi_k\sqnorm{\topa x_{k+1}}
    -2\inprod{\topa x_{k+1}}{y_0-x_{k+1}}
    -2\mu\sqnorm{y_0-x_{k+1}}
    \right).
\end{align}
$k=0,\dots,N-1$.
Then $V_0=0$, and for $k=0,\dots,N-2$, \eqref{eqn:oc-halpern-lyapunov-difference} holds with
\[
    \lambda_{k+1,k+2}
    =
    \frac{2(1+\gamma)}{(1+\gamma^{-N})\Gamma_N^2}
    \gamma^{-2(k+1)}\varphi_{k+1}\varphi_k .
\]
Consequently, $V_{N-1}\le\dots\le V_0 = 0$, and this implies
\[
    \sqnorm{\topa x_N}\le \frac{1}{\Gamma_N^2}\sqnorm{y_0-x_\star}.
\]
\end{proposition}

\begin{proof}
The identity \eqref{eqn:oc-halpern-lyapunov-difference} is verified in \texttt{examples\_lyapunov/oc\_halpern/oc\_halpern\_example\_lyap.ipynb} within the codebase.
This immediately establishes $V_{N-1} \le V_0 = 0$.
Now, directly plugging in $\varphi_{N-1} = \frac{1-\gamma^{2N}}{1-\gamma^2}$, $\Gamma_N = \frac{1-\gamma^N}{1-\gamma}$ and $\mu = \frac{\gamma - 1}{2}$ into 
the expression \eqref{eqn:oc-halpern-lyapunov} with $k=N-1$, we obtain
\begin{align*}
    V_{N-1} & = \frac{(1-\gamma)\gamma^{-N}}{1-\gamma^N} \left( \frac{1-\gamma^{2N}}{1-\gamma} \sqnorm{\topa x_N} - 2\inprod{\topa x_{N}}{y_0-x_{N}} - 2\mu\sqnorm{y_0-x_{N}} \right) \\
    & = (1 + \gamma^{-N}) \sqnorm{\topa x_N} - \frac{2(1-\gamma)\gamma^{-2N}}{1-\gamma^N} \inprod{\topa x_{N}}{y_0-x_{N}} + \frac{(1-\gamma)^2 \gamma^{-2N}}{1-\gamma^N} \sqnorm{y_0-x_{N}} .
\end{align*}
Then, let $\lambda_{N,\star} = \frac{2}{\Gamma_N}$ and consider the inequality
\begin{align*}
    0 & \ge V_{N-1} - \lambda_{N,\star} \left(
    \inprod{\topa x_N}{x_N-x_\star}
    -\mu\sqnorm{x_N-x_\star}
    \right) - \frac{\gamma^{-N}}{\Gamma_N^2}
    \sqnorm{\Gamma_N \topa x_N+(x_N-y_0)-\gamma^N(x_N-x_\star)} .
\end{align*}
The proof is completed by showing that the right hand side simplifies to $\sqnorm{\topa x_N} - \frac{1}{\Gamma_N^2} \sqnorm{y_0 - x_\star}$.
We provide this verification via \texttt{sympy} within \texttt{examples\_lyapunov/oc\_halpern/oc\_halpern\_example\_lyap.ipynb}, and omit the tedious details here.
\end{proof}

\paragraph{Remark.} The original work \citep{ParkRyu2022_exact} that introduced \eqref{alg:oc-halpern} also provided a Lyapunov-style proof.
Our Lyapunov function in \cref{proposition:oc-halpern} is equivalent to their proof under rescaling and translation.

\subsection{Smooth convex minimization}

We consider the setting provided in Section~\ref{section:setting-smooth-convex-minimization}, where we solve minimization problems 
with respect to convex and $L$-smooth $f\colon \reals^d \to \reals$ with a minimizer $x_\star$. 

\subsubsection{Optimized gradient method}
\label{section:ogm}

One of the equivalent forms of the \emph{optimized gradient method} \citep[OGM2]{KimFessler2016_optimized} is given as
\begin{equation*}
\label{alg:ogm2}
\tag{OGM2}
    \begin{aligned}
        y_{k+1} &= x_k - \frac{1}{L} \nabla f(x_k) \\
        z_{k+1} &= z_k - \frac{2\theta_k}{L} \nabla f(x_k) \\
        x_{k+1} &= \pr{ 1 - \frac{1}{\theta_{k+1}} } y_{k+1} + \frac{1}{\theta_{k+1}} z_{k+1},
    \end{aligned}
\end{equation*}
for $k=0,\dots,N-1$, initialized with $z_0=x_0$, where $\theta_k$ is the sequence defined by the recurrence
\begin{equation}
\label{eqn:ogm-theta-recurrence}
    \begin{aligned}
        \theta_{-1} &= 0 \\
        \theta_{k} &= \frac{1 + \sqrt{4\theta_{k-1}^2+1}}{2}, \qquad k=0,\dots, N-1 \\
        \theta_{N} &= \frac{1 + \sqrt{8\theta_{N-1}^2+1}}{2} .
    \end{aligned}
\end{equation}
\ref{alg:ogm2} converges with the rate
\[
    f(x_N) - f(x_\star) \le \frac{L}{2\theta_N^2}\sqnorm{x_0 - x_\star},
\]
and the PEP-style proof of this convergence takes the form
\begin{align*}
    f(x_N) - f(x_\star) - \frac{L}{2\theta_N^2} \sqnorm{x_0 - x_\star}
    & =
    \sum_{i=1}^N \lambda_{i-1,i}
    \left(
        f(x_i)-f(x_{i-1})
        +\inprod{\nabla f(x_i)}{x_{i-1}-x_i}
        +\frac{1}{2L}\sqnorm{\nabla f(x_{i-1})-\nabla f(x_i)}
    \right) \\
    & \quad
    + \sum_{i=0}^N \lambda_{\star,i}
    \left(
        f(x_i)-f(x_\star)
        +\inprod{\nabla f(x_i)}{x_\star-x_i}
        +\frac{1}{2L}\sqnorm{\nabla f(x_i)}
    \right) \\
    & \quad
    - \frac{L}{2\theta_N^2}
    \sqnorm{z_N - x_\star - \frac{\theta_N}{L}\nabla f(x_N)},
\end{align*}
where $\lambda_{i-1,i}, \lambda_{\star,i} \ge 0$ are specified later.
We have $\cI = \{(i,i+1)\,|\,i=0,\dots,N-1\} \cup \{(\star,i)\,|\,i=0,\dots,N\}$ and $|\cJ| = 1$.

\paragraph{Step 1.}
Take $\cI_0 = \{(\star,0)\}$ and for $k=1,\dots,N$,
\[
    \cI_k
    =
    \{(i,i+1)\,|\,i=0,\dots,k-1\}
    \cup \{(\star,i)\,|\,i=0,\dots,k\}.
\]
We set $\cJ_k=\emptyset$ for $k=0,\dots,N$.

\paragraph{Step 2.}
This choice is sufficient because $\cI_N=\cI$ and $|\cJ\setminus\cJ_N| = 1$.
Numerically, the quadratic part of the partial sum $V_k$ has rank $3$ for $k=1,\dots,N-1$, with $V_0$ and $V_N$ being the only exceptions with rank $2$.
The function-value support is also sparse: $V_k$ only involves $f(x_k)$ and $f(x_\star)$.

\paragraph{Step 3.}
We construct the set of special vectors using the atomic vectors
\begin{align*}
    \cT
    =
    \{\vx_0,\dots,\vx_N,\vx_\star,\vg_0,\dots,\vg_N,\vz_1,\dots,\vz_N\},
    \qquad
    \cS = \cT \cup \{\vu-\vv \,|\, \vu,\vv\in\cT\},
\end{align*}
where $\vx_i$, $\vg_i$, and $\vz_i$ represent $x_i$, $\nabla f(x_i)$, and $z_i$, respectively.

\paragraph{Step 4.}
For $k=0,\dots,N-1$, we find the basis vectors
\[
    \left\{ \vg_k, \vx_0-\vx_\star, \vz_{k+1}-\vx_\star \right\} \subset \cS\cap\cR(\vV_k).
\]
Together with the sparse function-value support identified in Step~2, and using numerical observation, we hypothesize:
\begin{align*}
    V_k = a_k \bigl(f(x_k)-f(x_\star)\bigr)
    +b_k \sqnorm{\nabla f(x_k)}
    +c \sqnorm{z_{k+1}-x_\star}
    -c \sqnorm{x_0-x_\star}
\end{align*}
where $c\in \reals$ is constant over $k=0,\dots,N-1$. 
Adding the constant term $c\sqnorm{x_0 - x_\star}$ to each $V_k$, we consider the translated Lyapunov function
\begin{align*}
    \widetilde V_k
    =
    V_k + c\sqnorm{x_0-x_\star}
    =
    a_k\bigl(f(x_k)-f(x_\star)\bigr)
    +b_k\sqnorm{\nabla f(x_k)}
    +c\sqnorm{z_{k+1}-x_\star}
\end{align*}
for $k=0,\dots,N-1$.

\paragraph{Step 5.}
We use the analytic form of the PEP certificates:
\begin{align*}
    \lambda_{i,i+1}
    =
    \frac{2\theta_i^2}{\theta_N^2}
    \,\, (i=0,\dots,N-1),
    \qquad
    \lambda_{\star,i}
    =
    \begin{cases}
        \lambda_{i,i+1}-\lambda_{i-1,i} = \frac{2\theta_i}{\theta_N^2}, & i=0,\dots,N-1,\\[2mm]
        1-\lambda_{N-1,N} = \frac{1}{\theta_N}, & i=N .
    \end{cases}
\end{align*}
Here we use the recurrence $\theta_i^2 - \theta_i = \theta_{i-1}^2$ for $i=1,\dots,N-1$ and $\theta_N^2-\theta_N=2\theta_{N-1}^2$, 
together with the convention $\lambda_{-1,0}=0$.
Then, we solve the equation
\begin{align}
\label{eqn:ogm2-interior-identity}
\begin{aligned}
    \widetilde V_k
    =
    \widetilde V_{k+1}
    & - \frac{2\theta_k^2}{\theta_N^2}
    \left(
        f(x_{k+1})-f(x_k)
        +\inprod{\nabla f(x_{k+1})}{x_k-x_{k+1}}
        +\frac{1}{2L}\sqnorm{\nabla f(x_k)-\nabla f(x_{k+1})}
    \right) \\
    & - \frac{2\theta_{k+1}}{\theta_N^2}
    \left(
        f(x_{k+1})-f(x_\star)
        +\inprod{\nabla f(x_{k+1})}{x_\star-x_{k+1}}
        +\frac{1}{2L}\sqnorm{\nabla f(x_{k+1})}
    \right)
\end{aligned}
\end{align}
for $a_k,b_k,c,a_{k+1},b_{k+1}$.
Enforcing the consistency of formulas in $k$, this uniquely yields
\[
    a_k=\frac{2\theta_k^2}{\theta_N^2}, \quad 
    b_k=-\frac{\theta_k^2}{L\theta_N^2}, \quad
    c=\frac{L}{2\theta_N^2}.
\]
Finally, we obtain the analytic expression for $\widetilde V_N$ by simplifying
\begin{align}
\label{eqn:ogm2-final-identity}
\begin{aligned}
    \widetilde V_N
    =
    \widetilde V_{N-1} 
    &+ \frac{2\theta_{N-1}^2}{\theta_N^2}
    \left(
        f(x_N)-f(x_{N-1})
        +\inprod{\nabla f(x_N)}{x_{N-1}-x_N}
        +\frac{1}{2L}\sqnorm{\nabla f(x_{N-1})-\nabla f(x_N)}
    \right) \\
    &+\frac{1}{\theta_N}
    \left(
        f(x_N)-f(x_\star)
        +\inprod{\nabla f(x_N)}{x_\star-x_N}
        +\frac{1}{2L}\sqnorm{\nabla f(x_N)}
    \right)
\end{aligned}
\end{align}
using $\theta_N^2-\theta_N=2\theta_{N-1}^2$. This reveals
\[
    \widetilde V_N
    =
    f(x_N)-f(x_\star)
    +\frac{L}{2\theta_N^2}
    \sqnorm{z_N-x_\star-\frac{\theta_N}{L}\nabla f(x_N)} .
\]
Putting these pieces together, we obtain the following final result:
\begin{proposition}
\label{proposition:ogm2}
Let $f\colon\reals^d\to\reals$ be convex and $L$-smooth, let $x_\star$ be a minimizer of $f$, and let $x_k,y_k,z_k$ be generated by \eqref{alg:ogm2} with $z_0=x_0$.
Define
\begin{align*}
    \widetilde V_k
    =
    \frac{2\theta_k^2}{\theta_N^2}\bigl(f(x_k)-f(x_\star)\bigr)
    -\frac{\theta_k^2}{L\theta_N^2}\sqnorm{\nabla f(x_k)}
    +\frac{L}{2\theta_N^2}\sqnorm{z_{k+1}-x_\star}
\end{align*}
for $k=0,\dots,N-1$ and 
\begin{align*}
    \widetilde V_N
    =
    f(x_N)-f(x_\star)
    +\frac{L}{2\theta_N^2}
    \sqnorm{z_N-x_\star-\frac{\theta_N}{L}\nabla f(x_N)} .
\end{align*}
Then, with
\[
    \lambda_{i,i+1}
    =
    \frac{2\theta_i^2}{\theta_N^2},
    \qquad
    \lambda_{\star,i}
    =
    \begin{cases}
        \frac{2\theta_i}{\theta_N^2}, & i=0,\dots,N-1,\\[2mm]
        \frac{1}{\theta_N}, & i=N ,
    \end{cases}
\]
the identities \eqref{eqn:ogm2-interior-identity} for $k=0,\dots,N-2$ and \eqref{eqn:ogm2-final-identity} hold.
Furthermore,
\begin{align*}
    \widetilde V_0-\frac{L}{2\theta_N^2}\sqnorm{x_0-x_\star}
    =
    \frac{2}{\theta_N^2}
    \left(
        f(x_0)-f(x_\star)
        +\inprod{\nabla f(x_0)}{x_\star-x_0}
        +\frac{1}{2L}\sqnorm{\nabla f(x_0)}
    \right) 
    \le 0 .
\end{align*}
This establishes
\begin{align*}
    f(x_N) - f(x_\star) \le \widetilde V_N \le \widetilde V_0 \le \frac{L}{2\theta_N^2}\sqnorm{x_0-x_\star} .
\end{align*}
\end{proposition}

\begin{proof}
The symbolic verification of the identities \eqref{eqn:ogm2-interior-identity} and \eqref{eqn:ogm2-final-identity} is provided in 
\path{examples_lyapunov/ogm/ogm_example_lyap.ipynb} within the codebase.
Finally, using $\theta_0 = 1$ and $z_1 = z_0 - \frac{2\theta_0}{L} \nabla f(x_0) = x_0 - \frac{2}{L} \nabla f(x_0)$, we have
\begin{align*}
    \widetilde V_0 -\frac{L}{2\theta_N^2}\sqnorm{x_0-x_\star} & = 
    \frac{2}{\theta_N^2} \left( f(x_0) - f(x_\star) - \frac{1}{2L} \sqnorm{\nabla f(x_0)} + \frac{L}{4} \sqnorm{x_0 - \frac{2}{L} \nabla f(x_0) - x_\star} - \frac{L}{4} \sqnorm{x_0 - x_\star} \right) \\
    & = \frac{2}{\theta_N^2} \left( f(x_0)-f(x_\star)
    +\inprod{\nabla f(x_0)}{x_\star-x_0}
    +\frac{1}{2L}\sqnorm{\nabla f(x_0)} \right) 
\end{align*}
as desired, and the proof is complete.
\end{proof}

\paragraph{Remark.}
The original proof of OGM in \citep{KimFessler2016_optimized} was first presented in the PEP style, and was later reformulated into Lyapunov-style proofs equivalent to \cref{proposition:ogm2} in \citep{dAspremontScieurTaylor2021_acceleration,ParkParkRyu2023_factorsqrt2}.

\subsubsection{OGM without the prior knowledge of \texorpdfstring{$z_k$}{zk}}
\label{section:ogm-without-z}

An alternative definition for \ref{alg:ogm2} \citep[OGM1]{KimFessler2016_optimized} does not use the auxiliary iterate $z_k$, and is given by the update rule
\begin{align}
\label{alg:ogm1}
\tag{OGM1}
\begin{aligned}
    y_{k+1} & = x_k - \frac{1}{L} \nabla f(x_k) \\
    x_{k+1} & = y_{k+1} + \frac{\theta_k - 1}{\theta_{k+1}} (y_{k+1} - y_k) + \frac{\theta_k}{\theta_{k+1}} (y_{k+1} - x_k ) .
\end{aligned}
\end{align}
In this case, assuming that we do not have the knowledge of the equivalent reformulation \ref{alg:ogm2} that seems essential for the Lyapunov analysis to work, can we still recover the same Lyapunov function?
We show that this is possible, using our basis completion procedure of Section~\ref{section:basis-completion}.

Because \ref{alg:ogm1} is equivalent to \ref{alg:ogm2}, it produces the same sequence of iterates $x_k$; hence, up to Step~2, there is no difference with Section~\ref{section:ogm}.
The first distinction arises in the following step:

\paragraph{Step 3.} We construct the set of special vectors as
\begin{align*}
    \cT
    =
    \{\vx_0,\dots,\vx_N,\vx_\star,\vg_0,\dots,\vg_N\} ,
    \qquad
    \cS = \cT \cup \{\vu-\vv \,|\, \vu,\vv\in\cT\} ,
\end{align*}
where we lack the vectors $\vz_k$.

\paragraph{Step 4.} Then, unlike in Section~\ref{section:ogm}, we only find
\[
    \{ \vg_k, \vx_0 - \vx_\star \} \subset \cS \cap \cR(\vV_k)
\]
and fail to complete the basis (recall that $\mathrm{rank}(\vV_k) = 3$).
However, using the basis completion procedure, we numerically observe that there exist constants $c,d\in \reals$ and $v_{k} \in d(x_0 - x_\star) + \mathrm{span}\{\nabla f(x_0), \dots, \nabla f(x_{k})\}$ such that
\begin{equation*}
    V_k = a_k (f(x_k) - f(x_{\star})) + b_k \norm{ \nabla f(x_k) }^2 + c \sqnorm{x_0 - x_\star} + \norm{ v_{k} }^2
\end{equation*}
and $v_{k+1} = v_{k} - \eta_{k} \nabla f(x_{k+1})$ for some $\eta_{k} > 0$, for $k=0,\dots,N-2$.
This is seen from the numerical values of $v_k$ found for $k=0,\dots,N-1$ in the case $N=5$:
\[
\begin{array}{c|rrrrrrrr}
 & x_0 & x_\star & \nabla f(x_0) & \nabla f(x_1) & \nabla f(x_2) & \nabla f(x_3) & \nabla f(x_4) & \nabla f(x_5) \\
\hline
v_0 & 0.136 & -0.136 & -0.273 & 0.0 & 0.0 & 0.0 & 0.0 & 0.0 \\
v_1 & 0.136 & -0.136 & -0.273 & -0.441 & 0.0 & 0.0 & 0.0 & 0.0 \\
v_2 & 0.136 & -0.136 & -0.273 & -0.441 & -0.598 & 0.0 & 0.0 & 0.0 \\
v_3 & 0.136 & -0.136 & -0.273 & -0.441 & -0.598 & -0.750 & 0.0 & 0.0 \\
v_4 & 0.136 & -0.136 & -0.273 & -0.441 & -0.598 & -0.750 & -0.898 & 0.0
\end{array}
\]
where the column labels indicate the PEP basis directions associated with the coefficients of $v_k$ within each column.
Applying translation by the constant term $-c\sqnorm{x_0 - x_\star}$, and re-running the previous procedure, we observe that
\begin{align}
\label{eqn:ogm-without-z-lyapunov-skeleton}
    \widetilde V_k = a_k (f(x_k) - f(x_{\star})) + b_k \norm{ \nabla f(x_k) }^2 + \norm{ v_{k} }^2 
\end{align}
for $k=0,\dots,N-1$, and $v_{k+1} = v_k - \eta_k \nabla f(x_{k+1})$ also holds for all $k=0,\dots,N-2$, for some $\eta_k \in \reals$.
By construction, this sequence must satisfy the identity \eqref{eqn:ogm2-interior-identity}.

\paragraph{Step 5.}
Using \eqref{eqn:ogm2-interior-identity} with \eqref{eqn:ogm-without-z-lyapunov-skeleton}, we obtain the identity
\begin{align}
\label{eqn:ogm-without-z-system-of-equations}
\begin{aligned}
    & a_k (f(x_k) - f(x_{\star})) + b_k \norm{ \nabla f(x_k) }^2 + \norm{ v_{k} }^2 - a_{k+1} (f(x_{k+1}) - f(x_{\star})) - b_{k+1} \norm{ \nabla f(x_{k+1}) }^2 - \norm{ v_{k+1} }^2 \\
    & = - \frac{2\theta_k^2}{\theta_N^2}
    \left(
        f(x_{k+1})-f(x_k)
        +\inprod{\nabla f(x_{k+1})}{x_k-x_{k+1}}
        +\frac{1}{2L}\sqnorm{\nabla f(x_k)-\nabla f(x_{k+1})}
    \right) \\
    & \phantom{=}\, - \frac{2\theta_{k+1}}{\theta_N^2}
    \left(
        f(x_{k+1})-f(x_\star)
        +\inprod{\nabla f(x_{k+1})}{x_\star-x_{k+1}}
        +\frac{1}{2L}\sqnorm{\nabla f(x_{k+1})}
    \right)
\end{aligned}
\end{align}
for $k=0,\dots,N-2$, and we have the additional constraint $v_{k+1} = v_k - \eta_k \nabla f(x_{k+1})$.
Comparing the coefficients of $f(x_k)$, $f(x_{k+1})$, and $f(x_\star)$ in both sides gives
\[
    a_k - \frac{2\theta_k^2}{\theta_N^2} = 0,
    \qquad
    -a_{k+1} + \frac{2\theta_k^2}{\theta_N^2} + \frac{2\theta_{k+1}}{\theta_N^2} = 0,
    \qquad
    -a_k + a_{k+1} - \frac{2\theta_{k+1}}{\theta_N^2} = 0 ,
\]
which, using $\theta_{k+1}^2-\theta_{k+1}=\theta_k^2$, implies
\[
    a_k = \frac{2\theta_k^2}{\theta_N^2},
    \qquad
    a_{k+1}=\frac{2(\theta_k^2+\theta_{k+1})}{\theta_N^2}
    = \frac{2\theta_{k+1}^2}{\theta_N^2}.
\]
Next, comparing the coefficient of $\|\nabla f(x_k)\|^2$ in both sides, we immediately obtain $b_k = -\frac{\theta_k^2}{L\theta_N^2}$.

It remains to determine $\eta_k$ and $v_k$.
For consistency of the formula for $b_k$ across iterations, we set $b_{k+1}=-\frac{\theta_{k+1}^2}{L\theta_N^2}$.
Then, substituting the expressions for $a_k, b_k, a_{k+1}$ and $b_{k+1}$, together with the identity
\[
    \sqnorm{v_k}-\sqnorm{v_{k+1}}
    =
    2\eta_k\inprod{v_k}{\nabla f(x_{k+1})}
    -\eta_k^2\sqnorm{\nabla f(x_{k+1})}
\]
into \eqref{eqn:ogm-without-z-system-of-equations} and gathering all terms paired with $\nabla f(x_{k+1})$, we obtain
\begin{align}
\label{eqn:ogm-without-z-vector-equation}
    2\eta_k v_k
    + \frac{2\theta_k^2}{\theta_N^2}(x_k-x_{k+1})
    + \frac{2\theta_{k+1}}{\theta_N^2}(x_\star-x_{k+1})
    - \frac{2\theta_k^2}{L\theta_N^2}\nabla f(x_k)
    + \left(\frac{2\theta_{k+1}^2}{L\theta_N^2}-\eta_k^2\right)\nabla f(x_{k+1})
    = 0.
\end{align}
Since the completed vector $v_k$ belongs to $d(x_0-x_\star)+\mathrm{span}\{\nabla f(x_0),\dots,\nabla f(x_k)\}$, the coefficient of the new direction $\nabla f(x_{k+1})$ must vanish.
Therefore, we obtain $\eta_k = \sqrt{\frac{2}{L}}\frac{\theta_{k+1}}{\theta_N}$.
Then, \eqref{eqn:ogm-without-z-vector-equation} reduces to
\[
    2\eta_k v_k
    + \frac{2\theta_k^2}{\theta_N^2}(x_k-x_{k+1})
    + \frac{2\theta_{k+1}}{\theta_N^2}(x_\star-x_{k+1})
    - \frac{2\theta_k^2}{L\theta_N^2}\nabla f(x_k)
    =0,
\]
and hence
\[
    v_k
    =
    \sqrt{\frac{L}{2\theta_N^2}}
    \left(
        \theta_{k+1}x_{k+1}
        -(\theta_{k+1}-1)x_k
        + \frac{\theta_{k+1}-1}{L}\nabla f(x_k)
        -x_\star
    \right) .
\]
In the notation of Step~4, this identifies $d = \sqrt{\frac{L}{2\theta_N^2}}$.
We can then rewrite \eqref{eqn:ogm-without-z-lyapunov-skeleton} to conclude that
\begin{align*}
    \widetilde V_k = \frac{2\theta_k^2}{\theta_N^2} (f(x_k) - f(x_\star)) - \frac{\theta_k^2}{L\theta_N^2} \sqnorm{\nabla f(x_k)}
    + \frac{L}{2\theta_N^2} \sqnorm{
        \theta_{k+1}x_{k+1}
        -(\theta_{k+1}-1)x_k
        + \frac{\theta_{k+1}-1}{L}\nabla f(x_k)
        -x_\star
    }
\end{align*}
defined for $k=0,\dots,N-1$ is a Lyapunov function.
In fact, rearranging \eqref{alg:ogm2} yields
\[
    z_{k+1}
    =
    \theta_{k+1}x_{k+1}
    -(\theta_{k+1}-1)x_k
    +\frac{\theta_{k+1}-1}{L}\nabla f(x_k),
\]
so this coincides with the Lyapunov function found in \cref{proposition:ogm2}.
We include the Steps~1--4 above and the verification of $\widetilde V_k - \widetilde V_{k+1} \ge 0$ for $k=0,\dots,N-2$ using only the \ref{alg:ogm1} update rule in \path{examples_lyapunov/ogm/ogm_example_lyap_without_z.ipynb}.
One can then proceed in the same way for the last-step modification (identifying $\widetilde V_N$) and deriving the final convergence guarantee, which we omit here.

\subsubsection{OGM-G}
\label{section:ogm-g}

\textit{OGM-G} \citep{KimFessler2021_optimizing} is an algorithm designed for efficiently reducing the gradient norm of smooth convex functions with respect to the initial function value condition.
Define $\theta_k$ as in \eqref{eqn:ogm-theta-recurrence} (so it is the same sequence as in OGM), and define the reversed sequence $\psi_i = \theta_{N-i}$ for $i=0,\dots,N+1$. 
In particular, $\psi_{N+1} = \theta_{-1} = 0$.
The originally proposed form of OGM-G from \citep{KimFessler2021_optimizing} is given by
\begin{align*}
    y_{k+1} & = x_k - \frac{1}{L} \nabla f(x_k) \\
    x_{k+1} & = y_{k+1} + \frac{(\psi_k - 1)(2\psi_{k+1} - 1)}{\psi_k (2\psi_{k} - 1)} (y_{k+1} - y_k) + \frac{2\psi_{k+1} - 1}{2\psi_k - 1} (y_{k+1} - x_k) 
\end{align*}
for $k=0,\dots,N-1$.
This is similar to how \ref{alg:ogm1} was defined.
We rewrite it into a form similar to \ref{alg:ogm2}, which is a straightforward reformulation that results in the unique form: set $z_0 = x_0, z_1 = z_0 - \frac{\psi_0+1}{2L}\nabla f(x_0)$, and for $k=1,\dots,N$, define
\begin{equation}
\label{alg:ogm-g}
\tag{OGM-G}
    \begin{aligned}
        y_k & = x_{k-1} - \frac{1}{L}\nabla f(x_{k-1}),\\
        x_k & = \left( \frac{\psi_{k+1}}{\psi_k} \right)^4 y_k
        + \left( 1 - \left(\frac{\psi_{k+1}}{\psi_k}\right)^4 \right) z_k,\\
        z_{k+1} & = z_k - \frac{\psi_k}{L}\nabla f(x_k) .
    \end{aligned}
\end{equation}
Note that $\psi_{N+1} = 0$, so the second line implies $x_N = z_N$.
OGM-G exhibits the rate
\[
    \sqnorm{\nabla f(x_N)}
    \le \frac{2L}{\psi_0^2}\bigl(f(x_0)-f(x_\star)\bigr) .
\]
The PEP-style proof of this convergence takes the form
\begin{align*}
    \frac{1}{L}\sqnorm{\nabla f(x_N)}
    -\frac{2}{\psi_0^2}\bigl(f(x_0)-f(x_\star)\bigr) &
    =
    \sum_{i=0}^{N-1} \lambda_{i,i+1} \left( f(x_{i+1}) - f(x_i) + \inprod{\nabla f(x_{i+1})}{x_i - x_{i+1}} + \frac{1}{2L}\sqnorm{\nabla f(x_{i}) - \nabla f(x_{i+1})} \right) \\
    & \phantom{=} + \sum_{i=0}^{N-1} \lambda_{N,i} \left( f(x_i) - f(x_N) + \inprod{\nabla f(x_i)}{x_N - x_i} + \frac{1}{2L}\sqnorm{\nabla f(x_N) - \nabla f(x_i)} \right) \\
    & \phantom{=} + \lambda_{N,\star} \left( f(x_\star) - f(x_N) + \frac{1}{2L}\sqnorm{\nabla f(x_N)} \right) 
\end{align*}
where we specify $\lambda_{i,i+1} , \lambda_{N,j}$ for $i=0,\dots,N-1$ and $j=0,\dots,N-1,\star$ later.
In particular we have, $\cI = \{(i,i+1)\,|\,i=0,\dots,N-1\} \cup \{(N,i)\,|\,i=0,\dots,N-1,\star\}$, and because there are no norm square terms in the proof, $\cJ = \emptyset$.

\paragraph{Step 1.}
We set
\[
    \cI_0=\{(N,0)\},
    \qquad
    \cI_k
    =
    \{(i,i+1)\mid i=0,\dots,k-1\}
    \cup
    \{(N,i)\mid i=0,\dots,\min\{k,N-1\}\}
\]
for $k=1,\dots,N$, and set $\cJ_k=\emptyset$ for all $k$.
In particular, $\cI \setminus \cI_N = \{(N,\star)\}$.

\paragraph{Step 2.}
This choice is clearly sufficient.
We numerically analyze the matrices $\vV_k$ encoding the quadratic part of $V_k$, and observe that 
$\mathrm{rank}(\vV_k) = 4$ for $k=1,\dots,N-2$, with $\mathrm{rank}(V_0) = 3 = \mathrm{rank}(V_{N-1})$ and
$\mathrm{rank}(V_N)=1$ being the exceptions. 
The function-value support $\vf \odot \vb_k$ is also sparse: it only involves $f(x_0)$, $f(x_k)$, and $f(x_N)$.

\paragraph{Step 3.}
We construct the set of special vectors as
\begin{align*}
    \cT
    =
    \{\vx_0,\dots,\vx_N,\vx_\star,\vg_0,\dots,\vg_N,
      \vy_1,\dots,\vy_N,\vz_1,\dots,\vz_{N+1}\},
    \qquad
    \cS = \cT \cup \{\vu-\vv \,|\, \vu,\vv\in\cT\},
\end{align*}
where $\vx_i$, $\vg_i$, $\vy_i$, and $\vz_i$ represent
$x_i$, $\nabla f(x_i)$, $y_i$, and $z_i$, respectively.

\paragraph{Step 4.}
For $i=1,\dots,N-2$, we find that
\[
    \{\vg_i,\vg_N,\vx_{i+1}-\vy_{i+1},\vz_{i+1}-\vz_{N+1}\}
    \subset \cS\cap\cR(\vV_i)
\]
is a basis of $\cR(\vV_i)$. 
At $i=0$ and $i=N-1$, these four vectors become linearly dependent and the $\mathrm{rank}(\vV_i)$ drops to 3. 

With this information and numerical observations, we hypothesize a Lyapunov function of the form
\begin{align}
\label{eqn:ogm-g-lyapunov-skeleton}
    V_i
    &=
    a_i\bigl(f(x_i)-f(x_N)\bigr)
    +b_i \bigl(f(x_0)-f(x_N)\bigr)
    +c_i \sqnorm{\nabla f(x_i)}
    +d_i \sqnorm{\nabla f(x_N)}
    +e_i \inprod{x_{i+1}-y_{i+1}}{z_{i+1}-z_{N+1}}
\end{align}
for $i=0,\dots,N-1$.
We handle the final term $V_N$ separately below.

\paragraph{Step 5.}
We first specify the analytic forms of $\lambda$ as $\lambda_{i,i+1} = \frac{1}{\psi_{i+1}^2}$ for $i=0,\dots,N-1$, 
and
\[
    \lambda_{N,0}
    =
    \frac{1}{\psi_1^2}-\frac{2}{\psi_0^2},
    \qquad
    \lambda_{N,i}
    =
    \frac{1}{\psi_{i+1}^2}-\frac{1}{\psi_i^2}
    \quad (i=1,\dots,N-1),
    \qquad
    \lambda_{N,\star}
    =
    \frac{2}{\psi_0^2}.
\]
We plug \eqref{eqn:ogm-g-lyapunov-skeleton} into
\begin{align}
\label{eqn:ogm-g-consecutive-difference}
\begin{aligned}
    0
    =
    V_{i+1}-V_i
    & -\frac{1}{\psi_{i+1}^2} \left( f(x_{i+1}) - f(x_i) + \inprod{\nabla f(x_{i+1})}{x_i - x_{i+1}} + \frac{1}{2L}\sqnorm{\nabla f(x_i) - \nabla f(x_{i+1})}\right) \\
    & -\left(\frac{1}{\psi_{i+2}^2}-\frac{1}{\psi_{i+1}^2}\right)
    \left( f(x_{i+1}) - f(x_N) + \inprod{\nabla f(x_{i+1})}{x_N - x_{i+1}} + \frac{1}{2L}\sqnorm{\nabla f(x_{i+1}) - \nabla f(x_N)} \right) 
\end{aligned}    
\end{align}
for $i=0,\dots,N-2$ to solve for the coefficients. 
Using the update rule \eqref{alg:ogm-g} and the $\psi$-recurrence $\psi_0^2-\psi_0=2\psi_1^2$ and $\psi_i^2-\psi_i=\psi_{i+1}^2$ for $i=1,\dots,N$, we obtain
\begin{align*}
    a_i = \frac{1}{\psi_{i+1}^2} , \qquad b_i \equiv -\frac{2}{\psi_0^2} , \qquad c_i = -\frac{1}{2L\psi_{i+1}^2} , \qquad 
    d_i = \frac{\psi_0^2-2\psi_{i+1}^2} {2L\psi_0^2\psi_{i+1}^2} , \qquad e_i = \frac{L}{\psi_{i+1}^2(2\psi_{i+1}-1)}
\end{align*}
for $i=0,\dots,N-1$.
In particular, $b_i$ is constant across $i$, so we can remove it from \eqref{eqn:ogm-g-lyapunov-skeleton} and consider the translated Lyapunov function
\begin{align}
\label{eqn:ogm-g-lyapunov}
    \widetilde V_k = \frac{1}{\psi_{k+1}^2}\bigl(f(x_k)-f(x_N)\bigr)
    -\frac{1}{2L\psi_{k+1}^2}\sqnorm{\nabla f(x_k)}
    +\frac{\psi_0^2-2\psi_{k+1}^2}
    {2L\psi_0^2\psi_{k+1}^2}\sqnorm{\nabla f(x_N)} 
    + \frac{L}{\psi_{k+1}^2(2\psi_{k+1}-1)}
    \inprod{x_{k+1}-y_{k+1}}{z_{k+1}-z_{N+1}} 
\end{align}
for $k=0,\dots,N-1$.
Finally, we can separately verify that
\begin{align}
\label{eqn:ogm-g-consecutive-difference-last-step}
    \widetilde{V}_N = \widetilde V_{N-1} + \frac{1}{\psi_N^2} \left( f(x_{N}) - f(x_{N-1}) + \inprod{\nabla f(x_N)}{x_{N-1} - x_N} + \frac{1}{2L}\sqnorm{\nabla f(x_{N-1}) - \nabla f(x_N)} \right)
\end{align}
simplifies to $\frac{\psi_0^2-1}{L\psi_0^2}\sqnorm{\nabla f(x_N)}$.
Summarizing the result, we obtain the following: 

\begin{proposition}
\label{proposition:ogm-g}
Let $f\colon\reals^d\to\reals$ be convex and $L$-smooth, let $x_\star$ be a
minimizer of $f$, and let $x_k,y_k,z_k$ be generated by \eqref{alg:ogm-g}.
Define $\widetilde{V}_k$ as in \eqref{eqn:ogm-g-lyapunov}, and let $\widetilde V_N = \frac{\psi_0^2-1}{L\psi_0^2}\sqnorm{\nabla f(x_N)}$.
Then, \eqref{eqn:ogm-g-consecutive-difference} for $k=0,\dots,N-2$ and \eqref{eqn:ogm-g-consecutive-difference-last-step} hold, 
and
\begin{align*}
    \widetilde{V}_0 & = \frac{2}{\psi_0^2} (f(x_0) - f(x_N)) + \left( \frac{1}{\psi_1^2} - \frac{2}{\psi_0^2} \right)  \left( f(x_0) - f(x_N) + \inprod{\nabla f(x_0)}{x_N - x_0} + \frac{1}{2L}\sqnorm{\nabla f(x_0) - \nabla f(x_N)} \right) \\
    & \le \frac{2}{\psi_0^2} (f(x_0) - f(x_N)) .
\end{align*}
This establishes $\widetilde{V}_N \le \cdots \le \widetilde{V}_0 \le \frac{2}{\psi_0^2} (f(x_0) - f(x_N))$, and consequently,
\begin{align*}
    & \frac{1}{L}\sqnorm{\nabla f(x_N)} 
    -\frac{2}{\psi_0^2}\bigl(f(x_0)-f(x_\star)\bigr) = \widetilde{V}_N - \frac{2}{\psi_0^2} (f(x_0) - f(x_N))
    + \frac{2}{\psi_0^2} \left( f(x_\star) - f(x_N) + \frac{1}{2L}\sqnorm{\nabla f(x_N)} \right) 
    \le 0 .
\end{align*}
\end{proposition}

\begin{proof}
The symbolic verification of the displayed identities is provided in \path{examples_lyapunov/ogm_g/ogm_g_example_lyap.ipynb} within the codebase. 
This immediately implies $\widetilde{V}_N \le \cdots \le \widetilde{V}_0$ and $\widetilde V_0 \le \frac{2}{\psi_0^2} (f(x_0) - f(x_N))$,
and then we easily see that the following identity holds:
\begin{align*}
    \frac{1}{L}\sqnorm{\nabla f(x_N)} 
    -\frac{2}{\psi_0^2}\bigl(f(x_0)-f(x_\star)\bigr) = \underbrace{\widetilde{V}_N - \frac{2}{\psi_0^2} (f(x_0) - f(x_N))}_{\le 0} + \underbrace{\frac{2}{\psi_0^2} \left( f(x_\star) - f(x_N) + \frac{1}{2L}\sqnorm{\nabla f(x_N)} \right)}_{\le 0}
\end{align*}
which shows that the right hand side is at most $0$, completing the proof.
\end{proof}

\paragraph{Remark.}
The original proof of OGM-G in \citep{KimFessler2021_optimizing} was
presented in the PEP style. It was later reformulated into a simpler
potential-function argument by \citep{DiakonikolasWang2022_potential}, although
their analysis did not explicitly establish a single nonincreasing quantity. 
To the best of our knowledge, the first Lyapunov-style proof in that sense, 
essentially equivalent to our \cref{proposition:ogm-g}, was given in \citep{LeeParkRyu2021_geometric}.

\subsection{Smooth minimax optimization}

Consider the setting of \cref{section:minimax-optimization}, where we solve minimax optimization problem with respect to convex-concave and $L$-smooth $\phi\colon \reals^{d_u} \times \reals^{d_v} \to \reals$ with a minimax optimum (saddle point) $x_\star = (u_\star, v_\star)$.
This can be recast into the monotone inclusion problem with respect to monotone and $L$-Lipschitz operator by considering the saddle gradient operator $\opA$ defined by \eqref{eqn:saddle-operator}.

\subsubsection{Fast extragradient}
\label{section:feg}

The \textit{fast extragradient (FEG)} algorithm\footnote{In \citep{LeeKim2021_fast}, the FEG algorithm was designed for structured nonconvex-nonconcave minimax optimization problems whose saddle gradient operator is $L$-Lipschitz and $\rho$-\textit{comonotone}, i.e., $\inprod{\opA x - \opA y}{x-y} \ge \rho \sqnorm{\opA x - \opA y}$ for all $x,y\in\reals^d$ with some $\rho \le 0$.
Here, we consider only its special case $\rho=0$, where $\opA$ is monotone.} with $\rho=0$ \citep{LeeKim2021_fast} generates, for $k=0,\dots,N-1$,
\begin{align}
\begin{aligned}
    x_{k+\frac12} &= x_k+\frac{1}{k+1}(x_0-x_k)-\frac{k}{k+1}\frac{1}{L}\opA x_k \\
    x_{k+1} &= x_k+\frac{1}{k+1}(x_0-x_k)-\frac{1}{L}\opA x_{k+\frac12} .
\end{aligned}
\label{alg:feg}
\tag{FEG}
\end{align}
The first half-step gives $x_{\frac12}=x_0$.
FEG has the rate $\sqnorm{\opA x_N} \le \frac{4L^2}{N^2}\sqnorm{x_0-x_\star}$, 
whose PEP-style proof takes the form
\begin{align}
\begin{aligned}
    \sqnorm{\opA x_N}-\frac{4L^2}{N^2}\sqnorm{x_0-x_\star} 
    &= -\sum_{j=1}^{N-1}\lambda^\cM_{j,j+1}\cM(x_j,x_{j+1})
    -\lambda^\cM_{N,\star}\cM(x_N,x_\star)
    -\sum_{j=0}^{N-1}\lambda^\cL_{j+\frac12,j+1}\cL(x_{j+\frac12},x_{j+1}) \\
    &\quad
    -\sqnorm{\frac{2L}{N}(x_0-x_\star)-\opA x_N} 
\end{aligned}
\label{eq:feg_full_pep_certificate}
\end{align}
where we write
\[
    \cM(x,y) := \inprod{x-y}{\opA x-\opA y} \ge 0,
    \qquad
    \cL(x,y) := L^2\sqnorm{x-y} - \sqnorm{\opA x - \opA y} \ge 0
\]
for the monotonicity and Lipschitz inequalities. 
The nonnegative weights $\lambda^\cM_{j,j+1}$, $\lambda^\cM_{N,\star}$, and $\lambda^\cL_{j+\frac12,j+1}$ are specified later.
In this case, the active monotonicity and Lipschitz inequalities and sum of squares terms are
\[
    \begin{aligned}
    \cI^\cM = \{(j,j+1)\,|\,j=1,\dots,N-1\}\cup\{(N,\star)\},    
    \qquad
    \cI^\cL = \left\{\left(j+\frac12,j+1\right)\,\middle|\,j=0,\dots,N-1 \right\},
    \qquad
    |\cJ|=1.
    \end{aligned}
\]

\paragraph{Step 1.}
Take $\cI_0^\cM=\cI_0^\cL=\cJ_0=\emptyset$ and 
\[
    \cI_k^\cM = \{(j,j+1)\,|\,j=1,\dots,k-1\},
    \qquad
    \cI_k^\cL = \left\{\left(j+\frac12, j+1\right)\,\middle|\,j=0,\dots,k-1 \right\},
    \qquad
    \cJ_k=\emptyset .
\]
With this choice, we have $|\cI^\cM \setminus \cI^\cM_N| = |\cJ \setminus \cJ_N| = 1$ and $|\cI^\cL \setminus \cI^\cL_k| = 0$.

\paragraph{Step 2.}
This choice satisfies sufficiency, with the remaining terms in the proof being only the monotonicity inequality $\cM(x_N,x_\star)$ and the single square term in \eqref{eq:feg_full_pep_certificate}.
Numerically, $\vV_0=0$, while $\vV_k$ has rank $2$ for $k=1,\dots,N$.

\paragraph{Step 3.}
We construct the set of special vectors as
\[
    \cT=\{\vx_j,\vx_{j+\frac12},\vg_j,\vg_{j+\frac12} \,|\, j=0,\dots,N-1\}
    \cup\{\vx_N,\vx_\star,\vg_N\},
    \qquad
    \cS=\cT\cup\{\vu-\vv \,|\, \vu,\vv\in\cT\},
\]
where $\vx_t$ and $\vg_t$ respectively represent $x_t$ and $\opA x_t$ for $t=0,\frac{1}{2},1,\dots,N$.

\paragraph{Step 4.}
For $k=1,\dots,N$, we identify $\{\vg_k,\vx_0-\vx_k\}\subset \cS\cap\cR(\vV_k)$ as a basis of $\cR(\vV_k)$.
Hence, by \cref{proposition:linear-algebra-of-lyapunov}, we search for a Lyapunov function of the form
\begin{align*}
    V_k
    =
    a_k\sqnorm{\opA x_k}
    +b_k\inprod{\opA x_k}{x_0-x_k}
    +c_k\sqnorm{x_0-x_k}
\end{align*}
for $k=1,\dots,N$, while we quickly observe numerically that $c_k \equiv 0$.

\paragraph{Step 5.}
We substitute the Lyapunov skeleton form into the identity
\begin{align*}
    V_k-V_{k+1}
    =
    \lambda^\cM_{k,k+1}\cM(x_k,x_{k+1})
    +\lambda^\cL_{k+\frac12,k+1}\cL(x_{k+\frac12},x_{k+1})
\end{align*}
for $k=1,\dots,N-1$ and use the analytic forms $\lambda^\cM_{k,k+1} = \frac{4Lk(k+1)}{N^2}$ and $\lambda^\cL_{k+\frac12,k+1} = \frac{2(k+1)^2}{N^2}$, together with the update rule of \ref{alg:feg} and the numerical observation $c_k=c_{k+1}=0$.
Then, solving the identity for the coefficients, we obtain
\[
    a_k=\frac{2k^2}{N^2},
    \qquad
    b_k=-\frac{4Lk}{N^2},
    \qquad
    a_{k+1}=\frac{2(k+1)^2}{N^2},
    \qquad
    b_{k+1}=-\frac{4L(k+1)}{N^2},
\]
i.e., the Lyapunov function is, for $k=0,\dots,N$,
\begin{align}
\label{eqn:feg-lyapunov}
    V_k=\frac{2k^2}{N^2}\sqnorm{\opA x_k}
    -\frac{4Lk}{N^2}\inprod{\opA x_k}{x_0-x_k} .
\end{align}
Summarizing the above outcome, we obtain the following result.

\begin{proposition}
\label{proposition:feg}
Let $\opA\colon\reals^d\to\reals^d$ be monotone and $L$-Lipschitz, and let $\opA x_\star=0$.
Let $x_k$ be generated by \eqref{alg:feg}.
Define $V_k$ by \eqref{eqn:feg-lyapunov} for $k=0,\dots,N$. Then for $k=0,\dots,N-1$,
\begin{align}
\label{eqn:feg-step-identity}
    V_k-V_{k+1}
    =
    \frac{4Lk(k+1)}{N^2}\cM(x_k,x_{k+1})
    +\frac{2(k+1)^2}{N^2}\cL(x_{k+\frac12},x_{k+1}).
\end{align}
Consequently, $V_N\le \cdots \le V_0=0$, and this implies
\[
    \sqnorm{\opA x_N}\le\frac{4L^2}{N^2}\sqnorm{x_0-x_\star}.
\]
\end{proposition}

\begin{proof}
The identity \eqref{eqn:feg-step-identity} is verified symbolically in the associated FEG notebook \path{fast_extragradient_example_lyap.ipynb} within the codebase.
Since $\cM\ge0$ and $\cL\ge0$ for any pair of iterates, this identity immediately establishes $V_N\le \cdots \le V_0=0$.
Once this is established, it is straightforward to verify that
\begin{align*}
    \sqnorm{\opA x_N}
    -\frac{4L^2}{N^2}\sqnorm{x_0-x_\star}
    =
    V_N-\frac{4L}{N}\cM(x_N,x_\star)
    -\sqnorm{\opA x_N-\frac{2L}{N}(x_0-x_\star)} \le V_N \le V_0 = 0 ,
\end{align*}
which compeletes the proof.
\end{proof}

\subsubsection{Dual fast extragradient}
\label{section:dual-feg}

We next consider the H-dual algorithm of \ref{alg:feg}, the \textit{dual fast extragradient (Dual-FEG)} \cite{YoonKimSuhRyu2024_optimal},
which has a terminal horizon $N$ and the update rule
\begin{align}
\begin{aligned}
    x_{k+\frac12} &= x_k- \frac{1}{L} z_k - \frac{1}{L} \opA x_k, \\
    x_{k+1} &= x_{k+\frac12}-\frac{N-k-1}{N-k} \frac{1}{L} \left(\opA x_{k+\frac12} - \opA x_k\right), \\
    z_{k+1} &= \frac{N-k-1}{N-k}z_k - \frac{1}{N-k}\opA x_{k+\frac12}    
\end{aligned}
\label{alg:dual-feg}
\tag{Dual-FEG}
\end{align}
for $k=0,\dots,N-1$, where $z_0 = 0$.
Dual-FEG has the rate $\sqnorm{\opA x_N}\le \frac{4L^2\sqnorm{x_0-x_\star}}{N^2}$,
whose PEP-style proof takes the form
\begin{align}
\begin{aligned}
    \sqnorm{\opA x_N}-\frac{4L^2}{N^2}\sqnorm{x_0-x_\star}
    &= -\sum_{j=0}^{N-2}\lambda^\cM_{j+\frac12,N}\cM(x_{j+\frac12},x_N)
    -\lambda^\cM_{N,\star}\cM(x_N,x_\star) \\
    &\quad
    -\sum_{j=0}^{N-2}\lambda^\cL_{j,j+\frac12}\cL(x_j,x_{j+\frac12})
    -\lambda^\cL_{N-1,N}\cL(x_{N-1},x_N) 
    -\sqnorm{\opA x_N-\frac{2L}{N}(x_0-x_\star)} .
\end{aligned}
\label{eq:dual_feg_full_pep_certificate}
\end{align}
The nonnegative weights $\lambda^\cM_{j+\frac12,N}$, $\lambda^\cM_{N,\star}$, $\lambda^\cL_{j,j+\frac12}$, and $\lambda^\cL_{N-1,N}$ are specified later.
In this case, the active monotonicity and Lipschitz inequalities are
\[
    \begin{aligned}
        \cI^\cM = \left\{ \left(j+\frac12,N \right)\,\middle|\,j=0,\dots,N-2 \right\}\cup\{(N,\star) \} ,
        \quad
        \cI^\cL = \left\{\left( j, j+\frac12 \right)\,\middle|\,j=0,\dots,N-2 \right\}\cup\{(N-1,N) \} ,
        \quad
        |\cJ| = 1 .
    \end{aligned}
\]

\paragraph{Step 1.}
Take $\cI_0^\cM=\cI_0^\cL=\cJ_0=\emptyset$ and, for $k=1,\dots,N-1$,
\[
    \cI_k^\cM = \left\{ \left(j+\frac12,N \right)\,\middle|\, j=0,\dots,k-1 \right\},
    \qquad
    \cI_k^\cL = \left\{\left( j, j+\frac12 \right)\,\middle|\, j=0,\dots,k-1 \right\},
    \qquad
    \cJ_k=\emptyset .
\]
At the terminal iteration, take
\[
    \cI_N^\cM = \cI_{N-1}^\cM\cup\{(N,\star)\},
    \qquad
    \cI_N^\cL=\cI_{N-1}^\cL\cup\{(N-1,N)\},
    \qquad
    \cJ_N=\emptyset .
\]

\paragraph{Step 2.}
This choice satisfies sufficiency, with the remaining term in the proof being the single square term in \eqref{eq:dual_feg_full_pep_certificate}.
Numerically, $\vV_0=0$, while $\vV_k$ has rank $4$ for $k=1,\dots,N-1$ and $\vV_N$ has rank $2$.

\paragraph{Step 3.}
We construct the set of special vectors as
\[
    \cT=\{\vx_j,\vx_{j+\frac12},\vg_j,\vg_{j+\frac12} \,|\, j=0,\dots,N-1\}
    \cup\{\vx_N,\vx_\star,\vg_N,\vz_1,\dots,\vz_N\},
    \qquad
    \cS=\cT\cup\{\vu-\vv \,|\, \vu,\vv\in\cT\},
\]
where $\vx_t$ and $\vg_t$ respectively represent $x_t$ and $\opA x_t$ for $t=0,\frac{1}{2},\dots,N$ and $\vz_j$ represents $z_j$ for $j=1,\dots,N$.

\paragraph{Step 4.}
We identify $\{\vx_0-\vx_N,\vx_k-\vx_N,\vz_k,\vg_N\}\subset\cS\cap\cR(\vV_k)$ as a basis of $\cR(\vV_k)$ for $k=1,\dots,N-1$, and $\{\vx_0-\vx_\star,\vg_N\}\subset\cS\cap\cR(\vV_N)$ as a basis of $\cR(\vV_N)$.
Hence, using \cref{proposition:linear-algebra-of-lyapunov} and excluding the coefficients that are numerically observed to be zero, we search for a Lyapunov function of the form
\begin{align*}
    V_k
    =
    a_k\inprod{x_0-x_N}{\opA x_N}
    +b_k\inprod{x_k-x_N}{z_k}
    +c_k\inprod{x_k-x_N}{\opA x_N}
    +d_k\sqnorm{z_k}
    +e_k\inprod{z_k}{\opA x_N}
\end{align*}
for $k=1,\dots,N-1$.

\paragraph{Step 5.}
Substituting the above form into the identity
\begin{align*}
    V_k-V_{k+1}
    =
    \lambda^\cM_{k+\frac12,N}\cM(x_{k+\frac12},x_N)
    +\lambda^\cL_{k,k+\frac12}\cL(x_k,x_{k+\frac12})
\end{align*}
for $k=0,\dots,N-2$, and using the analytic forms $\lambda^\cM_{k+\frac12,N} = \frac{4L}{(N-k-1)(N-k)}$ and $\lambda^\cL_{k,k+\frac12} = \frac{2}{(N-k)^2}$, together with the update relation~\eqref{alg:dual-feg}, we can symbolically verify the solution
\[
    a_k=a_{k+1}=-\frac{4L}{N}, 
    \quad
    b_k=c_k=\frac{4L}{N-k}, 
    \quad
    b_{k+1}=c_{k+1}=\frac{4L}{N-k-1},
    \quad
    d_k=d_{k+1}=-2,
    \quad
    e_k=e_{k+1}=-4.
\]
Thus the Lyapunov function is
\begin{align}
\label{eqn:dual-feg-lyapunov}
    V_k
    =
    -\frac{4L}{N}\inprod{x_0-x_N}{\opA x_N}
    +\frac{4L}{N-k}\inprod{x_k-x_N}{z_k}
    +\frac{4L}{N-k}\inprod{x_k-x_N}{\opA x_N}
    -2\sqnorm{z_k}
    -4\inprod{z_k}{\opA x_N}
\end{align}
for $k=0,\dots,N-1$.
Finally, we can separately verify that
\[
    V_N = V_{N-1} - \frac{4L}{N}\cM(x_N,x_\star) - 2\cL(x_{N-1},x_N)
\]
simplifies to 
\begin{align}
\label{eqn:dual-feg-terminal-lyapunov}
    V_N
    =
    2\sqnorm{\opA x_N}
    -\frac{4L}{N}\inprod{\opA x_N}{x_0-x_\star}.
\end{align}
Summarizing the above, we obtain the following result.

\begin{proposition}
\label{proposition:dual-feg}
Let $\opA\colon\reals^d\to\reals^d$ be monotone and $L$-Lipschitz, and let $\opA x_\star=0$.
Let $x_k,z_k$ be generated by \eqref{alg:dual-feg}.
Define $V_k$ by \eqref{eqn:dual-feg-lyapunov} for $k=0,\dots,N-1$ and by \eqref{eqn:dual-feg-terminal-lyapunov} for $k=N$.
Then $V_0=0$, and we have 
\begin{align}
\label{eqn:dual-feg-step-identity}
    V_k-V_{k+1}
    =
    \frac{4L}{(N-k-1)(N-k)}\cM(x_{k+\frac12},x_N)
    +\frac{2}{(N-k)^2}\cL(x_k,x_{k+\frac12}).
\end{align}
for $k=0,\dots,N-2$ and 
\begin{align}
\label{eqn:dual-feg-terminal-identity}
    V_{N-1}-V_N
    =
    \frac{4L}{N}\cM(x_N,x_\star)
    +2\cL(x_{N-1},x_N).
\end{align}
Consequently, $V_N\le\cdots\le V_0=0$, and this implies
\[
    \sqnorm{\opA x_N}\le\frac{4L^2}{N^2}\sqnorm{x_0-x_\star}.
\]
\end{proposition}

\begin{proof}
The identities \eqref{eqn:dual-feg-step-identity} and \eqref{eqn:dual-feg-terminal-identity} are verified symbolically in the associated notebook \path{dual_feg_example_lyap.ipynb} within the codebase.
Since $\cM\ge0$ and $\cL\ge0$, these identities immediately establish $V_N\le\cdots\le V_0=0$, which implies
\begin{align*}
    \sqnorm{\opA x_N}
    -\frac{4L^2}{N^2}\sqnorm{x_0-x_\star}
    =
    V_N-\sqnorm{\opA x_N-\frac{2L}{N}(x_0-x_\star)} 
    \le V_N \le V_0 = 0 
\end{align*}
which is the desired convergence guarantee.
\end{proof}

%% file: sections/5_new_lyapunov_analyses.tex
\section{Discovering new analyses and algorithms via unified framework}
\label{section:new-lyapunov}

In this section, we utilize our framework to discover previously unknown Lyapunov-style proofs:
\begin{enumerate}
    \item We consider the smooth convex minimization problem and present the Lyapunov analysis corresponding to the tight convergence proof for the \emph{gradient descent (GD)} algorithm from \cite{DroriTeboulle2014_performance}.
    
    \item We consider the composite minimization problem with a nonsmooth proximable term and a Bregman-relatively-smooth term, and present a tight Lyapunov-style proof for the \emph{Bregman proximal gradient method (BPGM)} yielding the same rate as \cite{ZhouLiangShen2019_simple, GutmanPena2023_perturbed}.
    
    \item We consider the maximally monotone inclusion problem and present the Lyapunov-style proof corresponding to the tight convergence proof of the \emph{proximal point method (PPM)} from \cite{guTightSublinearConvergence2020}.
\end{enumerate}
The aforementioned results are new in the sense that the tight rates were previously known but had not been formally stated anywhere in the form of Lyapunov analysis, to the best of our knowledge.
Our final result, on the other hand, discovers a novel algorithm \emph{dual optimal contractive Halpern (Dual-OC-Halpern)}, equipped with a Lyapunov-style proof showing that it is another exact optimal algorithm for the equivalent settings of maximally $\mu$-strongly monotone inclusion (with proximal operations) and $\gamma$-contractive fixed-point problems.

\subsection{Gradient descent with constant step size}
\label{section:gd}

Let $f\colon \reals^d \to \reals$ be a convex, $L$-smooth function with a minimizer $x_\star$.
Consider the standard \emph{gradient descent} with step size $\frac{1}{L}$:
\begin{equation} \label{eq:gd} \tag{GD}
    x_{k+1} = x_k - \frac{1}{L} \nabla f(x_k).
\end{equation}
The tight PEP-style convergence proof for GD takes the form
\begin{equation*}
    \begin{aligned}
        f(x_N) - f(x_\star) - \frac{L}{4N+2} \|x_0 - x_\star\|^2
        &= \sum _{i=1}^{N} \lambda_{i-1,i} \left( f(x_{i})-f(x_{i-1}) + \langle \nabla f(x_{i}) , x_{i-1} - x_{i} \rangle + \frac{1}{2 L} \| \nabla f(x_{i-1}) - \nabla f(x_i) \|^2 \right)  \\ 
        &\phantom{=}+\sum _{i=0}^{N} \lambda_{\star,i} \left( f(x_i)-f(x_{\star}) + \langle \nabla f(x_{i}) , x_{\star} - x_{i} \rangle + \frac{1}{2 L} \| \nabla f(x_i) \|^2 \right) \\ &\phantom{=}
        - \begin{bmatrix} x_0 & x_\star & \nabla f(x_0) & \cdots & \nabla f(x_N) \end{bmatrix} \vS_N \begin{bmatrix} x_0 & x_\star & \nabla f(x_0) & \cdots & \nabla f(x_N) \end{bmatrix}^{\intercal} 
    \end{aligned}
\end{equation*}
where we specify $\lambda_{i-1,i}, \lambda_{\star,i} \ge 0$ and $\vS_N \succeq 0$ later.
To run our procedure, we need a decomposition of the positive semidefinite matrix $\vS_N \in \mathbb{S}^{N+3}_+$ into rank-$1$ matrices.
This can be done numerically by employing techniques from numerical linear algebra.
We propose one concrete strategy here, which we call the \emph{reverse-basis LDL decomposition}, where we find a unit $(N+3) \times (N+3)$ lower triangular matrix $\vL_N$ 
and a diagonal matrix $\vD$ such that
\begin{align*}
    & \begin{bmatrix} x_0 & x_\star & \nabla f(x_0) & \cdots & \nabla f(x_N) \end{bmatrix} \vS_N \begin{bmatrix} x_0 & x_\star & \nabla f(x_0) & \cdots & \nabla f(x_N) \end{bmatrix}^{\intercal} \\
    & = \begin{bmatrix} \nabla f(x_N) & \cdots & \nabla f(x_0) & x_\star & x_0 \end{bmatrix} 
    \vL_N \vD \vL_N^\transpose 
    \begin{bmatrix} \nabla f(x_N) & \cdots & \nabla f(x_0) & x_\star & x_0 \end{bmatrix}^{\intercal} . 
\end{align*}
Note that in the right hand side, the order of the PEP basis vectors is reversed to $\begin{bmatrix} \nabla f(x_N) & \cdots & \nabla f(x_0) & x_\star & x_0 \end{bmatrix}$.
This design lets the $i$-th term in the sum of squares decomposition depend only on the algorithm's history up to $i$-th iteration.
To make this point precise, we need to first note the two characteristics of $\vL_N$ and $\vD$ in the case of GD:
\begin{itemize}
    \item[(i)] $\vS_N$ has rank $N+1$, and consequently, $\vD$ has $N+1$ nonzero diagonal entries and the remaining two are zeros, precisely the last two entries.
    Hence, we can write $\vD = \mathrm{diag}(\alpha_N, \dots, \alpha_0, 0, 0)$.
    \item[(ii)] The first $N+1$ columns of $\vL_N$ are of the form $\bmat{* & \cdots & * & c & -c}^\transpose$, i.e., the sum of last two entries are zero.
\end{itemize}
Hence, writing $\vL_N = \begin{bmatrix} \mid & \cdots & \mid & \mid & \mid  \\ \boldsymbol{\ell}_N & \cdots & \boldsymbol{\ell}_0 & \boldsymbol{\ell}_{-1} & \boldsymbol{\ell}_{-2} \\ \mid & \cdots & \mid & \mid & \mid \end{bmatrix}$, we can write
\begin{align*}
    & \begin{bmatrix} x_0 & x_\star & \nabla f(x_0) & \cdots & \nabla f(x_N) \end{bmatrix} \vS_N \begin{bmatrix} x_0 & x_\star & \nabla f(x_0) & \cdots & \nabla f(x_N) \end{bmatrix}^{\intercal} = \sum_{i=0}^N \alpha_i \sqnorm{s_i}
\end{align*}
where $s_i = \begin{bmatrix} \nabla f(x_N) & \cdots & \nabla f(x_0) & x_\star & x_0 \end{bmatrix} \boldsymbol{\ell}_{i} \in \text{span}\set{x_0-x_\star, \nabla f(x_0), \dots, \nabla f(x_{i})}$ since $\vL_N$ is lower triangular.
We highlight again that the process of obtaining $\alpha_i$ and $s_i$ as above is purely numerical, and thus the steps 1--4 below are also performed numerically.

\paragraph{Step 1.} We have $\cI = \{(i-1,i)\,|\,i=1,\dots,N\} \cup \{(\star,i)\,|\,i=0,\dots,N\}$ and $\cJ = \{0,\dots,N\}$.
Let us naturally define 
\[
    \cI_k = \{(i-1,i)\,|\,i=1,\dots,k\} \cup \{(\star,i)\,|\,i=0,\dots,k\} , \quad \cJ_k = \{0,\dots,k\}
\]
for $k=0,\dots,N$. To clarify, in particular, $\cI_0 = \{(\star,0)\}$ and $\cJ_0 = \{0\}$.

\paragraph{Step 2.} We have $\cI_N = \cI$ and $\cJ_N = \cJ$, so sufficiency is immediate.
The rank computation of $\vV_k$ reveals the almost constant rank of $3$, except when $k=0$ and $k=N$.
Finally, the function value term $\vf \odot \vb_k$ associated with the partial sum given by Step~1 is sparse, identifying the simple term $a_k (f(x_k) - f(x_\star))$.

\paragraph{Step 3.} We construct the set of special vectors as 
\begin{align*}
    \cT = \{\vx_0, \dots, \vx_{N}, \vx_\star, \vg_0, \dots, \vg_N \} , \quad \cS = \cT \cup \{\vu - \vv: \vu, \vv \in \cT \} .
\end{align*}

\paragraph{Step 4.} We numerically find that $V_0 = 0$ and
\[
    \{\vx_0 - \vx_\star , \vg_k , \vx_{k+1} - \vx_\star \} \subset \cS \cap \cR(\vV_k)
\]
for $k=1,\dots,N-1$.
Using \cref{proposition:linear-algebra-of-lyapunov} and numerical observations, we can hypothesize a Lyapunov function of the form
\begin{equation*}
    V_k = a_k ( f(x_k) - f(x_\star) ) +  b_k \norm{ \nabla f(x_k) }^2 + c_k \norm{ x_{k+1} - x_\star }^2 - d \norm{ x_{0} - x_\star }^2
\end{equation*}
where $d \in \reals$ stays constant over iterations.
At the last iteration $N$, we have a simpler form
\begin{align*}
    V_N = f(x_N) - f(x_\star) - d\sqnorm{x_0 - x_\star} .
\end{align*}
This suggests us to instead use a translated Lyapunov function 
\[
    \widetilde{V}_k = V_k + d\sqnorm{x_0 - x_\star} = a_k ( f(x_k) - f(x_\star) ) +  b_k \norm{ \nabla f(x_k) }^2 + c_k \norm{ x_{k+1} - x_\star }^2 .
\]

\paragraph{Step 5.} We have $\widetilde{V}_N = f(x_N) - f(x_\star)$ and $\widetilde{V}_0 = d\sqnorm{x_0 - x_\star}$, 
and $\widetilde{V}_N \le \widetilde{V}_0$ should be directly equivalent to the final bound $f(x_N) - f(x_\star) \le \frac{L\sqnorm{x_0 - x_\star}}{4N+2}$ because $\cI \setminus \cI_N = \emptyset = \cJ \setminus \cJ_N$.
Hence, we must have $d=\frac{L}{4N+2}$.

The reverse-basis LDL decomposition with $N=5$ with $L=1$ yields the numerical $\vL_N$ with the following column vectors:
\[
\begin{array}{c|rrrrrrrr}
 & \boldsymbol{\ell}_5 & \boldsymbol{\ell}_4 & \boldsymbol{\ell}_3 & \boldsymbol{\ell}_2 & \boldsymbol{\ell}_1 & \boldsymbol{\ell}_0 & \boldsymbol{\ell}_{-1} & \boldsymbol{\ell}_{-2} \\
\hline
\nabla f(x_5) & 1.0   & 0.0   & 0.0   & 0.0   & 0.0 & 0.0   & 0.0   & 0.0 \\
\nabla f(x_4) & 0.167 & 1.0   & 0.0   & 0.0   & 0.0 & 0.0   & 0.0   & 0.0 \\
\nabla f(x_3) & 0.167 & 0.143 & 1.0   & 0.0   & 0.0 & 0.0   & 0.0   & 0.0 \\
\nabla f(x_2) & 0.167 & 0.143 & 0.125 & 1.0   & 0.0 & 0.0   & 0.0   & 0.0 \\
\nabla f(x_1) & 0.167 & 0.143 & 0.125 & 0.111 & 1.0 & 0.0   & 0.0   & 0.0 \\
\nabla f(x_0) & 0.167 & 0.143 & 0.125 & 0.111 & 0.1 & 1.0   & 0.0   & 0.0 \\
x_*           & 0.167 & 0.143 & 0.125 & 0.111 & 0.1 & 0.091 & 1.0   & 0.0 \\
x_0           & -0.167 & -0.143 & -0.125 & -0.111 & -0.1 & -0.091 & 0.604 & 1.0
\end{array}
\]
where the row labels indicate the PEP basis directions associated with the coefficients within each row.
From this numerical data and the fact that $x_k - x_\star = x_0 - x_\star - \frac{1}{L}\sum_{j=0}^{k-1} \nabla f(x_j)$ for each $k=1,\dots,N$, we can easily infer that $s_i = \begin{bmatrix} \nabla f(x_N) & \cdots & \nabla f(x_0) & x_\star & x_0 \end{bmatrix} \boldsymbol{\ell}_{i} = \frac{1}{L}\nabla f(x_i) - \frac{1}{2N+1-i} (x_i - x_\star)$ for $i=0,\dots,N$.
Now we use this information, together with the analytic form of proof certificates
\begin{align*}
\lambda_{i-1,i} = \frac{i}{2N + 1 - i}, \quad i = 1, \dots, N,
\qquad
\lambda_{\star,i} =
\begin{cases}
\lambda_{0,1} & i = 0 \\
\lambda_{i,i+1} - \lambda_{i-1,i} & i = 1, \dots, N-1 \\
1 - \lambda_{i-1,i} & i = N
\end{cases}
\end{align*}
which can be obtained either via numerical inspection or from prior work \citep{DroriTeboulle2014_performance},
to solve the equation
\begin{align}
    \label{eqn:gd-lyapunov-consecutive-difference}
    \begin{aligned}
        \widetilde{V}_k - \widetilde{V}_{k+1} & = -\lambda_{k,k+1} \left( f(x_{k+1})-f(x_k) + \langle \nabla f(x_{k+1}) , x_{k} - x_{k+1} \rangle + \frac{1}{2 L} \| \nabla f(x_{k}) - \nabla f(x_{k+1}) \|^2 \right)  \\ 
        & \phantom{=} - \lambda_{\star,k+1} \left( f(x_{k+1})-f(x_{\star}) + \langle \nabla f(x_{k+1}) , x_{\star} - x_{k+1} \rangle + \frac{1}{2 L} \| \nabla f(x_{k+1}) \|^2 \right) \\ 
        &\phantom{=} + \alpha_{k+1} \sqnorm{\frac{1}{L}\nabla f(x_{k+1}) - \frac{1}{2N-k} (x_{k+1} - x_\star)} 
    \end{aligned}
\end{align}
for $a_k, b_k, c_k, a_{k+1}, b_{k+1}, c_{k+1}$ and $\alpha_{k+1}$.
This yields
\begin{align*}
    a_k = \frac{k+1}{2N-k} , \quad b_k = -\frac{1}{2L} \frac{k+1}{2N-k} , \quad c_k = \frac{L}{2} \frac{2N-2k-1}{(2N-k)^2} , \quad \alpha_k = \frac{L}{2} \frac{2N-k+1}{2N - k} \pr{ \lambda_{k,k+1} + \lambda_{k-1,k} } 
\end{align*}
holds for $0 \le k \le N-1$, and also for $k=N$ for $\alpha_k$, with the convention $\lambda_{-1,0} = 0 = \lambda_{N,N+1}$.

Putting the above altogether, we conclude as follows:
\begin{theorem} \label{theorem:gd}
    Let $f\colon \reals^d \to \reals$ be convex and $L$-smooth. Let $x_\star$ be a minimizer of $f$, and let $x_0 \in \reals^d$ be an initial point.
    Let $x_k$ be iterates generated from \eqref{eq:gd}, and define
    \begin{equation*}
        \widetilde{V}_k = 
            \frac{k+1}{2N-k} \pr{ f(x_k) - f(x_\star) - \frac{1}{2L} \norm{ \nabla f(x_k) }^2 }
            + \frac{L}{2} \frac{2N-2k-1}{(2N-k)^2} \norm{ x_{k+1} - x_\star }^2 
    \end{equation*}
    for $k=0,\dots,N-1$, $\widetilde{V}_{-1} =  \frac{L}{4N+2} \norm{ x_{0} - x_\star }^2$ and $\widetilde{V}_N =  f(x_N) - f(x_\star)$. 
    Then, the identity \eqref{eqn:gd-lyapunov-consecutive-difference} holds for $k=0,\dots,N-1$, and $\widetilde{V}_0 \le \widetilde{V}_{-1}$.
    Consequently, $\widetilde{V}_N \le \widetilde{V}_{-1}$, and equivalently, $f(x_N) - f(x_\star) \le \frac{L\sqnorm{x_0 - x_\star}}{4N+2}$.
\end{theorem}

\begin{proof} 
Once we prove the identity~\eqref{eqn:gd-lyapunov-consecutive-difference} for $k=0,\dots,N-1$ and show that
\begin{align*}
    \widetilde{V}_0 - \widetilde{V}_{-1} = \lambda_{\star,0} \left( f(x_0) - f(x_\star) + \inprod{\nabla f(x_0)}{x_\star - x_0} + \frac{1}{2L} \sqnorm{\nabla f(x_0)} \right) - \alpha_0 \sqnorm{\frac{1}{L}\nabla f(x_0) - \frac{1}{2N+1} (x_0 - x_\star)} ,
\end{align*}
we establish the chain $\widetilde{V}_{N} \le \cdots \le \widetilde{V}_0 \le \widetilde{V}_{-1}$ and the proof is complete.
Hence the key step is the algebraic verification, which we provide in \texttt{examples\_lyapunov/gd/gd\_example\_lyap.ipynb}.
\end{proof} 

\paragraph{Remark.} This convergence guarantee has been presented in \citep{DroriTeboulle2014_performance} via a PEP-style proof, 
and \citep{TeboulleVaisbourd2023_elementary} provided the same rate with a more elementary proof.
However, neither work has formally presented their arguments in the form of Lyapunov analysis similar to \cref{theorem:gd}.
While the proof of \citep{TeboulleVaisbourd2023_elementary} can be reassembled into a Lyapunov-style proof, it still produces a Lyapunov function $W_k = k(f(x_k) - f(x_\star)) + \frac{k(k+1)}{2L} \sqnorm{\nabla f(x_k)} + \frac{L}{2} \sqnorm{x_k - x_\star}$, different from our \cref{theorem:gd}.

\subsection{Bregman Proximal Gradient Method (BPGM)}
\label{section:bpgm}

Consider the setting of Appendix~\ref{section:setting-composite-minimization-bregman}, 
where we solve the composite convex minimization problem with respect to $F = \bregsm+\bregprox$ where $\bregsm$ is a relatively $L$-smooth term and $\bregprox$ is a nonsmooth proximable term with respect to $h$, the Legendre-type convex function defining the Bregman distance $D_h$.
We consider the \textit{Bregman proximal gradient method (BPGM)}, defined by
\begin{align}
\label{eq:bpgm_main}
\tag{BPGM}
    x_{k+1}\in \argmin_{y\in C}
    \left\{
        \bregprox(y)+\inprod{\nabla \bregsm(x_k)}{y-x_k}
        +L D_h(y,x_k)
    \right\} 
\end{align}
with the associated optimality condition 
\begin{align}
\label{eq:bpgm_optimality_main}
    \tnabla \bregprox(x_{k+1})+\nabla \bregsm(x_k)
    =-L\pr{\nabla h(x_{k+1})-\nabla h(x_k)}.
\end{align}
BPGM has the convergence bound $F(x_N)-F(x_{\compidx})\le \frac{L}{N}D_h(x_{\compidx}, x_0)$ which holds for any $x_{\compidx} \in \operatorname{dom}F\cap\operatorname{dom}h$ with finite $D_h(x_{\compidx},x_0)$, which is a special case of \cite[Corollary~1]{GutmanPena2023_perturbed}. 
Here, $x_{\compidx}$ is an arbitrary reference point and is not necessarily a minimizer of $F$.
Define $\bregaux = h - \frac{1}{L}\bregsm$.
For each convex function $q \in \set{\bregsm,\bregprox,h,\bregaux}$, define
\[
    \cC_q(x,y) = q(y)-q(x)+\inprod{\tnabla q(y)}{x-y}\le 0,
\]
to be the nonpositive convexity inequality, 
where $\tnabla q(y)$ denotes a selected subgradient when $q$ is nonsmooth and $\tnabla q(y)=\nabla q(y)$ when $q$ is smooth. 
The PEP-style proof of BPGM's convergence rate takes the form
\begin{align}
    F(x_N)-F(x_{\compidx})-\frac{L}{N}D_h(x_{\compidx},x_0)
    &=
    \sum_{i=1}^{N-1}\lambda^{\bregsm}_{i,i+1} \cC_{\bregsm}(x_i,x_{i+1})
    +\sum_{i=0}^{N-1}\lambda^{\bregsm}_{\compidx,i} \cC_{\bregsm}(x_{\compidx},x_i) \nonumber \\
    &\quad
    +\sum_{i=1}^{N-1}\lambda^{\bregprox}_{i,i+1} \cC_{\bregprox}(x_i,x_{i+1})
    +\sum_{i=1}^{N}\lambda^{\bregprox}_{\compidx,i} \cC_{\bregprox}(x_{\compidx},x_i) \nonumber \\
    &\quad
    +\sum_{i=1}^{N-1}\lambda^{\bregaux}_{i,i+1} \cC_{\bregaux}(x_i,x_{i+1})
    +\sum_{i=1}^{N}\lambda^{\bregaux}_{i,i-1} \cC_{\bregaux}(x_i,x_{i-1})
    +\lambda^h_{\compidx,N} \cC_h(x_{\compidx},x_N),
    \label{eq:bpgm_full_pep_certificate}
\end{align}
where the nonnegative weights are specified later.
In this case, the active interpolation inequalities are
\[
    \begin{aligned}
    \cI^{\bregsm} &= \{(i,i+1)\,|\,i=1,\dots,N-1\}\cup\{(\compidx,i)\,|\,i=0,\dots,N-1\} \\
    \cI^{\bregprox} &= \{(i,i+1)\,|\,i=1,\dots,N-1\}\cup\{(\compidx,i)\,|\,i=1,\dots,N\} \\
    \cI^{\bregaux} &= \{(i,i+1)\,|\,i=1,\dots,N-1\}\cup\{(i,i-1)\,|\,i=1,\dots,N\} \\
    \cI^h &= \{(\compidx,N)\} \\
    \cJ &= \emptyset .
    \end{aligned}
\]

\paragraph{Step 1.}
Take $\cI_0^{\bregsm}=\cI_0^{\bregprox}=\cI_0^{\bregaux}=\cI_0^h=\cJ_0=\emptyset$, and for $k=1,\dots,N$, take
\[
    \begin{aligned}
    \cI_k^{\bregsm} &= \{(i,i+1)\,|\,i=1,\dots,k-1\}\cup\{(\compidx,i)\,|\,i=0,\dots,k-1\} \\
    \cI_k^{\bregprox} &= \{(i,i+1)\,|\,i=1,\dots,k-1\}\cup\{(\compidx,i)\,|\,i=1,\dots,k\} \\
    \cI_k^{\bregaux} &= \{(i,i+1)\,|\,i=1,\dots,k-1\}\cup\{(i,i-1)\,|\,i=1,\dots,k\} \\
    \cI_k^h &= \emptyset \\
    \cJ_k &= \emptyset .
    \end{aligned}
\]
The only remaining term in the proof~\eqref{eq:bpgm_full_pep_certificate} is the terminal interpolation inequality $\cC_h(x_{\compidx},x_N)$.

\paragraph{Step 2.}
This choice clearly satisfies sufficiency.
Numerically, $\vV_0 = 0$ and $\vV_k$ has the constant rank $4$ for $k=1,\dots,N$.

\paragraph{Step 3.}
We construct the set of special vectors as
\[
    \cT = 
    \left\{\vx_i \,\middle|\, i=0,\dots,N,\compidx \right\}
    \cup \left\{\vxi_i^{\bregsm},\vxi_i^h \,\middle|\, i=0,\dots,N \right\}
    \cup \left\{\vxi_i^{\bregprox} \,\middle|\, i=1,\dots,N\right\},
    \qquad
    \cS = \cT \cup \{\vu-\vw \,|\, \vu,\vw\in\cT\},
\]
where $\vxi_i^{\bregprox}$ represents $\tnabla \bregprox(x_i)$ for $i=1,\dots,N$, and $\vxi_i^{\bregsm}$ and $\vxi_i^h$ represent $\nabla \bregsm(x_i)$ and $\nabla h(x_i)$ for $i=0,\dots,N$.

\paragraph{Step 4.}
Numerical observation based on the results of Step~3 reveals that all $V_k$ contain the term $-\frac{L}{N} D_h(x_{\compidx}, x_0)$.
We add this term throughout and consider
\[
    \widetilde V_k
    =
    V_k+\frac{L}{N}D_h(x_{\compidx},x_0)
\]
for $k=0,\dots,N$.
Applying Step~3 again to this adjusted sequence, we obtain $\{\vxi_k^h,\vx_k-\vx_{\compidx}\}\subset\cS\cap\cR(\widetilde{\vV}_k)$ with the corresponding coefficient matrix given by \cref{proposition:linear-algebra-of-lyapunov} being simply $\begin{bmatrix} 0 & \frac{L}{2N}\\ \frac{L}{2N} & 0 \end{bmatrix}$, representing the single inner product term $\frac{L}{N}\inprod{\nabla h(x_k)}{x_k-x_{\compidx}}$.
Combining this with the function-value coordinate support, we find that
\[
    \widetilde V_k = a_k\pr{F(x_k)-F(x_{\compidx})} + b_kD_h(x_{\compidx},x_k)
\]
is the Lyapunov skeleton for $k=0,\dots,N$.

\paragraph{Step 5.}
Substituting the above form into the identity
\begin{align}
\label{eq:bpgm_difference_identity}
\begin{aligned}
    \widetilde V_k-\widetilde V_{k+1}
    & =
    -\frac{k}{N}\cC_{\bregsm}(x_k,x_{k+1})
    -\frac{1}{N}\cC_{\bregsm}(x_{\compidx},x_k) \\
    & \quad 
    -\frac{k}{N}\cC_{\bregprox}(x_k,x_{k+1})
    -\frac{1}{N}\cC_{\bregprox}(x_{\compidx},x_{k+1}) \\
    & \quad 
    -\frac{Lk}{N}\cC_{\bregaux}(x_k,x_{k+1})
    -\frac{L(k+1)}{N} \cC_{\bregaux}(x_{k+1},x_k)
\end{aligned}
\end{align}
for $k=0,\dots,N-1$, where the terms multiplied by $k$ are omitted when $k=0$, 
which follows from the analytic inequality weights given by
\[
    \begin{aligned}
    \lambda^{\bregsm}_{i,i+1}&=\frac{i}{N},&
    \lambda^{\bregsm}_{\compidx,i}&=\frac{1}{N},&
    \lambda^{\bregprox}_{i,i+1}&=\frac{i}{N},&
    \lambda^{\bregprox}_{\compidx,i}&=\frac{1}{N},\\
    \lambda^{\bregaux}_{i,i+1}&=\frac{Li}{N},&
    \lambda^{\bregaux}_{i,i-1}&=\frac{Li}{N},&
    \lambda^h_{\compidx,N}&=\frac{L}{N},
    \end{aligned}
\]
and using the condition~\eqref{eq:bpgm_optimality_main}, 
we can symbolically identify the solution $a_k = \frac{k}{N}$, $a_{k+1} = \frac{k+1}{N}$, and $b_k=b_{k+1}=\frac{L}{N}$.
Therefore,
\begin{align}
\label{eqn:bpgm-lyapunov}
    \widetilde V_k=\frac{k}{N}\pr{F(x_k)-F(x_{\compidx})}+\frac{L}{N}D_h(x_{\compidx},x_k),
    \qquad k=0,\dots,N.
\end{align}
Summarizing the above, we obtain the following result.

\begin{theorem}
\label{theorem:bpgm-lyapunov}
Let $h$ be a Legendre-type convex function with $C=\operatorname{int}\operatorname{dom}h\neq\emptyset$.
Let $\bregprox$ be closed, convex, and proper with $\operatorname{dom}\bregprox\cap C\neq\emptyset$, and let $\bregsm$ be convex on its domain containing $C$, and differentiable on $C$.
Suppose that $\bregaux=h-\frac{1}{L}\bregsm$ is convex on $C$ for some $L>0$.
Let $x_0 \in C$ and assume that the subproblems in \eqref{eq:bpgm_main} admit selected minimizers in $C$ satisfying the optimality condition \eqref{eq:bpgm_optimality_main}.
Let $x_k$ be the resulting iterates, and let $x_{\compidx}\in\operatorname{dom}F\cap\operatorname{dom}h$ be any reference point satisfying $D_h(x_{\compidx},x_0)<\infty$.
Then, the sequence $\widetilde V_k$ defined in \eqref{eqn:bpgm-lyapunov} satisfies \eqref{eq:bpgm_difference_identity} for $k=0,\dots,N-1$.
Consequently, $\widetilde V_N\le\cdots\le \widetilde V_0=\frac{L}{N}D_h(x_{\compidx},x_0)$, and this implies
\[
    F(x_N)-F(x_{\compidx})\le \frac{L}{N}D_h(x_{\compidx},x_0).
\]
\end{theorem}

\begin{proof}
The identity \eqref{eq:bpgm_difference_identity} follows by direct expansion using the optimality condition \eqref{eq:bpgm_optimality_main}; the associated BPGM notebook \path{bpgm_example_lyap.ipynb} provides a symbolic verification.
Since $\cC_q \le 0$ for $q\in \{\bregsm,\bregprox,\bregaux\}$, the right-hand side of \eqref{eq:bpgm_difference_identity} is nonnegative, and therefore $\widetilde V_N\le\cdots\le \widetilde V_0=\frac{L}{N}D_h(x_{\compidx},x_0)$.
Finally, because $D_h(x_{\compidx}, x_N) \ge 0$,
\[
    F(x_N)-F(x_{\compidx}) \le \widetilde V_N \le \frac{L}{N}D_h(x_{\compidx},x_0) ,
\]
which completes the proof.
\end{proof}

In the special case $\bregprox\equiv0$, BPGM reduces to Bregman gradient descent (BGD). 
The same specialization gives the following Lyapunov analysis for BGD relative to an arbitrary comparison point. 
This recovers the rate in \cite{LuFreundNesterov2018_relatively} and matches the lower-bound result established in \cite{dragomirOptimalComplexityCertification2022}.

\begin{corollary}[Bregman gradient descent]
\label{corollary:bgd-rate}
Let $\bregsm,h$ satisfy the conditions of \cref{theorem:bpgm-lyapunov}.
Let $x_0 \in C$ and let $x_k$ be generated by the Bregman gradient descent
\begin{align}
\tag{BGD}\label{eq:bgd_main}
    x_{k+1}\in\argmin_{y\in C}
    \left\{
        \inprod{\nabla \bregsm(x_k)}{y-x_k}+L D_h(y,x_k)
    \right\}
\end{align}
for $k=0,\dots,N-1$, or equivalently via the relation
\[
    \nabla \bregsm(x_k)+L\pr{\nabla h(x_{k+1})-\nabla h(x_k)}=0.
\]
Let $x_{\compidx} \in \operatorname{dom}\bregsm\cap\operatorname{dom}h$ satisfy $D_h(x_{\compidx},x_0)<\infty$.
Then the Lyapunov sequence
\[
    \widetilde V_k=\frac{k}{N}\pr{\bregsm(x_k)-\bregsm(x_{\compidx})}+\frac{L}{N}D_h(x_{\compidx},x_k)
\]
is nonincreasing for $k=0,\dots,N$, and consequently,
\[
    \bregsm(x_N)-\bregsm(x_{\compidx})\le \frac{L}{N}D_h(x_{\compidx},x_0).
\]
\end{corollary}

\begin{proof}
    
Set $\bregprox\equiv0$ in \cref{theorem:bpgm-lyapunov}.
Then we have $F=\bregsm$ and \eqref{eq:bpgm_main} becomes \eqref{eq:bgd_main}.
We obtain the conclusion by applying \cref{theorem:bpgm-lyapunov} with the same reference point $x_{\compidx}$.
\end{proof}

Similarly, in the special case $\bregsm\equiv0$, BPGM reduces to the \textit{Bregman proximal point method (BPPM)}, a classical algorithm which dates back to \cite{CZ1992}, and the \cref{theorem:bpgm-lyapunov} immediately yields the following corollary for BPPM, recovering the rate obtained in \cite{AuslenderTeboulle2006_interior}.

\begin{corollary}[Bregman proximal point method]
\label{corollary:bppm-rate}
Let $\bregprox,h$ satisfy the conditions of \cref{theorem:bpgm-lyapunov}.
Let $x_0 \in C$ and let $x_k$ be generated by the Bregman proximal point method
\begin{align}
\tag{BPPM}\label{eq:bppm_main}
    x_{k+1}\in\argmin_{y\in C}
    \left\{
        \bregprox(y)+L D_h(y,x_k)
    \right\}
\end{align}
for $k=0,\dots,N-1$.
Let $x_{\compidx} \in \operatorname{dom}\bregprox\cap\operatorname{dom}h$ satisfy $D_h(x_{\compidx},x_0)<\infty$.
Then the Lyapunov sequence
\[
    \widetilde V_k=\frac{k}{N}\pr{\bregprox(x_k)-\bregprox(x_{\compidx})}+\frac{L}{N}D_h(x_{\compidx},x_k)
\]
is nonincreasing for $k=0,\dots,N$, and consequently,
\[
    \bregprox(x_N)-\bregprox(x_{\compidx})\le \frac{L}{N}D_h(x_{\compidx},x_0) .
\]
\end{corollary}

\begin{proof}
Set $\bregsm\equiv0$ in \cref{theorem:bpgm-lyapunov}.
Then $F=\bregprox$, $\bregaux=h$, and \eqref{eq:bpgm_main} becomes \eqref{eq:bppm_main}.
We obtain the conclusion by applying \cref{theorem:bpgm-lyapunov} with the same reference point $x_{\compidx}$.
\end{proof}

As a side observation, the following lemma proves tightness even after the comparison point in \cref{corollary:bppm-rate} is chosen to be the unique minimizer $x_\star$ of $\bregprox$.
In particular, it also shows that the rate in \cref{corollary:bppm-rate}, which allows an arbitrary comparison point, cannot be improved.
It uses a lower bound construction that is closely related to the pathological limiting examples used in \cite{dragomirOptimalComplexityCertification2022}. 
We present the proof of this result in Appendix~\ref{appendix:bppm-tightness}.

\begin{lemma}[Tightness of the BPPM rate]
\label{lemma:bppm-tightness}
For any $0<\varepsilon<1$, $N\ge1$, $L>0$, and dimension $d\ge1$, there exist a proper closed convex function $\bregprox:\reals^d\to\reals$, a Legendre-type function $h:\reals^d\to\reals$, and an initial point $x_0\in\reals^d$ such that the iterates of \ref{eq:bppm_main} are uniquely defined and satisfy
\[
    \bregprox(x_N)-\bregprox(x_\star)
    \ge (1-\varepsilon)\frac{L}{N}D_h(x_\star,x_0) 
\]
where $x_\star$ is the unique minimizer of $\bregprox$.
Therefore, the rate of BPPM from \cref{corollary:bppm-rate} 
cannot be improved further.
\end{lemma}

\paragraph{Remark.}
It is worth noting that, after specializing the comparison point to a minimizer, the tight sublinear convergence factor for BGD in \cref{corollary:bgd-rate} (also reported in \cite{dragomirOptimalComplexityCertification2022}) and BPPM in \cref{corollary:bppm-rate} coincide, while this is not the case in the Euclidean setting, i.e., the tight convergence factors for GD and PPM differ \citep{TeboulleVaisbourd2023_elementary}.
Furthermore, perhaps interestingly, the single Lyapunov function in \cref{theorem:bpgm-lyapunov}, when specialized to the cases of BGD and BPPM by letting $\bregprox\equiv 0$ or $\bregsm\equiv 0$, still yields the tight rates for each respective algorithm.

\subsection{Proximal point method for monotone inclusion}
\label{section:ppm}

Let $\opA\colon\reals^d\rightrightarrows\reals^d$ be maximally monotone and let $x_\star$ satisfy $0\in\opA x_\star$.
We consider the standard proximal point method
\begin{align}
\label{alg:ppm}
\tag{PPM}
    x_{k+1}=\opJ_{\opA}x_k 
\end{align}
with the resolvent $\opJ_\opA = (\opI + \opA)^{-1}$.
For $k=0,\dots,N$, let $\topa x_{k+1} := x_k-x_{k+1}\in\opA x_{k+1}$.
Let $q=\frac{N}{N+1}$ and $\tau_N=\frac{q^N}{N+1}$.
The tight last-iterate residual norm guarantee is provided in \citep{guTightSublinearConvergence2020} via PEP-style proof:
\begin{align}
\label{eqn:ppm-pep-proof}
    & \sqnorm{\topa x_{N+1}}-\tau_N\sqnorm{x_0-x_\star} \notag\\
    & = -\sum_{i=1}^{N} \lambda_{i,i+1}
    \inprod{\topa x_i-\topa x_{i+1}}{x_i-x_{i+1}}
    -\sum_{i=1}^{N} \lambda_{i,\star}
    \inprod{\topa x_i}{x_i-x_\star} -\frac{2}{N+1}\inprod{\topa x_{N+1}}{x_{N+1}-x_\star}
    -\sum_{i=1}^{N}\alpha_i\sqnorm{s_i} ,
\end{align}
where
\[
    \lambda_{i,i+1} = 2 \left( \frac{N}{N+1} \right)^{N-i}\frac{i}{N+1},
    \qquad
    \lambda_{i,\star} = 2\left( \frac{N}{N+1} \right)^{N-i}\frac{N-i}{(N+1)^2},
    \qquad
    \alpha_i = \left( \frac{N}{N+1} \right)^{N-i}\frac{i}{N},
\]
and
\[
    s_i = \frac{N}{N+1} (x_{i-1}-x_\star) - \frac{2N}{N+1}(x_i-x_\star) + (x_{i+1}-x_\star).
\]

\paragraph{Step 1.}
Here, the evaluated residual is $\topa x_{N+1} = x_N-x_{N+1}$, so the natural Lyapunov index runs from $0$ to $N$ 
(as opposed to $0$ to $N-1$ in the case of Section~\ref{section:recovering-proximal-monotone}, where we evaluate $\sqnorm{\topa x_N}$).
We take $\cI_0 = \cJ_0=\emptyset$ and, for $k=1,\dots,N$,
\[
    \cI_k
    =
    \{(j,j+1)\,|\,j=1,\dots,k\}
    \cup
    \{(j,\star)\,|\,j=1,\dots,k\},
    \qquad
    \cJ_k=\{1,\dots,k\}.
\]
The only term left outside the Lyapunov analysis is the terminal monotonicity inequality for the pair $(N+1,\star)$.

\paragraph{Step 2.}
This choice trivially satisfies sufficiency.
The numerical computation shows that $\vV_k$ has an almost constant rank $3$ for $k=1,\dots,N-2$ and $k=N$, 
while the rank drops to $2$ for $k=N-1$.

\paragraph{Step 3.}
We construct the special-vector set from the PEP context vectors
\[
    \cT=\{\vx_0,\vx_\star,\vg_1,\dots,\vg_{N+1},\vx_1,\dots,\vx_{N+1}\},
    \qquad
    \cS=\cT\cup\{\vu-\vv:\vu,\vv\in\cT\},
\]
where $\vg_i$ represents $\topa x_i$ for $i=1,\dots,N+1$.

\paragraph{Step 4.}
For $k=1,\dots,N-2$, the code identifies
\[
    \{\vx_0 - \vx_\star, \vg_{k+1}, \vx_k-\vx_\star\}\subset \cS\cap\cR(\vV_k)
\]
and the coefficient matrix $\vA_k$ from this basis, computed according to \cref{proposition:linear-algebra-of-lyapunov}, is of the form
\[
    \vA_k = \bmat{-\tau_N & 0 & 0 \\ 0 & * & * \\ 0 & * & * } .
\]
Indeed, letting $\widetilde{V}_k = V_k + \tau_N \sqnorm{x_0 - x_\star}$ and re-running Steps 2--4 on it reveals the structure
\begin{align}
\label{eqn:ppm-monotone-lyapunov}
    \widetilde{V}_k = a_k \sqnorm{\topa x_{k+1}} + b_k \inprod{\topa x_{k+1}}{x_k-x_\star} + c_k \sqnorm{x_k-x_\star}
\end{align}
for $k=1,\dots,N-2$, while $\widetilde{V}_0 = \tau_N \sqnorm{x_0 - x_\star}$, $\widetilde{V}_N = \sqnorm{\topa x_{N+1}} + b_N \inprod{\topa x_{N+1}}{x_N - x_\star}$ (so we can let $a_0 = b_0 = 0$, $c_0 = \tau_N$ and $a_N = 1$, $c_N = 0$).
We hypothesize that \eqref{eqn:ppm-monotone-lyapunov} in fact holds for all $k=0,\dots,N$, while for $k=N-1$, the rank drop occurs because $b_{N-1}^2 = 4a_{N-1} c_{N-1}$ and \eqref{eqn:ppm-monotone-lyapunov} collapses to a perfect square of a linear combination of $\topa x_N$ and $x_{N-1} - x_\star$.

\paragraph{Step 5.}
To verify the hypothesis and determine the coefficients, we substitute $x_{k+2} = x_{k+1} - \topa x_{k+2}$ and \eqref{eqn:ppm-monotone-lyapunov} into 
\begin{align}
\label{eqn:ppm-lyapunov-difference}
\begin{aligned}
    \widetilde{V}_k - \widetilde{V}_{k+1}
    & =
    \lambda_{k+1,k+2}
    \inprod{\topa x_{k+1}-\topa x_{k+2}}{x_{k+1}-x_{k+2}}
    + \lambda_{k+1,\star} \inprod{\topa x_{k+1}}{x_{k+1}-x_\star} + \alpha_{k+1}\sqnorm{s_{k+1}}
\end{aligned}
\end{align}
and solve the equation for $a_k, b_k, c_k, a_{k+1}, b_{k+1}$ and $c_{k+1}$.
This reveals the unique solution 
\begin{alignat*}{3}
    a_k & = \left( \frac{N}{N+1} \right)^{N-k} \frac{k(N-1)}{N(N+1)}, & b_k & = 2\left( \frac{N}{N+1} \right)^{N-k}\frac{k}{N(N+1)} , & c_k & = \left( \frac{N}{N+1} \right)^{N-k}\frac{N-k}{N(N+1)} \\
    a_{k+1} & = \left( \frac{N}{N+1} \right)^{N-k-1} \frac{(k+1)(N-1)}{N(N+1)}, \quad & b_{k+1} & = 2\left( \frac{N}{N+1} \right)^{N-k-1}\frac{k+1}{N(N+1)}, \quad & c_{k+1} & = \left( \frac{N}{N+1} \right)^{N-k-1}\frac{N-k-1}{N(N+1)}
\end{alignat*}
that is consistent with respect to $k$.
Summarizing, we conclude:

\begin{theorem}
\label{theorem:ppm-lyapunov}
Let $\opA\colon\reals^d\rightrightarrows\reals^d$ be maximally monotone, let $x_\star$ satisfy $0\in\opA x_\star$, and let $\{x_k\}_{k=0}^{N+1}$ be generated by \eqref{alg:ppm}.
Define
\[
    \widetilde{V}_k = \left( \frac{N}{N+1} \right)^{N-k} \left( \frac{k(N-1)}{N(N+1)} \sqnorm{\topa x_{k+1}} + \frac{2k}{N(N+1)} \inprod{\topa x_{k+1}}{x_k-x_\star} 
    + \frac{N-k}{N(N+1)} \sqnorm{x_k-x_\star} \right) .
\]
Then \eqref{eqn:ppm-lyapunov-difference} holds for $k=0,\dots,N-1$, which implies
\[
    \widetilde{V}_N \le \dots \le \widetilde{V}_0
\]
and consequently, 
\[
    \sqnorm{\topa x_{N+1}}
    \le \frac{1}{N+1} \left( \frac{N}{N+1} \right)^N \sqnorm{x_0 - x_\star}
    = \frac{1}{(1+1/N)^N(N+1)}\sqnorm{x_0-x_\star}.
\]
\end{theorem}

\begin{proof}
The identity \eqref{eqn:ppm-lyapunov-difference} is verified symbolically in
\path{examples_lyapunov/ppm_monotone/ppm_monotone_example_lyap.ipynb}.
It is obtained by expanding both sides in the independent vectors
$x_k-x_\star$, $\topa x_{k+1}$, and $\topa x_{k+2}$, using
$x_{k+1}=x_k-\topa x_{k+1}$ and $x_{k+2}=x_{k+1}-\topa x_{k+2}$.

Observe that when $k=0$, we have $\widetilde{V}_0 = \frac{1}{N+1} \left( \frac{N}{N+1} \right)^N \sqnorm{x_0-x_\star}$, 
while when $k=N$, using $x_N - x_\star = (x_{N+1} - x_\star) + \topa x_{N+1}$, we have
\begin{align*}
    \widetilde{V}_N & = \frac{N-1}{N+1} \sqnorm{\topa x_{N+1}} + \frac{2}{N+1} \inprod{\topa x_{N+1}}{x_{N+1} - x_\star + \topa x_{N+1}} \\
    & = \sqnorm{\topa x_{N+1}} + \frac{2}{N+1}\inprod{\topa x_{N+1}}{x_{N+1}-x_\star} \\
    & \ge \sqnorm{\topa x_{N+1}}
\end{align*}
by monotonicity of $\opA$.
Therefore, we conclude that $\sqnorm{\topa x_{N+1}} \le \widetilde{V}_N \le \widetilde{V}_0$ as desired.
\end{proof}

\paragraph{Remark.}
This convergence guarantee has been presented in the form of PEP-style proof without any specification of Lyapunov function in \cite{guTightSublinearConvergence2020}, which also provided a matching lower bound construction showing that the rate is tight.

\subsection{\dualoc}
\label{section:dual-oc-halpern}

Recall the setting of Section~\ref{section:os-ppm-and-oc-halpern}, where $\opA\colon \reals^d \rightrightarrows \reals^d$ is a maximally $\mu$-strongly monotone operator with a zero $x_\star$.
Let $\gamma = 1+2\mu$.
The \ref{alg:oc-halpern} algorithm with $\opT = \left(1 + \frac{1}{\gamma}\right) \opJ_\opA - \frac{1}{\gamma} \opI$ achieved the exact optimal residual guarantee 
\(\sqnorm{\topa x_N}\le \Gamma_N^{-2}\sqnorm{y_0-x_\star}\) \citep{ParkRyu2022_exact}, which we recovered in \cref{proposition:oc-halpern}.
When $\mu=0 \iff \gamma=1$, \ref{alg:oc-halpern} reduces to \ref{alg:ohm}, and in that setting, \ref{alg:dual-ohm} shares the same exact optimal residual guarantee at the final iteration, as we recovered in \cref{proposition:dual-ohm}.
The two algorithms, \ref{alg:ohm} and \ref{alg:dual-ohm} are in the so-called \emph{H-dual} relationship---their step size matrices are anti-diagonal transposes of each other---and 
the correspondence between H-dual algorithms' convergence guarantees is referred to as \emph{H-duality} \citep{KimOzdaglarParkRyu2023_timereversed,YoonKimSuhRyu2024_optimal}.
In this section, we extend this observation to the case $\mu>0 \iff \gamma>1$ by analyzing the novel algorithm \emph{dual optimal contractive Halpern}:
\begin{equation}   \label{alg:dual-oc-halpern}  \tag{\dualoc}
    y_{k+1} = y_k + \pr{1 - \frac{1}{\varphi_{N-k-1}} }\bigl(\opT y_k - \opT y_{k-1}\bigr)
\end{equation}
for $k=0,\dots,N-2$, where $\varphi_k = \sum_{i=0}^{k} \gamma^{2i}$ and $\opT y_{-1} = y_0$. 
Specifically, we show that \ref{alg:dual-oc-halpern} exhibits the convergence guarantee $\sqnorm{\topa x_N}\le \Gamma_N^{-2}\sqnorm{y_0-x_\star}$,
exactly the same as \ref{alg:oc-halpern} (\cref{theorem:dual-oc-halpern}).

\paragraph{Remark.}
We do not formally show that \ref{alg:dual-oc-halpern} is the H-dual of \ref{alg:oc-halpern}, or establish an H-duality theorem generalizing the correspondence between the two algorithms, 
as it is not the focus of the current work.\bigskip

As before, let $\Gamma_N=\sum_{i=0}^{N-1}\gamma^i, x_k=\opJ_\opA y_{k-1}$ and $\topa x_k=y_{k-1}-x_k\in\opA x_k$.
The PEP-style proof identified by the solver takes the form
\begin{align*}
    \sqnorm{\topa x_N}-\frac{1}{\Gamma_N^2}\sqnorm{y_0-x_\star}
    &=
    -\sum_{k=1}^{N-1}\lambda_{k,N}
    \left(
    \inprod{\topa x_k-\topa x_N}{x_k-x_N}
    -\mu\sqnorm{x_k-x_N}
    \right) \notag\\
    &\quad
    -\lambda_{N,\star}
    \left(
    \inprod{\topa x_N}{x_N-x_\star}
    -\mu\sqnorm{x_N-x_\star}
    \right) \notag\\
    &\quad
    -\frac{\gamma^{-N}}{\Gamma_N^2}
    \sqnorm{\Gamma_N\topa x_N+(x_N-y_0)-\gamma^N(x_N-x_\star)},
\end{align*}
where
\begin{align}
\label{eqn:dual-oc-halpern-lambda-k-N}
    \lambda_{k,N} = \frac{2(1+\gamma^{-N})}{1+\gamma} \frac{\gamma^{2(N-k)}}{\varphi_{N-k}\varphi_{N-k-1}} = \frac{4}{\Gamma_N^2} \frac{1}{\lambda_{N-k,N-k+1}^{\text{(OC-Halpern)}}} ,
\end{align}
exhibiting the similar relationship between the inequality multipliers used in the proofs of \ref{alg:ohm} and \ref{alg:dual-ohm}.
Furthermore, $\lambda_{N,\star} = \frac{2}{\Gamma_N}$ and the final square term is identical to the \ref{alg:oc-halpern} case.
We will not need these analytic expressions until the final stage, Step~5.
As in \ref{alg:dual-ohm}, we have \(\cI=\{(k,N)\,|\,k=1,\dots,N-1\}\cup \{(N,\star)\}\), and $|\cJ| = 1$.

\paragraph{Step 1.}
Take \(\cI_0=\cJ_0=\emptyset\), and for \(k=1,\dots,N-1\),
\[
    \cI_k=\{(j,N)\,|\,j=1,\dots,k\},
    \qquad
    \cJ_k=\emptyset .
\]

\paragraph{Step 2.}
This choice clearly satisfies sufficiency. 
Numerically, we observe that \(\vV_0=0\) and the ranks of \(\vV_k\) are \(2,4,\dots,4,2\) for \(k=1,\dots,N-1\) as in the case of \ref{alg:dual-ohm}.

\paragraph{Step 3.}
We construct the same set of special vectors as \ref{alg:dual-ohm}:
\[
    \cT
    =
    \{\vy_0,\dots,\vy_{N-1},\vx_\star,\vx_1,\dots,\vx_N,\vg_1,\dots,\vg_N,\vt_0,\dots,\vt_{N-1}\},
    \qquad
    \cS=\cT\cup\{\vu-\vv:\vu,\vv\in\cT\} .
\]

\paragraph{Step 4.}
For $k=1,\dots,N-2$, we numerically identify that
\[
    \{\vg_N,\; \vy_0-\vx_N,\; \vy_k-\vy_{N-1},\; \vt_{k-1}-\vt_{N-1}\}
    \subset \cS\cap\cR(\vV_k)
\]
is a basis of $\cR(\vV_k)$.
The first two directions have coefficients that are constant over \(k\), and do not yield inner-product terms with the last two vectors.
Therefore, we hypothesize that by subtracting these constant terms involving $\topa x_N$ and $y_0 - x_N$, we have a translated lyapunov function of the form
\begin{align*}
    \widetilde{V}_k = a_k \sqnorm{y_k - y_{N-1}} + b \inprod{y_k - y_{N-1}}{\opT y_{k-1} - \opT y_{N-1}} + c_k \sqnorm{\opT y_{k-1} - \opT y_{N-1}} .
\end{align*}
Here, we enforce $b\in \reals$ to be constant across iterations as numerical observation suggests.

\paragraph{Step 5.}
We solve the equation
\begin{align}
\label{eqn:dual-oc-halpern-difference}
    \widetilde V_k-\widetilde V_{k+1}
    =
    \lambda_{k+1,N}
    \left(
    \inprod{\topa x_{k+1}-\topa x_N}{x_{k+1}-x_N}
    -\mu\sqnorm{x_{k+1}-x_N}
    \right)
\end{align}
for the coefficients $a_k, c_k, a_{k+1}, c_{k+1}$ and $b$, using the analytic formula \eqref{eqn:dual-oc-halpern-lambda-k-N} for $\lambda_{k+1,N}$.
The solution that yields consistent expressions for $a_k, a_{k+1}$ and $c_k, c_{k+1}$ is unique:
\begin{gather*}
    a_k = -C_N \frac{\varphi_{N-k}}{\varphi_{N-k}-1} , \qquad a_{k+1} = -C_N \frac{\varphi_{N-k-1}}{\varphi_{N-k-1}-1} , \qquad b = 2C_N \\
    c_k = -C_N \frac{\varphi_{N-k-1}-1}{\varphi_{N-k-1}} , \qquad c_{k+1} = -C_N \frac{\varphi_{N-k-2}-1}{\varphi_{N-k-2}}
\end{gather*}
where $C_N=\frac{\gamma^2(1+\gamma^{-N})}{(1+\gamma)^2}$.
We summarize the result as follows.

\begin{theorem} \label{theorem:dual-oc-halpern}
Let $\opA\colon\reals^d\rightrightarrows\reals^d$ be maximally $\mu$-strongly monotone with $\mu\ge 0$, and let $\gamma=1+2\mu$.
Let $x_\star$ satisfy $0 \in \opA x_\star$, set $y_\star=x_\star$, and let $y_0\in\reals^d$ be an initial point.
Let \(\varphi_k=\sum_{i=0}^{k}\gamma^{2i}\), \(\Gamma_N=\sum_{i=0}^{N-1}\gamma^i\), and \(C_N=\frac{\gamma^2(1+\gamma^{-N})}{(1+\gamma)^2}\).
Let $x_k,y_k$ be generated by \eqref{alg:dual-oc-halpern} with $\opT = \left(1 + \frac{1}{\gamma}\right) \opJ_\opA - \frac{1}{\gamma} \opI$, with $x_k=\opJ_\opA y_{k-1}$ and $\topa x_k=y_{k-1}-x_k$.
Define $\widetilde V_k$ by
\begin{align}
\label{eqn:dual-oc-halpern-lyapunov}
    \widetilde V_k = 
    C_N\bigg(
    &-\frac{\varphi_{N-k}}{\varphi_{N-k}-1}\sqnorm{y_k-y_{N-1}}
    +2\inprod{y_k-y_{N-1}}{\opT y_{k-1}-\opT y_{N-1}}
    -\frac{\varphi_{N-k-1}-1}{\varphi_{N-k-1}}
    \sqnorm{\opT y_{k-1}-\opT y_{N-1}}
    \bigg)
\end{align}
for $k=0,\dots,N-1$.
Then \(\widetilde V_{N-1}=0\), and for \(k=0,\dots,N-2\), \eqref{eqn:dual-oc-halpern-difference} holds with $\lambda_{k,N}$ defined as in \eqref{eqn:dual-oc-halpern-lambda-k-N}.
Consequently, \(\widetilde V_{N-1}\le\dots\le\widetilde V_0\), and
\[
    \sqnorm{\topa x_N}\le \frac{1}{\Gamma_N^2}\sqnorm{y_0-x_\star}.
\]
\end{theorem}

\begin{proof}
The identity \eqref{eqn:dual-oc-halpern-difference} is verified symbolically in 
\texttt{dual\_oc\_halpern\_example\_lyap.ipynb}.
The verification follows by eliminating \(y_{k+1}\) using the update rule, and applying the identities 
\(\varphi_m-\varphi_{m-1}=\gamma^{2m}\) and \(\varphi_m-1=\gamma^2\varphi_{m-1}\) for $m\ge 0$ to check all terms algebraically cancel out.
Then, strong monotonicity of \(\opA\) immediately implies \(\widetilde V_{N-1} \le \dots \le \widetilde V_0\), and \(\widetilde V_{N-1}=0\) follows directly from \(\varphi_0=1\).

It remains to show that this implies the final bound on the squared residual norm.
This is established once we verify the identity
\begin{align*}
    & \sqnorm{\topa x_N} - \frac{1}{\Gamma_N^2}\sqnorm{y_0-x_\star} \\
    & = \widetilde V_{N-1} - \widetilde V_0 - \frac{2}{\Gamma_N}
    \left( \inprod{\topa x_N}{x_N-x_\star} - \mu\sqnorm{x_N-x_\star} \right) 
    - \frac{\gamma^{-N}}{\Gamma_N^2} \sqnorm{\Gamma_N\topa x_N+(x_N-y_0)-\gamma^N(x_N-x_\star)} ,
\end{align*}
as the right hand side is a sum of nonpositive terms.
This is symbolically verified within the same notebook.
\end{proof}

%% file: sections/6_conclusion.tex
\section{Conclusion}

We propose a principled approach toward searching for Lyapunov functions in optimization.
This is achieved by combining the well-known framework leveraging computer assistance for identifying tight proof certificates with the novel procedure for decomposing them into layers via constructing partial sums and checking their admissibility. 
This suggests a new viewpoint on Lyapunov analysis in general, as a stratification of the best possible proof that can be attempted within a specific search space.
Whether an analogous treatment of Lyapunov analysis will be possible in other contexts, such as continuous-time optimization analysis \cite{SuBoydCandes2014_differential, 
SuBoydCandes2016_differential, 
KricheneBayenBartlett2015_accelerated, WibisonoWilsonJordan2016_variational, AttouchChbaniPeypouquetRedont2018_fast, DiakonikolasJordan2021_generalized, WilsonRechtJordan2021_lyapunov, SuhRohRyu2022_continuoustime, BotCsetnekNguyen2023_fast, MoucerTaylorBach2023_systematic, SuhParkRyu2023_continuoustime},
would be an interesting question for future work.

There are several open avenues that are more closely relevant, within the scope of first-order optimization.
Our framework currently does not cover the cases of algorithms with block structures, which do not progress linearly with the iteration numbers, 
such as accelerated algorithms without momentum, including those using long/silver step-size schedules 
\cite{grimmerProvablyFasterGradient2024, altschulerAccelerationStepsizeHedging2025, altschulerAccelerationStepsizeHedging2025a, grimmerAcceleratedObjectiveGap2025, GrimmerShuWang2025_composing, BokAltschuler2025_accelerating, ParkRoySiegelBhattacharya2025_acceleration, zhangAnytimeAccelerationGradient2025, WangMaYangZhou2026_relaxed, ZhangJiang2026_accelerated} 
and the recently discovered optimal fixed-point algorithms with similar fractal step-sizes \cite{yoon2026theorycompositiondualityextremal}. 
We also do not capture the analyses where ergodic average of iterates are taken in the end, with the final convergence guarantee being established via Jensen's inequality.
We foresee that our work could be extended to handle such settings by adjusting the partial sum formulation or admissibility criteria.

%% file: sections/appendix_background.tex
\section{Problem Settings Omitted in Section~\ref{section:background}}

\subsection{Smooth convex-concave minimax optimization via Lipschitz monotone operators}
\label{section:minimax-optimization}

Finally, consider an unconstrained smooth convex-concave minimax optimization problem
\[
\begin{array}{cc}
    \underset{u\in\reals^{d_u}}{\text{minimize}}\,\,\,
    \underset{v\in\reals^{d_v}}{\text{maximize}}
    & \phi(u,v),
\end{array}
\]
where $\phi$ is convex in $u$ when $v$ is fixed, concave in $v$ when $u$ is fixed, differentiable, and has $L$-Lipschitz continuous gradient.
Writing $x=(u,v)\in\reals^{d_u} \times \reals^{d_v}$, the associated saddle-gradient operator is
\begin{align}
\label{eqn:saddle-operator}
    \opA(x)
    =
    \begin{bmatrix}
        \nabla_u\phi(u,v)\\
        -\nabla_v\phi(u,v)
    \end{bmatrix}.
\end{align}
Convexity-concavity implies that $\opA$ is monotone \cite{rockafellarMonotoneOperatorsAssociated1970}, while the Lipschitz continuity of $\nabla\phi$ implies that $\opA$ is $L$-Lipschitz.
Furthermore, the goal of minimax optimization is to find $x_\star = (u_\star, v_\star)$ satisfying
\[
    \phi(u_\star, v) \le \phi(u_\star, v_\star) \le \phi(u, v_\star)
\]
for all $(u,v) \in \reals^{d_u} \times \reals^{d_v}$, and this condition is equivalent to $\opA(x_\star) = 0$ as $\phi$ is convex-concave.
Therefore, we can recast the problem (with slight generalization) into the monotone inclusion
\[
\begin{array}{cc}
    \underset{x\in\reals^{d}}{\text{find}} & 0=\opA x 
\end{array}
\]
where $d = d_u+d_v$ and $\opA\colon \reals^d \to \reals^d$ is monotone and $L$-Lipschitz.
We consider the \textit{explicit}, or \textit{forward} algorithms with update rule
\[
    x_{k+1} = x_k - \sum_{j=0}^k h_{k+1,j} \opA x_j 
\]
for $k=0,1,\dots$ and $j=0,\dots,k$.
The inequalities that we use in this problem class are the pairwise inequalities 
\[
    \inprod{\opA x_i-\opA x_j}{x_i-x_j} \ge 0,
    \qquad
    L^2\sqnorm{x_i-x_j}-\sqnorm{\opA x_i-\opA x_j} \ge 0 
\]
where $x_i, x_j$ are either the algorithm's iterates or the minimax optimum $x_\star$.
Here, let us note that unlike the settings introduced earlier, we do not have an interpolation theorem showing that a finite set of tuples $(x_i, g_i)$ pairwise satisfying $\inprod{g_i - g_j}{x_i - x_j} \ge 0$ and $L^2 \sqnorm{x_i - x_j} - \sqnorm{g_i - g_j} \ge 0$ can be realized by a monotone and $L$-Lipschitz operator $\opA\colon \reals^d \to \reals^d$, i.e., $\opA x_i = g_i$ for all $i$ in that finite index set.
In fact, there is a simple three-point counterexample where this fails \citep[Proposition~3]{RyuTaylorBergelingGiselsson2020_operator}.
Therefore, the formulation below based on the pairwise monotonicity and Lipschitzness inequalities should be interpreted as optimizing over the class of proof using only those inequalities, together with sum-of-squares inequalities, which is not necessarily the tightest worst-case guarantee with a matching lower-bound construction.

For $i\in \{0,\dots,N,\star\}$, let $g_i = \opA x_i$ (in particular, $g_\star = 0$).
Let $\vP = \bmat{x_0 & x_\star & g_0 & \cdots & g_N} \in \reals^{d\times (N+3)}$ and let $\vx_0 = \ve_1, \vx_\star = \ve_2, \vg_i = \ve_{i+3}$ for $i=0,\dots,N$ and $\vg_\star = 0$, so that
$\vP\vg_i = g_i = \opA x_i$ for $i\in \{0,\dots,N,\star\}$.
Define $\vx_k$ as follows for $k=1,\dots,N$ so that $\vP\vx_k = x_k$:
\[
    \vx_{k+1} = \vx_k - \sum_{j=0}^k h_{k+1,j} \vg_j .
\]
Let $\vG=\vP^\transpose\vP$ and define
\[
    \vM_{i,j} := \mathrm{Sym}\left((\vg_i-\vg_j)(\vx_i-\vx_j)^\transpose\right) ,
    \qquad
    \vL_{i,j} := L^2(\vx_i-\vx_j)(\vx_i-\vx_j)^\transpose - (\vg_i-\vg_j)(\vg_i-\vg_j)^\transpose 
\]
for $i,j \in \{0,\dots,N,\star\}$, so that $\mathrm{Tr}(\vG\vM_{i,j})$ and $\mathrm{Tr}(\vG\vL_{i,j})$ are respectively the monotonicity and Lipschitzness inequality terms.
We measure the performance of the algorithm in terms of the squared operator (residual) norm, so we are interested in the guarantee of the form $\sqnorm{\opA x_N}\le\nu\sqnorm{x_0-x_\star}$.
A PEP-style proof for this is given by the identity
\[
\begin{aligned}
    & \mathrm{Tr}(\vG\vg_N\vg_N^\transpose)
    -\nu\mathrm{Tr}\left(\vG(\vx_0-\vx_\star)(\vx_0-\vx_\star)^\transpose\right) \\
    & =
    -\sum_{\substack{i,j \in \{0,\dots,N,\star\}\\ i<j}}\lambda_{i,j}^\cM
    \mathrm{Tr}(\vG\vM_{i,j})
    -\sum_{\substack{i,j \in \{0,\dots,N,\star\}\\ i<j}}\lambda_{i,j}^\cL
    \mathrm{Tr}(\vG\vL_{i,j})
    -\sum_{\ell\in\cJ}\alpha_\ell\mathrm{Tr}(\vG\vS_\ell)
\end{aligned}
\]
with nonnegative coefficients $\lambda_{i,j}^\cM, \lambda_{i,j}^\cL$ and $\alpha_\ell$.

\subsection{Composite minimization with Bregman relative smoothness}
\label{section:setting-composite-minimization-bregman}

Let $h:\reals^d\to\reals\cup\{+\infty\}$ be a proper, closed and convex function of Legendre type \cite[Section~26]{Rockafellar1970_convex}, satisfying: $C=\operatorname{int}\operatorname{dom}h\neq\emptyset$, $h$ is differentiable and strictly convex on $C$ and $\lim_{i\to \infty} \norm{\nabla h(y_i)} = +\infty$ whenever $y_i \in C, y_i \to y \in \partial C$.
For such $h$, the Bregman distance 
\[
    D_h(x,y) = h(x)-h(y)-\inprod{\nabla h(y)}{x-y}.
\]
is well defined on $\operatorname{dom}h\times C$.
We consider the composite minimization problem over this Bregman domain
\[
\begin{array}{cc}
    \underset{x\in C}{\text{minimize}} & F(x) = \bregsm(x)+\bregprox(x) .
\end{array}
\]
We assume that $\bregprox$ is proper, closed and convex with $\operatorname{dom}\bregprox\cap C \neq \emptyset$, which will be the proximable term for which the proximal operation can be computed.
We assume that $\bregsm$ is convex, $C\subset\operatorname{dom}\bregsm$, $\bregsm$ differentiable on $C$, and that $\bregsm$ is relatively $L$-smooth with respect to $h$ on $C$ \citep{BauschkeBolteTeboulle2017_descent,LuFreundNesterov2018_relatively}, i.e., for any $x,y \in C$,
\[
    \bregsm(x) - \bregsm(y) - \inprod{\nabla \bregsm(y)}{x-y} - L D_h(x,y) \le 0 .
\]
The relative smoothness condition is equivalent to the function
\[
    \bregaux:=h-\frac{1}{L}\bregsm
\]
being convex on $C$ \citep{LuFreundNesterov2018_relatively}.
For this problem class, we focus on the PEP formulation for the \textit{Bregman proximal-gradient method (BPGM)} rather than handling the general algorithm classes as in the previously introduced settings.
Starting from $x_0\in C$, we consider BPGM with standard step-size $1/L$, whose update rule is given by
\begin{equation*}
\tag{BPGM}
    x_{k+1}\in\argmin_{y\in C}
    \left\{
        \bregprox(y)+\inprod{\nabla \bregsm(x_k)}{y-x_k}
        +L D_h(y,x_k)
    \right\} .
\end{equation*}
Here, we assume that each subproblem defining the BPGM update admits a selected minimizer $x_{k+1}\in C$; standard sufficient conditions for this are discussed in \citep[Lemma~2]{BauschkeBolteTeboulle2017_descent}. 
Rewriting the first-order optimality condition for the BPGM update, there exists 
$\tnabla \bregprox(x_{k+1})\in\partial \bregprox(x_{k+1})$ satisfying
\begin{equation}
\label{eq:appendix-bpgm-optimality}
    \tnabla \bregprox(x_{k+1})
    =
    - \nabla \bregsm(x_k)
    -L\left(\nabla h(x_{k+1})-\nabla h(x_k)\right).
\end{equation}
The PEP formulation for this setting features basic convexity inequalities.
Let $x_{N+1}:=x_{\compidx}$, where $x_{\compidx}\in\dom F\cap\dom h$ satisfies $D_h(x_{\compidx},x_0)<\infty$ and is the arbitrary reference point appearing in the convergence guarantee.
For $q=\bregsm,h,\bregaux$, define the corresponding value vectors $\vbregsm,\vh,\vbregaux$ by
\[
    \vq=\bmat{q(x_0) & \cdots & q(x_N) & q(x_{\compidx})}^\transpose\in\reals^{N+2},
\]
with entries $\vq_i=q(x_i)$ for $i=0,\dots,N+1$.
For the nonsmooth term, the BPGM optimality conditions provide subgradients at the output points $x_1,\dots,x_N$.
Thus, let
\[
    \vbregprox=\bmat{\bregprox(x_1) & \cdots & \bregprox(x_N) & \bregprox(x_{\compidx})}^\transpose\in\reals^{N+1},
\]
with entries denoted by $\vbregprox_i=\bregprox(x_i)$ for $i=1,\dots,N+1$. 
For $\bregprox$, we only use convexity inequalities whose subgradient is taken at one of the BPGM output points:
\[
    \cC_{\bregprox}(x_i,x_j)
    :=
    \bregprox(x_j)-\bregprox(x_i)+\inprod{\tnabla \bregprox(x_j)}{x_i-x_j}\le 0,
    \qquad i=1,\dots,N+1,\quad j=1,\dots,N.
\]
For $q\in\{\bregsm,h\}$, we use
\[
    \cC_q(x_i,x_j)
    :=
    q(x_j)-q(x_i)+\inprod{\nabla q(x_j)}{x_i-x_j}\le 0,
    \qquad i=0,\dots,N+1,\quad j=0,\dots,N.
\]
For $\bregaux$, whose convexity is imposed on $C$, we use the iterate-index inequalities
\[
    \cC_{\bregaux}(x_i,x_j)
    :=
    \bregaux(x_j)-\bregaux(x_i)+\inprod{\nabla \bregaux(x_j)}{x_i-x_j}\le 0,
    \qquad i,j=0,\dots,N.
\]
Since $\bregaux=h-\frac{1}{L}\bregsm$, the corresponding function value coordinates satisfy $\vbregaux_i=\vh_i-\frac{1}{L}\vbregsm_i$ for $i=0,\dots,N+1$. 
Let
\[
\begin{aligned}
    \vP
    =
    \Bigl[
     x_0,\ldots,x_N, x_{\compidx},\
    \nabla \bregsm(x_0),\ldots,\nabla \bregsm(x_N),
    \nabla h(x_0),\ldots,\nabla h(x_N)
    \Bigr]
    \in \reals^{d\times(3N+4)} .
\end{aligned}
\]
Using standard basis vectors in $\reals^{3N+4}$, define the point coordinates
\[
    \vx_i=\ve_{i+1},
    \qquad i=0,\dots,N+1,
\]
and the gradient coordinates
\[
    \vxi_i^{\bregsm}=\ve_{N+3+i},
    \qquad
    \vxi_i^h=\ve_{2N+4+i},
    \qquad i=0,\dots,N.
\]
For $\bregaux=h-\frac{1}{L}\bregsm$ and for the selected subgradients of $\bregprox$ given by \eqref{eq:appendix-bpgm-optimality}, set
\[
    \vxi_i^{\bregaux}=\vxi_i^h-\frac{1}{L}\vxi_i^{\bregsm},
    \qquad i=0,\dots,N,
\]
and
\[
    \vxi_i^{\bregprox}=-\vxi_{i-1}^{\bregsm}-L\left(\vxi_i^h-\vxi_{i-1}^h\right),
    \qquad i=1,\dots,N.
\]
Then $\vP\vx_i=x_i$ for $i=0,\dots,N+1$, $\vP\vxi_i^{\bregprox}=\tnabla \bregprox(x_i)$ for $i=1,\dots,N$, and $\vP\vxi_i^q=\nabla q(x_i)$ for $q\in\{\bregsm,h,\bregaux\}$ and $i=0,\dots,N$.
With $\vG=\vP^\transpose\vP\in\mathbb{S}_+^{3N+4}$, define
\[
    \vC_{i,j}^q
    :=
    \mathrm{Sym}\left(\vxi_j^q(\vx_i-\vx_j)^\transpose\right)\in\mathbb{S}^{3N+4},
\]
whenever the coordinate vector $\vxi_j^q$ has been defined.
The initial Bregman distance is represented by the scalar coordinate
\[
    \vd
    :=
    \vh_{N+1}-\vh_0-\mathrm{Tr}\left(\vG\vC_{N+1,0}^h\right).
\]
Consequently, a PEP-style proof of the rate
$F(x_N)-F(x_{\compidx})\le\nu D_h(x_{\compidx},x_0)$ has the form
\[
\begin{aligned}
    \vbregsm_N+\vbregprox_N-\vbregsm_{N+1}-\vbregprox_{N+1}
    -\nu\vd
    &=
    \sum_{i=1}^{N+1}\sum_{j=1}^{N}
    \lambda_{i,j}^{\bregprox}
    \left(\vbregprox_j-\vbregprox_i+\mathrm{Tr}\left(\vG\vC_{i,j}^{\bregprox}\right)\right) \\
    &
    +\sum_{q\in\{\bregsm,h\}}\sum_{i=0}^{N+1}\sum_{j=0}^{N}
    \lambda_{i,j}^q
    \left(\vq_j-\vq_i+\mathrm{Tr}\left(\vG\vC_{i,j}^q\right)\right) \\
    &    
    +\sum_{i=0}^{N}\sum_{j=0}^{N}
    \lambda_{i,j}^{\bregaux}
    \left(\vbregaux_j-\vbregaux_i+\mathrm{Tr}\left(\vG\vC_{i,j}^{\bregaux}\right)\right) 
    -\sum_{i\in\cJ}\alpha_i\mathrm{Tr}(\vG\vS_i).
\end{aligned}
\]
The coefficients $\lambda_{i,j}^{\bregprox}$, $\lambda_{i,j}^q$, and $\alpha_i$ are nonnegative, and $\vS_i=\vs_i\vs_i^\transpose\in\mathbb{S}_+^{3N+4}$.
The BPGM example in \cref{section:bpgm} is obtained by selecting the active interpolation inequalities in this template.

%% file: sections/appendix_linear_algebra_prop.tex
\section{Proof of \cref{proposition:linear-algebra-of-lyapunov}}
\label{appendix:lyapunov-proposition-proof}

\begin{proof}
The identity $V_k=\Tr(\vG\vV_k)$ follows directly from
\eqref{eqn:Vk-definition-monotone-proximal} and the definition of $\vV_k$.
Because $\vG=\vP^\transpose \vP$, once we show the decomposition $\vV_k = \vB_k \vA_k \vB_k^\transpose$ with $\vB_k=\begin{bmatrix} \vv_{k,1} & \cdots & \vv_{k,r_k} \end{bmatrix} \in \reals^{(N+2)\times r_k}$, 
it follows that 
\[
    V_k
    =
    \Tr(\vP^\transpose \vP \vB_k \vA_k \vB_k^\transpose)
    =
    \Tr(\vB_k^\transpose \vP^\transpose \vP \vB_k \vA_k)
    =
    \sum_{p,q=1}^{r_k}(\vA_k)_{p,q}
    \left( \vB_k^\transpose \vP^\transpose \vP \vB_k \right)_{p,q}
    =
    \sum_{p,q=1}^{r_k}(\vA_k)_{p,q}
    \inprod{\vP\vv_{k,p}}{\vP\vv_{k,q}} .
\]
which gives \eqref{eqn:Lyapunov-expression-quadratic-form}.

It remains to prove that $\vV_k = \vB_k \vA_k \vB_k^\transpose$ for some $\vA_k \in \mathbb{S}^{r_k}$. If $r_k=0$, then $\vV_k=\mathbf{0}$ and the claim is trivial. Otherwise, because the columns of $\vB_k$ form a basis of $\cR(\vV_k)$, the matrix $\boldsymbol{\Pi}_k := \vB_k(\vB_k^\transpose \vB_k)^{-1}\vB_k^\transpose \in \mathbb{S}^{N+2}$ is the orthogonal projection matrix onto $\cR(\vV_k)$. 
This implies that $\boldsymbol{\Pi}_k \vV_k = \vV_k$, and because $\vV_k$ is symmetric, we also have $\vV_k = \vV_k \boldsymbol{\Pi}_k$. Therefore,
\[
    \vV_k = \boldsymbol{\Pi}_k \vV_k \boldsymbol{\Pi}_k = \vB_k \underbrace{
        (\vB_k^\transpose \vB_k)^{-1}\vB_k^\transpose
        \vV_k
        \vB_k(\vB_k^\transpose \vB_k)^{-1}
    }_{:=\vA_k} \vB_k^\transpose = \vB_k \vA_k \vB_k^\transpose
\]
where $\vA_k \in \reals^{r_k \times r_k}$ is clearly symmetric.

\end{proof}

%% file: sections/appendix_bppm_tightness.tex
\section{Tightness of the BPPM rate}
\label{appendix:bppm-tightness}

\begin{proof}[Proof of \Cref{lemma:bppm-tightness}]
We first construct a one-dimensional instance.
Fix a parameter $0<\delta<\frac12$, to be chosen sufficiently small at the end.
Define the smoothed absolute value
\[
    s_\delta(x)
    :=
    \begin{cases}
        \dfrac{x^2}{2\delta}+\dfrac{\delta}{2}, & |x|\le\delta,\\[0.6em]
        |x|, & |x|>\delta.
    \end{cases}
\]
Define $\widetilde h_\delta:\reals\to\reals$ by
\[
\widetilde h_\delta(x)
:=
\begin{cases}
\dfrac{N}{2}s_\delta(x+1)-\dfrac{N+2}{2}x-\dfrac{N}{2},
    & x<-\dfrac12,\\[0.8em]
s_\delta(x), & |x|\le\dfrac12,\\[0.8em]
\dfrac{N}{2}s_\delta(x-1)+\dfrac{N+2}{2}x-\dfrac{N}{2},
    & x>\dfrac12.
\end{cases}
\]
A direct differentiation gives
\begin{equation}
\widetilde h_\delta'(x)
=
\begin{cases}
-(N+1), & x\le -1-\delta,\\
-\dfrac{N+2}{2}+\dfrac{N}{2\delta}(x+1),
    & -1-\delta\le x\le -1+\delta,\\[0.4em]
-1, & -1+\delta\le x\le -\delta,\\
\dfrac{x}{\delta}, & |x|\le\delta,\\[0.4em]
1, & \delta\le x\le 1-\delta,\\
\dfrac{N+2}{2}+\dfrac{N}{2\delta}(x-1),
    & 1-\delta\le x\le 1+\delta,\\[0.4em]
N+1, & x\ge 1+\delta.
\end{cases}
    \label{eq:bppm-tight-h-tilde-derivative}
\end{equation}
Thus $\widetilde h_\delta$ is continuously differentiable and convex.
Set
\begin{equation}
    h_\delta(x)
    =
    \frac{1}{L}\left(\widetilde h_\delta(x)+\frac{\delta}{2}x^2\right).
    \label{eq:bppm-tight-h-def}
\end{equation}
The function $h_\delta$ is $\delta/L$-strongly convex, its derivative is strictly increasing, and $h_\delta'(x)\to\pm\infty$ as $x\to\pm\infty$.
Hence $h_\delta$ is a Legendre function, and every BPPM subproblem has a unique minimizer.

Now choose the function and initial point as
\[
    \bregprox(x)=|x|,
    \qquad
    x_0=-1-\delta.
\]
Then $x_\star=0$ is a minimizer of $\bregprox$.
We first show by induction that, for $0\le k\le N$,
\begin{equation}
    x_k\in[-1-\delta,-1+\delta],
    \qquad
    Lh_\delta'(x_k)=-(N+1-k)-\delta(1+\delta).
    \label{eq:bppm-tight-induction}
\end{equation}
The claim is true for $k=0$ by the choice $x_0=-1-\delta$ and \eqref{eq:bppm-tight-h-tilde-derivative}.
Suppose it holds for some $k<N$.
The optimality condition of the BPPM update \eqref{eq:bppm_main} gives
\begin{equation*}
    0\in \partial\bregprox(x_{k+1})
    +L\bigl(h_\delta'(x_{k+1})-h_\delta'(x_k)\bigr),
\end{equation*}
or equivalently
\begin{equation}
    Lh_\delta'(x_{k+1})
    \in
    -\partial\bregprox(x_{k+1})+Lh_\delta'(x_k).
    \label{eq:bppm-tight-optimality}
\end{equation}
Since $-\partial\bregprox(x_{k+1})\subset[-1,1]$, the induction hypothesis implies
\[
    Lh_\delta'(x_{k+1})
    \le
    1+Lh_\delta'(x_k)
    =-(N-k)-\delta(1+\delta)
    <0.
\]
As $Lh_\delta'$ is strictly increasing and $Lh_\delta'(0)=0$, this shows that $x_{k+1}<0$.
Therefore $\partial\bregprox(x_{k+1})=\{-1\}$, and \eqref{eq:bppm-tight-optimality} reduces to
\[
    Lh_\delta'(x_{k+1})
    =
    1+Lh_\delta'(x_k)
    =
    -(N-k)-\delta(1+\delta).
\]
Moreover,
\[
    Lh_\delta'(-1-\delta)=-(N+1)-\delta(1+\delta),
    \qquad
    Lh_\delta'(-1+\delta)=-1-\delta(1-\delta).
\]
Since $0\le k\le N-1$,
\[
    Lh_\delta'(-1-\delta)
    <
    -(N-k)-\delta(1+\delta)
    \le
    Lh_\delta'(-1+\delta).
\]
As $Lh_\delta'$ is increasing, this implies $x_{k+1}\in[-1-\delta,-1+\delta]$.
Together with the identity for $Lh_\delta'(x_{k+1})$ above, this proves \eqref{eq:bppm-tight-induction} with $k$ replaced by $k+1$ and completes the induction.

On the interval $[-1-\delta,-1+\delta]$, \eqref{eq:bppm-tight-h-def} and \eqref{eq:bppm-tight-h-tilde-derivative} give
\begin{equation}
    Lh_\delta'(x)
    =
    -\frac{N+2}{2}+\frac{N}{2\delta}(x+1)+\delta x.
    \label{eq:bppm-tight-h-affine-branch}
\end{equation}
Substituting $x=x_k$ in \eqref{eq:bppm-tight-h-affine-branch} and comparing with \eqref{eq:bppm-tight-induction} gives $\left(\frac{N}{2\delta}+\delta\right)(x_k+1+\delta)=k$, and hence
\begin{equation}
    x_k=-1-\delta+\frac{2\delta}{N+2\delta^2}k,
    \qquad 0\le k\le N.
    \label{eq:bppm-tight-trajectory}
\end{equation}

Because $x_N<0$,
\begin{equation}
    \bregprox(x_N)-\bregprox(x_\star)
    =
    1+\delta-\frac{2\delta N}{N+2\delta^2}.
    \label{eq:bppm-tight-gap}
\end{equation}
For the initial Bregman distance, note that $\widetilde h_\delta(0)=\frac{\delta}{2}$, $\widetilde h_\delta(x_0)=1+(N+1)\delta$, and $\widetilde h_\delta'(x_0)=-(N+1)$, hence $D_{\widetilde h_\delta}(0,x_0)=N+\frac{\delta}{2}$.
The quadratic term in \eqref{eq:bppm-tight-h-def} yields
\begin{equation}
\begin{aligned}
    D_{h_\delta}(x_\star,x_0)
    &=
    \frac{1}{L}\left(
        N+\frac{\delta}{2}
        +\frac{\delta}{2}(1+\delta)^2
    \right)
    =
    \frac{1}{L}\left(
        N+\delta+\delta^2+\frac{\delta^3}{2}
    \right).
\end{aligned}
\label{eq:bppm-tight-divergence}
\end{equation}
Combining \eqref{eq:bppm-tight-gap} and \eqref{eq:bppm-tight-divergence},
\[
\frac{N\bigl(\bregprox(x_N)-\bregprox(x_\star)\bigr)}
     {L D_{h_\delta}(x_\star,x_0)}
=
\frac{N\left(1+\delta-\dfrac{2\delta N}{N+2\delta^2}\right)}
     {N+\delta+\delta^2+\dfrac{\delta^3}{2}}
\longrightarrow 1
\qquad\text{as }\delta\downarrow0.
\]
Thus, for the prescribed $\varepsilon$, one can choose $0<\delta<\frac12$ sufficiently small so that
\[
    \bregprox(x_N)-\bregprox(x_\star)
    \ge
    (1-\varepsilon)\frac{L}{N}D_{h_\delta}(x_\star,x_0).
\]
This proves the claim for dimension $d=1$.

For an arbitrary dimension $d>1$, define, for $z=(z_1,\ldots,z_d)$,
\[
    \widehat{\bregprox}(z)=\bregprox(z_1)+\frac12\sum_{j=2}^d z_j^2,
    \qquad
    \widehat h(z)=h_\delta(z_1)+\frac{1}{2L}\sum_{j=2}^d z_j^2,
\]
and take $z_0=(x_0,0,\ldots,0)$.
Then $\widehat{\bregprox}$ is proper, closed, and convex, $\widehat h$ is Legendre, the unique minimizer is $z_\star=0$, and the BPPM iterates satisfy $z_k=(x_k,0,\ldots,0)$.
Moreover, $D_{\widehat h}(z_\star,z_0)=D_{h_\delta}(x_\star,x_0)$, so the same lower bound holds in dimension $d$.
\end{proof}

%% file: ref.bib
@article{AttouchChbaniPeypouquetRedont2018_fast,
  title = {Fast Convergence of Inertial Dynamics and Algorithms with Asymptotic Vanishing Viscosity},
  author = {Attouch, Hedy and Chbani, Zaki and Peypouquet, Juan and Redont, Patrick},
  year = {2018},
  journal = {Mathematical Programming},
  volume = {168},
  number = {1},
  pages = {123--175},
  urldate = {2022-01-05},
  langid = {english}
}

@article{DasGuptaVanParysRyu2023_branchandbound,
  title = {Branch-and-Bound Performance Estimation Programming: A Unified Methodology for Constructing Optimal Optimization Methods},
  author = {Das Gupta, Shuvomoy and Van Parys, Bart P. G. and Ryu, Ernest K.},
  year = {2023},
  journal = {Mathematical Programming}
}

@article{Diakonikolas2020_halpern,
  title = {Halpern Iteration for Near-Optimal and Parameter-Free Monotone Inclusion and Strong Solutions to Variational Inequalities},
  author = {Diakonikolas, Jelena},
  year = {2020},
  journal = {Conference on Learning Theory}
}

@article{DiakonikolasWang2022_potential,
  title = {Potential Function-Based Framework for Making the Gradients Small in Convex and Min-Max Optimization},
  author = {Diakonikolas, Jelena and Wang, Puqian},
  year = {2022},
  journal = {SIAM Journal on Optimization}
}

@article{Drori2017_exact,
  title = {The Exact Information-Based Complexity of Smooth Convex Minimization},
  author = {Drori, Yoel},
  year = {2017},
  journal = {Journal of Complexity},
  volume = {39},
  pages = {1--16}
}

@article{DroriTeboulle2014_performance,
  title = {Performance of First-Order Methods for Smooth Convex Minimization: {{A}} Novel Approach},
  author = {Drori, Yoel and Teboulle, Marc},
  year = {2014},
  journal = {Mathematical Programming},
  volume = {145},
  number = {1},
  pages = {451--482}
}

@article{DroriTaylor2020_efficient,
  title = {Efficient First-Order Methods for Convex Minimization: A Constructive Approach},
  author = {Drori, Yoel and Taylor, Adrien B.\},
  year = {2020},
  journal = {Mathematical Programming},
  volume = {184},
  number = {1--2},
  pages = {183--220},
  publisher = {Springer}
}

@article{GorbunovLoizouGidel2022_extragradient,
  title = {Extragradient Method: {$O(1/K)$} Last-Iterate Convergence for Monotone Variational Inequalities and Connections with Cocoercivity},
  author = {Gorbunov, Eduard and Loizou, Nicolas and Gidel, Gauthier},
  year = {2022},
  journal = {International Conference on Artificial Intelligence and Statistics}
}

@article{JangGuptaRyu2025_computerassisted,
  title = {Computer-Assisted Design of Accelerated Composite Optimization Methods: {{OptISTA}}},
  author = {Jang, Uijeong and Gupta, Shuvomoy Das and Ryu, Ernest K.},
  year = {2025},
  journal = {Mathematical Programming}
}

@article{KimFessler2016_optimized,
  title = {Optimized First-Order Methods for Smooth Convex Minimization},
  author = {Kim, Donghwan and Fessler, Jeffrey A.},
  year = {2016},
  journal = {Mathematical Programming},
  volume = {159},
  number = {1-2},
  pages = {81--107}
}

@article{Kim2021_accelerated,
  title = {Accelerated Proximal Point Method for Maximally Monotone Operators},
  author = {Kim, Donghwan},
  year = {2021},
  journal = {Mathematical Programming},
  volume = {190},
  number = {1--2},
  pages = {57--87}
}

@article{KimFessler2021_optimizing,
  title = {Optimizing the Efficiency of First-Order Methods for Decreasing the Gradient of Smooth Convex Functions},
  author = {Kim, Donghwan and Fessler, Jeffrey A.},
  year = {2021},
  journal = {Journal of Optimization Theory and Applications},
  volume = {188},
  number = {1},
  pages = {192--219}
}

@article{KimOzdaglarParkRyu2023_timereversed,
  title = {Time-Reversed Dissipation Induces Duality between Minimizing Gradient Norm and Function Value},
  author = {Kim, Jaeyeon and Ozdaglar, Asuman E. and Park, Chanwoo and Ryu, Ernest K.},
  year = {2023},
  journal = {Neural Information Processing Systems}
}

@article{KimYang2023_convergence,
  title = {Convergence Analysis of {{ODE}} Models for Accelerated First-Order Methods via Positive Semidefinite Kernels},
  author = {Kim, Jungbin and Yang, Insoon},
  year = {2023},
  journal = {NeurIPS}
}

@article{KimYang2023_unifying,
  title = {Unifying {{Nesterov}}'s Accelerated Gradient Methods for Convex and Strongly Convex Objective Functions},
  author = {Kim, Jungbin and Yang, Insoon},
  year = {2023},
  journal = {International Conference on Machine Learning},
}

@article{LeeKim2021_fast,
  title = {Fast Extra Gradient Methods for Smooth Structured Nonconvex-Nonconcave Minimax Problems},
  author = {Lee, Sucheol and Kim, Donghwan},
  year = {2021},
  journal = {Neural Information Processing Systems}
}

@article{LeeParkRyu2021_geometric,
  title = {A Geometric Structure of Acceleration and Its Role in Making Gradients Small Fast},
  author = {Lee, Jongmin and Park, Chanwoo and Ryu, Ernest K.},
  year = {2021},
  journal = {Neural Information Processing Systems},
  eprint = {2106.10439},
  primaryclass = {math.OC},
  archiveprefix = {arxiv}
}

@article{Lieder2021_convergence,
  title = {On the Convergence Rate of the {{Halpern-iteration}}},
  author = {Lieder, Felix},
  year = {2021},
  journal = {Optimization Letters},
  volume = {15},
  number = {2},
  pages = {405--418}
}

@article{Nesterov1983_method,
  title = {A Method of Solving a Convex Programming Problem with Convergence Rate {$O(1/k^2)$}},
  author = {Nesterov, Yurii},
  year = {1983},
  journal = {Doklady Akademii Nauk SSSR},
  volume = {269},
  number = {3},
  pages = {543--547}
}

@book{Nesterov2004_introductory,
  title = {Introductory {{Lectures}} on {{Convex Optimization}}: {{A Basic Course}}},
  author = {Nesterov, Y.},
  year = {2004},
  publisher = {{Springer}},
  xxxpublisher = {Kluwer Academic Publ.}
}

@article{ParkRyu2022_exact,
  title = {Exact Optimal Accelerated Complexity for Fixed-Point Iterations},
  author = {Park, Jisun and Ryu, Ernest K.},
  year = {2022},
  journal = {International Conference on Machine Learning}
}

@article{RyuTaylorBergelingGiselsson2020_operator,
  title = {Operator Splitting Performance Estimation: {{Tight}} Contraction Factors and Optimal Parameter Selection},
  author = {Ryu, Ernest K. and Taylor, Adrien B. and Bergeling, Carolina and Giselsson, Pontus},
  year = {2020},
  journal = {SIAM Journal on Optimization},
  volume = {30},
  number = {3},
  pages = {2251--2271},
  publisher = {{Society for Industrial and Applied Mathematics}},
  urldate = {2021-09-26}
}

@article{SuBoydCandes2014_differential,
  title = {A Differential Equation for Modeling {{Nesterov}}'s Accelerated Gradient Method: {{Theory}} and Insights},
  author = {Su, Weijie and Boyd, Stephen and Cand{\`e}s, Emmanuel J.},
  year = {2014},
  journal = {Neural Information Processing Systems}
}

@article{SuBoydCandes2016_differential,
  title = {A Differential Equation for Modeling {{Nesterov}}'s Accelerated Gradient Method: {{Theory}} and Insights},
  author = {Su, Weijie and Boyd, Stephen and Cand{\`e}s, Emmanuel J.},
  year = {2016},
  journal = {Journal of Machine Learning Research},
  volume = {17},
  number = {153},
  pages = {1--43}
}

@article{SuhParkRyu2023_continuoustime,
  title = {Continuous-Time Analysis of Anchor Acceleration},
  author = {Suh, Jaewook J. and Park, Jisun and Ryu, Ernest K.},
  year = {2023},
  journal = {Neural Information Processing Systems}
}

@article{SuhRohRyu2022_continuoustime,
  title = {Continuous-Time Analysis of {{AGM}} via Conservation Laws in Dilated Coordinate Systems},
  author = {Suh, Jaewook J. and Roh, Gyumin and Ryu, Ernest K.},
  year = {2022},
  journal = {International Conference on Machine Learning},
  copyright = {Creative Commons Attribution Non Commercial No Derivatives 4.0 International}
}

@article{TaylorBach2019_stochastic,
  title = {Stochastic First-Order Methods: Non-Asymptotic and Computer-Aided Analyses via Potential Functions},
  author = {Taylor, Adrien and Bach, Francis},
  year = {2019},
  journal = {Conference on Learning Theory}
}

@article{TaylorDrori2023_optimal,
  title = {An Optimal Gradient Method for Smooth Strongly Convex Minimization},
  author = {Taylor, Adrien and Drori, Yoel},
  year = {2023},
  journal = {Mathematical Programming},
  volume = {199},
  number = {1},
  pages = {557--594}
}

@article{TaylorHendrickxGlineur2017_smooth,
  title = {Smooth Strongly Convex Interpolation and Exact Worst-Case Performance of First-Order Methods},
  author = {Taylor, Adrien B. and Hendrickx, Julien M. and Glineur, Fran{\c c}ois},
  year = {2017},
  journal = {Mathematical Programming},
  volume = {161},
  number = {1},
  pages = {307--345}
}

@article{Tran-DinhLuo2021_halperntype,
  title = {Halpern-Type Accelerated and Splitting Algorithms for Monotone Inclusions},
  author = {{Tran-Dinh}, Quoc and Luo, Yang},
  year = {2021},
  journal = {arXiv:2110.08150},
  eprint = {2110.08150},
  archiveprefix = {arxiv}
}

@article{VanScoyFreemanLynch2018_fastest,
  title = {The Fastest Known Globally Convergent First-Order Method for Minimizing Strongly Convex Functions},
  author = {Van Scoy, Bryan and Freeman, Randy A. and Lynch, Kevin M.},
  year = {2018},
  journal = {IEEE Control Systems Letters},
  volume = {2},
  number = {1},
  pages = {49--54}
}

@article{WilsonRechtJordan2021_lyapunov,
  title = {A {{Lyapunov}} Analysis of Accelerated Methods in Optimization},
  author = {Wilson, Ashia C. and Recht, Ben and Jordan, Michael I.},
  year = {2021},
  journal = {Journal of Machine Learning Research},
  volume = {22},
  number = {113},
  pages = {1--34},
  urldate = {2022-01-05}
}

@article{YoonRyu2021_accelerated,
  title = {Accelerated Algorithms for Smooth Convex-Concave Minimax Problems with $\mathcal{O}(1/k^2)$ Rate on Squared Gradient Norm},
  author = {Yoon, TaeHo and Ryu, Ernest K.},
  year = {2021},
  journal = {International Conference on Machine Learning},
  series = {Proceedings of Machine Learning Research}
}

@article{KricheneBayenBartlett2015_accelerated,
  title = {Accelerated Mirror Descent in Continuous and Discrete Time},
  author = {Krichene, Walid and Bayen, Alexandre and Bartlett, Peter L},
  year = {2015},
  journal = {Neural Information Processing Systems}
}

@article{BotCsetnekNguyen2023_fast,
  title = {Fast Optimistic Gradient Descent Ascent ({{OGDA}}) Method in Continuous and Discrete Time},
  author = {Bo{\c t}, Radu Ioan and Csetnek, Ern{\"o} Robert and Nguyen, Dang-Khoa},
  year = {2023},
  journal = {Foundations of Computational Mathematics}
}

@article{YoonKimSuhRyu2024_optimal,
  title={Optimal acceleration for minimax and fixed-point problems is not unique},
  author={Yoon, TaeHo and Kim, Jaeyeon and Suh, Jaewook J and Ryu, Ernest K.},
  journal={International Conference on Machine Learning},
  year={2024}
}

@article{yoonAcceleratedMinimaxAlgorithms2025,
  title = {Accelerated Minimax Algorithms Flock Together},
  author = {Yoon, TaeHo and Ryu, Ernest K.},
  year = 2025,
  journal = {SIAM Journal on Optimization},
  volume = {35},
  number = {1},
  pages = {180--209},
  doi = {10.1137/22M1504597}
}

@article{YoonRyuGrimmerInvariance2025,
  title={H-invariance theory: A complete characterization of minimax optimal fixed-point algorithms},
  author={Yoon, TaeHo and Ryu, Ernest K.\ and Grimmer, Benjamin},
  journal={arXiv:2511.14915},
  year={2025}
}

@article{dAspremontScieurTaylor2021_acceleration,
  title = {Acceleration Methods},
  author = {{d'Aspremont}, Alexandre and Scieur, Damien and Taylor, Adrien},
  year = {2021},
  journal = {Foundations and Trends{\textregistered} in Optimization},
  volume = {5},
  number = {1-2},
  pages = {1--245},
  publisher = {Now Publishers}
}

@article{ParkParkRyu2023_factorsqrt2,
  title = {Factor-${\sqrt{2}}$ Acceleration of Accelerated Gradient Methods},
  author = {Park, Chanwoo and Park, Jisun and Ryu, Ernest K.},
  year = {2023},
  journal = {Applied Mathematics \& Optimization},
  volume = {88},
  number = {3},
  pages = {77}
}

@article{CZ1992,
author = {Censor, Yair and Zenios, Stavros},
year = {1992},
month = {05},
pages = {451-464},
title = {On the proximal minimization algorithm with D-Functions},
volume = {73},
journal = {Journal of Optimization Theory and Applications},
doi = {10.1007/BF00940051}
}

@article{UpadhyayaBanertTaylorGiselsson2025_automated,
  title = {Automated Tight {{Lyapunov}} Analysis for First-Order Methods},
  author = {Upadhyaya, Manu and Banert, Sebastian and Taylor, Adrien B. and Giselsson, Pontus},
  year = 2025,
  journal = {Mathematical Programming},
  volume = {209},
  number = {1},
  pages = {133--170},
  issn = {1436-4646},
  doi = {10.1007/s10107-024-02061-8}
}

@article{TeboulleVaisbourd2023_elementary,
  title = {An Elementary Approach to Tight Worst Case Complexity Analysis of Gradient Based Methods},
  author = {Teboulle, Marc and Vaisbourd, Yakov},
  year = 2023,
  month = sep,
  journal = {Mathematical Programming},
  volume = {201},
  number = {1},
  pages = {63--96},
  issn = {1436-4646},
  doi = {10.1007/s10107-022-01899-0}
}

@inproceedings{SuhYingJiangNguyen2025_pepflow,
  title = {{{PEPFlow}}: A Python Library for the Workflow of Performance Estimation of Optimization Algorithms},
  booktitle = {{{NeurIPS}} Workshop on {{GPU-accelerated}} and Scalable Optimization},
  author = {Suh, Jaewook J. and Ying, Bicheng and Jiang, Xin and Nguyen, Edward Duc Hien},
  year = 2025
}

@article{lessardAnalysisDesignOptimization2016,
  title = {Analysis and Design of Optimization Algorithms via Integral Quadratic Constraints},
  author = {Lessard, Laurent and Recht, Benjamin and Packard, Andrew},
  year = 2016,
  journal = {SIAM Journal on Optimization},
  volume = {26},
  number = {1},
  pages = {57--95},
  publisher = {SIAM}
}

@article{GuYang2025_tight,
  title = {Tight Convergence Rate in Subgradient Norm of the Proximal Point Algorithm},
  author = {Gu, Guoyong and Yang, Junfeng},
  year = 2025,
  journal = {Optimization},
  pages = {1--15},
  publisher = {Taylor \& Francis},
  issn = {0233-1934},
  doi = {10.1080/02331934.2025.2602877}
}

@article{guTightSublinearConvergence2020,
  title = {Tight Sublinear Convergence Rate of the Proximal Point Algorithm for Maximal Monotone Inclusion Problems},
  author = {Gu, Guoyong and Yang, Junfeng},
  year = 2020,
  journal = {SIAM Journal on Optimization},
  volume = {30},
  number = {3},
  pages = {1905--1921},
  publisher = {{Society for Industrial and Applied Mathematics}},
  urldate = {2021-09-27}
}

@article{altschulerAccelerationStepsizeHedging2025,
  title = {Acceleration by Stepsize Hedging: {{Multi-step}} Descent and the Silver Stepsize Schedule},
  author = {Altschuler, Jason M. and Parrilo, Pablo A.},
  year = 2025,
  journal = {Journal of The ACM},
  volume = {72},
  number = {2},
  publisher = {Association for Computing Machinery},
  address = {New York, NY, USA}
}

@article{altschulerAccelerationStepsizeHedging2025a,
  title = {Acceleration by Stepsize Hedging: {{Silver Stepsize Schedule}} for Smooth Convex Optimization},
  author = {Altschuler, Jason M. and Parrilo, Pablo A.},
  year = 2025,
  journal = {Mathematical Programming},
  volume = {213},
  number = {1},
  pages = {1105--1118}
}

@article{grimmerAcceleratedObjectiveGap2025,
  title = {Accelerated Objective Gap and Gradient Norm Convergence for Gradient Descent via Long Steps},
  author = {Grimmer, Benjamin and Shu, Kevin and Wang, Alex L.},
  year = 2025,
  journal = {INFORMS Journal on Optimization},
  volume = {7},
  number = {2},
  eprint = {https://doi.org/10.1287/ijoo.2024.0057},
  pages = {156--169}
}

@book{RyuYin2022_largescale,
  title = {Large-Scale Convex Optimization: {{Algorithms}} \& Analyses via Monotone Operators},
  author = {Ryu, Ernest K. and Yin, Wotao},
  year = 2022,
  publisher = {Cambridge University Press},
  address = {Cambridge}
}

@article{BousselmiHendrickxGlineur2024_interpolation,
  title = {Interpolation Conditions for Linear Operators and Applications to Performance Estimation Problems},
  author = {Bousselmi, Nizar and Hendrickx, Julien M. and Glineur, Fran{\c c}ois},
  year = {2024},
  month = sep,
  journal = {SIAM Journal on Optimization},
  volume = {34},
  number = {3},
  pages = {3033--3063},
  publisher = {{Society for Industrial and Applied Mathematics}},
  issn = {1052-6234},
  doi = {10.1137/23M1575391},
  urldate = {2025-04-12},
}

@article{TaylorVanScoyLessard2018_lyapunov,
  title = {Lyapunov Functions for First-Order Methods: {{Tight}} Automated Convergence Guarantees},
  author = {Taylor, Adrien and Van Scoy, Bryan and Lessard, Laurent},
  year = {2018},
  journal = {International Conference on Machine Learning},
  series = {Proceedings of Machine Learning Research}
}

@article{UpadhyayaGuptaTaylorBanertGiselsson2026_autolyap,
  title = {The {{AutoLyap}} Software Suite for Computer-Assisted {{Lyapunov}} Analyses of First-Order Methods},
  author = {Upadhyaya, Manu and Gupta, Shuvomoy Das and Taylor, Adrien B. and Banert, Sebastian and Giselsson, Pontus},
  year = 2026,
  journal = {arXiv:2506.24076},
  eprint = {2506.24076},
  primaryclass = {math.OC},
  archiveprefix = {arXiv}
}

@article{ParkRyu2024_optimal,
  title = {Optimal First-Order Algorithms as a Function of Inequalities},
  author = {Park, Chanwoo and Ryu, Ernest K.},
  year = {2024},
  journal = {Journal of Machine Learning Research},
  volume = {25},
  number = {51},
  pages = {1--66}
}

@article{BansalGupta2019_potentialfunction,
  title = {Potential-Function Proofs for Gradient Methods},
  author = {Bansal, Nikhil and Gupta, Anupam},
  year = 2019,
  journal = {Theory of Computing},
  volume = {15},
  number = {4},
  pages = {1--32},
  publisher = {Theory of Computing},
  doi = {10.4086/toc.2019.v015a004}
}

@misc{yoon2026theorycompositiondualityextremal,
  title = {A Theory of Composition and Duality of Extremal Optimal Fixed-Point Algorithms},
  author = {Yoon, TaeHo and Grimmer, Benjamin},
  year = 2026,
  howpublished = {arXiv:2605.02231},
  eprint = {2605.02231},
  archiveprefix = {arXiv}
}

@article{dragomirOptimalComplexityCertification2022,
  title = {Optimal Complexity and Certification of {{Bregman}} First-Order Methods},
  author = {Dragomir, Radu-Alexandru and Taylor, Adrien B. and {d'Aspremont}, Alexandre and Bolte, J{\'e}r{\^o}me},
  year = 2022,
  journal = {Mathematical Programming},
  volume = {194},
  number = {1},
  pages = {41--83}
}

@article{TaylorHendrickxGlineur2018_exact,
  title = {Exact Worst-Case Convergence Rates of the Proximal Gradient Method for Composite Convex Minimization},
  author = {Taylor, Adrien B and Hendrickx, Julien M and Glineur, Fran{\c c}ois},
  year = 2018,
  journal = {Journal of Optimization Theory and Applications},
  volume = {178},
  number = {2},
  pages = {455--476},
  publisher = {Springer}
}

@article{ContrerasCominetti2023_optimal,
  title = {Optimal Error Bounds for Non-Expansive Fixed-Point Iterations in Normed Spaces},
  author = {Contreras, Juan Pablo and Cominetti, Roberto},
  year = 2023,
  journal = {Mathematical Programming},
  volume = {199},
  number = {1},
  pages = {343--374},
  issn = {1436-4646},
  doi = {10.1007/s10107-022-01830-7},
}

@article{grimmerProvablyFasterGradient2024,
  title = {Provably Faster Gradient Descent via Long Steps},
  author = {Grimmer, Benjamin},
  year = 2024,
  journal = {SIAM Journal on Optimization},
  volume = {34},
  number = {3},
  pages = {2588--2608}
}

@article{gorbunovLastiterateConvergenceOptimistic2022,
  title = {Last-Iterate Convergence of Optimistic Gradient Method for Monotone Variational Inequalities},
  author = {Gorbunov, Eduard and Taylor, Adrien and Gidel, Gauthier},
  editor = {Oh, Alice H. and Agarwal, Alekh and Belgrave, Danielle and Cho, Kyunghyun},
  year = 2022,
  journal = {Neural Information Processing Systems}
}

@article{zhangAnytimeAccelerationGradient2025,
  title = {Anytime Acceleration of Gradient Descent},
  author = {Zhang, Zihan and Lee, Jason and Du, Simon and Chen, Yuxin},
  editor = {Haghtalab, Nika and Moitra, Ankur},
  year = 2025,
  journal = {Conference on Learning Theory},
  series = {Proceedings of Machine Learning Research},
  publisher = {PMLR}
}

@article{GrimmerShuWang2025_composing,
  title = {Composing Optimized Stepsize Schedules for Gradient Descent},
  author = {Grimmer, Benjamin and Shu, Kevin and Wang, Alex L.},
  year = 2025,
  journal = {Mathematics of Operations Research},
  publisher = {INFORMS},
}

@article{BokAltschuler2025_accelerating,
  title = {Accelerating Proximal Gradient Descent via Silver Stepsizes},
  author = {Bok, Jinho and Altschuler, Jason M.},
  year = {2025},
  journal = {Conference on Learning Theory}
}

@article{ParkRoySiegelBhattacharya2025_acceleration,
  title = {Acceleration via Silver Step-Size on {{Riemannian}} Manifolds with Applications to {{Wasserstein}} Space},
  author = {Park, Jiyoung and Roy, Abhishek and Siegel, Jonathan W. and Bhattacharya, Anirban},
  year = {2025},
  journal = {Neural Information Processing Systems},
  eprint = {2506.06160},
  primaryclass = {math.OC},
  archiveprefix = {arXiv}
}

@article{WangMaYangZhou2026_relaxed,
  title = {Relaxed Proximal Point Algorithm: {{Tight}} Complexity Bounds and Acceleration Without Momentum},
  author = {Wang, Bofan and Ma, Shiqian and Yang, Junfeng and Zhou, Danqing},
  year = {2026},
  journal = {INFORMS Journal on Optimization},
  volume = {8},
  number = {2},
  pages = {141--162},
  publisher = {INFORMS},
  issn = {2575-1484},
  doi = {10.1287/ijoo.2025.0075},
  urldate = {2026-06-22}
}

@article{ZhangJiang2026_accelerated,
  title = {Accelerated Gradient Descent by Concatenation of Stepsize Schedules},
  author = {Zhang, Zehao and Jiang, Rujun},
  year = {2026},
  journal = {arXiv:2410.12395},
  eprint = {2410.12395},
  archiveprefix = {arXiv}
}

@article{goujaudFundamentalProofStructures2023,
  title = {On Fundamental Proof Structures in First-Order Optimization},
  author = {Goujaud, Baptiste and Dieuleveut, Aymeric and Taylor, Adrien},
  year = 2023,
  journal = {Conference on Decision and Control},
}

@incollection{rockafellarMonotoneOperatorsAssociated1970,
  title = {Monotone Operators Associated with Saddle-Functions and Minimax Problems},
  booktitle = {Nonlinear {{Functional Analysis}}, {{Part}} 1},
  author = {Rockafellar, R. T.},
  editor = {Browder, F. E.},
  year = 1970,
  series = {Proceedings of {{Symposia}} in {{Pure Mathematics}}},
  volume = {18},
  pages = {241--250},
  publisher = {American Mathematical Society},
  aseries = {Proc. Symp. in Pure Math.},
  xxxpublisher = {American Mathematical Society}
}

@article{BauschkeBolteTeboulle2017_descent,
  title = {A Descent Lemma Beyond {{Lipschitz}} Gradient Continuity: {{First}}-Order Methods Revisited and Applications},
  author = {Bauschke, Heinz H. and Bolte, J{\'e}r{\^o}me and Teboulle, Marc},
  year = 2017,
  journal = {Mathematics of Operations Research},
  volume = {42},
  number = {2},
  pages = {330--348},
  publisher = {INFORMS},
  issn = {0364-765X},
  doi = {10.1287/moor.2016.0817},
  urldate = {2026-06-22}
}

@article{LuFreundNesterov2018_relatively,
  title = {Relatively Smooth Convex Optimization by First-Order Methods, and Applications},
  author = {Lu, Haihao and Freund, Robert M. and Nesterov, Yurii},
  year = 2018,
  journal = {SIAM Journal on Optimization},
  volume = {28},
  number = {1},
  pages = {333--354},
  publisher = {{Society for Industrial and Applied Mathematics}},
  issn = {1052-6234},
  doi = {10.1137/16M1099546},
  urldate = {2026-06-22}
}

@article{pepit2024,
  title = {{{PEPit}}: Computer-Assisted Worst-Case Analyses of First-Order Optimization Methods in {{Python}}},
  author = {Goujaud, Baptiste and Moucer, C{\'e}line and Glineur, Francois and Hendrickx, Julien and Taylor, Adrien and Dieuleveut, Aymeric},
  year = {2024},
  journal = {Mathematical Programming Computation},
  volume = {16},
  pages = {337--367},
  publisher = {Springer},
}

@inproceedings{TaylorHendrickxGlineur2017_performance,
  title = {Performance Estimation Toolbox ({{PESTO}}): {{Automated}} Worst-Case Analysis of First-Order Optimization Methods},
  booktitle = {Conference on {{Decision}} and {{Control}}},
  author = {Taylor, Adrien B. and Hendrickx, Julien M. and Glineur, Francois},
  year = {2017}
}

@article{DiakonikolasJordan2021_generalized,
  title = {Generalized Momentum-Based Methods: {{A Hamiltonian}} Perspective},
  shorttitle = {Generalized {{Momentum-Based Methods}}},
  author = {Diakonikolas, Jelena and Jordan, Michael I.},
  year = {2021},
  journal = {SIAM Journal on Optimization},
  volume = {31},
  number = {1},
  pages = {915--944},
  publisher = {{Society for Industrial and Applied Mathematics}},
  issn = {1052-6234},
  doi = {10.1137/20M1322716}
}

@article{WibisonoWilsonJordan2016_variational,
  title = {A Variational Perspective on Accelerated Methods in Optimization},
  author = {Wibisono, Andre and Wilson, Ashia C. and Jordan, Michael I.},
  year = {2016},
  journal = {Proceedings of the National Academy of Sciences},
  volume = {113},
  number = {47},
  pages = {E7351--E7358},
  publisher = {National Academy of Sciences},
  issn = {0027-8424, 1091-6490},
  doi = {10.1073/pnas.1614734113},
  urldate = {2022-01-05},
  chapter = {PNAS Plus},
  copyright = {\copyright{} . Freely available online through the PNAS open access option.},
  langid = {english},
  pmid = {27834219}
}

@article{MoucerTaylorBach2023_systematic,
  title = {A Systematic Approach to {{Lyapunov}} Analyses of Continuous-Time Models in Convex Optimization},
  author = {Moucer, C{\'e}line and Taylor, Adrien and Bach, Francis},
  year = 2023,
  journal = {SIAM Journal on Optimization},
  volume = {33},
  number = {3},
  eprint = {https://doi.org/10.1137/22M1498486},
  pages = {1558--1586},
  doi = {10.1137/22M1498486}
}

@article{ZamaniAbbaszadehpeivastideKlerk2024_exact,
  title = {The Exact Worst-Case Convergence Rate of the Alternating Direction Method of Multipliers},
  author = {Zamani, Moslem and Abbaszadehpeivasti, Hadi and {de Klerk}, Etienne},
  year = 2024,
  journal = {Mathematical Programming},
  volume = {208},
  number = {1},
  pages = {243--276},
  issn = {1436-4646},
  doi = {10.1007/s10107-023-02037-0}
}

@article{NguyenSuhJiangMa2025_exact,
  title = {Exact Worst-Case Convergence Rates for {{Douglas}}--{{Rachford}} and {{Davis}}--{{Yin}} Splitting Methods},
  author = {Nguyen, Edward Duc Hien and Suh, Jaewook J. and Jiang, Xin and Ma, Shiqian},
  year = 2025,
  journal = {arXiv:2506.23475},
  eprint = {2506.23475},
  archiveprefix = {arXiv}
}

@book{Rockafellar1970_convex,
  title = {Convex {{Analysis}}},
  author = {Rockafellar, R. Tyrrell},
  year = 1970,
  eprint = {j.ctt14bs1ff},
  eprinttype = {jstor},
  publisher = {Princeton University Press},
  urldate = {2021-09-27}
}

@article{ZhouLiangShen2019_simple,
  title = {A Simple Convergence Analysis of {{Bregman}} Proximal Gradient Algorithm},
  author = {Zhou, Yi and Liang, Yingbin and Shen, Lixin},
  year = 2019,
  journal = {Computational Optimization and Applications},
  volume = {73},
  number = {3},
  pages = {903--912},
  publisher = {Kluwer Academic Publishers Norwell, MA, USA}
}

@article{AuslenderTeboulle2006_interior,
  title = {Interior {{Gradient}} and {{Proximal Methods}} for {{Convex}} and {{Conic Optimization}}},
  author = {Auslender, Alfred and Teboulle, Marc},
  year = 2006,
  journal = {SIAM Journal on Optimization},
  volume = {16},
  number = {3},
  pages = {697--725},
  publisher = {{Society for Industrial and Applied Mathematics}},
  issn = {1052-6234},
}

@article{GutmanPena2023_perturbed,
  title = {Perturbed {{Fenchel}} Duality and First-Order Methods},
  author = {Gutman, David H. and Pe{\~n}a, Javier F.},
  year = 2023,
  journal = {Mathematical Programming},
  volume = {198},
  number = {1},
  pages = {443--469},
  issn = {1436-4646},
  doi = {10.1007/s10107-022-01779-7}
}
